\documentclass[11pt,reqno]{amsart}
\usepackage{etex}
\usepackage{amsmath,amsthm,amsfonts,amssymb,amscd,flafter,pinlabel, booktabs}
\usepackage[mathscr]{eucal}
\usepackage{graphics}
 \usepackage[all,knot,arc]{xy}

\usepackage[usenames,dvipsnames]{color}
\usepackage{subfigure}
\usepackage{graphicx}
 \usepackage{tabu}

\usepackage{bbm}
\usepackage{mathtools}

\usepackage{pgf}
\usepackage{tikz}
\usetikzlibrary{arrows,automata}
\usepackage[latin1]{inputenc}

\usepackage[colorlinks=true,linkcolor=blue,citecolor=blue,urlcolor=blue]{hyperref}

\tikzset{cdlabel/.style={above,sloped,
    execute at begin node=$\scriptstyle,execute at end node=$}}    
\tikzset{al/.style={->, bend right=45, thick}}
\tikzset{ar/.style={->, bend left=45, thick}}
\def\mathcenter#1{%
  \vcenter{\hbox{$#1$}}%
}

\normalfont\upshape

\newtheorem{theorem}{Theorem}[subsection]
\newtheorem{corollary}[theorem]{Corollary}

\newtheorem{proposition}[theorem]{Proposition}
\newtheorem{lemma}[theorem]{Lemma}
\newtheorem{definition}[theorem]{Definition}

\newtheorem*{remark*}{Remark}
\newtheorem*{example*}{Example}

\newcommand{\R}{\ensuremath{\mathbb{R}}}
\newcommand{\Z}{\ensuremath{\mathbb{Z}}}

\newcommand\HFKh{{\rm {\widehat{HFK}}}}

\newcommand\CFKh{{\rm {\widehat{CFK}}}}

\newcommand\alphas{\mbox{\boldmath$\alpha$}}
\newcommand\betas{\mbox{\boldmath$\beta$}}

\newcommand\zs{\mathbf z}

\newcommand\RT{{\rm {RT}}}

\def\triv{\mathrm{triv}}

\newcommand{\bdy}{\partial}

\newcommand{\sss}{{\bf s}}
\newcommand{\ttt}{{\bf t}}

\DeclareMathOperator{\id}{id}

\DeclareMathOperator{\ev}{ev}

\DeclareMathOperator{\Hom}{Hom}
\DeclareMathOperator{\End}{End}

\DeclareMathOperator{\BBar}{Bar}

\newcommand{\K}{\mathcal{K}}

\newcommand{\C}{\mathbb C}

\newcommand{\ctt}{\widetilde{\rm {CT}}}

\def\P{\mathcal{P}}

\def\a{\mathbf{a}}

\def\b{\mathbf{b}}

\def\H{\mathcal{H}}
\def\Cc{\mathcal{C}}

\def\CDA^-{\mathit{CDA}^-}

\newcommand{\sltwo}{\mathfrak{sl}_2}
\newcommand{\gloneone}{\mathfrak{gl}_{1|1}}
\newcommand{\Rv}{\check{R}}
\newcommand{\hfrak}{\mathfrak{h}}

\newcommand{\ve}{\varepsilon}
\newcommand{\qb}{\mathbf{q}}
\newcommand{\Pb}{\mathbf{P}}

\newcommand{\undi}{\underline{\mathrm{i}}}

\newcommand{\undlam}{\underline{\lambda}}
\newcommand{\unda}{\underline{a}}
\def\ev{\mathrm{ev}}
\def\coev{\mathrm{coev}}
\def\lcap{\mathrm{lcap}}
\def\rcap{\mathrm{rcap}}
\def\lcup{\mathrm{lcup}}
\def\rcup{\mathrm{rcup}}

\def\AA{A} 
\newcommand{\II}{I} 
\newcommand{\PP}{P} 
\newcommand{\ainf}{A_\infty}
\newcommand{\AMod}[1]{\prescript{}{#1}{\mathrm{Mod}}} 
\newcommand{\DMod}[1]{\prescript{#1}{}{\mathrm{Mod}}} 
\newcommand{\DuMod}[1]{\prescript{#1}{\mathfrak{u}}{\mathrm{Mod}}}

\newcommand{\ModA}[1]{\mathrm{Mod}_{#1}} 

\newcommand{\AMor}[1]{\prescript{}{#1}{\mathrm{Mor}}}
\newcommand{\DMor}[1]{\prescript{#1}{}{\mathrm{Mor}}}
\newcommand{\DuMor}[1]{\prescript{#1}{\mathfrak{u}}{\mathrm{Mor}}}

\newcommand{\MorA}[1]{\mathrm{Mor}_{#1}}

\newcommand{\hcat}[1]{\mathcal{H}(#1)}
\newcommand{\hcatcf}[1]{\mathcal{H}_{\mathit{cf}}(#1)}
\newcommand{\dcat}[1]{\mathcal{D}(#1)}
\newcommand{\dcatc}[1]{\mathcal{D}_c(#1)}
\def\bfk{{\bf k}} 
\def\bfj{{\bf j}} 
\def\Ft{\mathbb{F}_2} 
\newcommand{\dtensor}[1]{\mathbin{\widetilde{\otimes}}_{#1}} 
\newcommand{\ee}[1]{e_{#1}} 
\newcommand{\sssA}[1]{\sss^A_{#1}}
\newcommand{\sssD}[1]{\sss^D_{#1}}
\newcommand{\eeA}[1]{\ee{#1}^A}
\newcommand{\eeD}[1]{\ee{#1}^D}
\newcommand{\T}{\mathcal{T}} 
\newcommand{\Tdec}{\mathbb{T}} 
\newcommand{\tA}{\mathit{A}} 
\newcommand{\tD}{\mathit{D}}
\newcommand{\tAA}{\mathit{AA}}
\newcommand{\tAD}{\mathit{AD}}
\newcommand{\tDA}{\mathit{DA}}
\newcommand{\tDD}{\mathit{DD}}
\newcommand{\EE}{E}

\newcommand{\FF}{F}

\newcommand{\refequal}[1]{\xy {\ar@{=}^{#1}
(-1,0)*{};(1,0)*{}};
\endxy}

\usepackage{tikz}
\usetikzlibrary{decorations.markings}
\usetikzlibrary{decorations.pathreplacing}
\usepackage[outline]{contour}
\contourlength{1.2pt}

\tikzset{->-/.style={decoration={
  markings,
  mark=at position #1 with {\arrow{>}}},postaction={decorate}}}

\tikzset{middlearrow/.style={
        decoration={markings,
            mark= at position 0.5 with {\arrow{#1}} ,
        },
        postaction={decorate}
    }
}

\begin{document}

\title[{Quantum $\gloneone$ and tangle Floer homology}]{Quantum $\gloneone$ and tangle Floer homology}

\author[Alexander P. Ellis]{Alexander P. Ellis}
\email{apellis@gmail.com}
\urladdr{\href{http://ape.wtf}{http://ape.wtf}}
\author{Ina Petkova}
\address {Department of Mathematics, Dartmouth College\\ Hanover, NH 03755}
\email {ina.petkova@dartmouth.edu}
\urladdr{\href{http://math.dartmouth.edu/~ina}{http://math.dartmouth.edu/~ina}}
\author[Vera V\'ertesi]{Vera V\'ertesi}
\address{Institut de Recherche Math\'ematique Avanc\'ee \\Universit\'e de Strasbourg}
\email{vertesi@math.unistra.fr}



\begin{abstract}
We identify the Grothendieck group of the tangle Floer dg algebra with a tensor product of certain $U_q(\gloneone)$ representations.  Under this identification, up to a scalar factor, the map on the Grothendieck group induced by the tangle Floer dg bimodule associated to a tangle agrees with the Reshetikhin-Turaev homomorphism for that tangle.  We also introduce dg bimodules which act on the Grothendieck group as the generators $E$ and $F$ of $U_q(\gloneone)$.
\end{abstract}

\maketitle


\tableofcontents



\section{Introduction} 
\label{sec:introduction}


\subsection{Alexander and Jones}

The Reshetikhin-Turaev construction \cite{RT} is a machine for turning a representation $W$ of a quantized enveloping algebra $U_q(\mathfrak{g})$ into a tangle invariant.  It takes:
\begin{enumerate}
	\item a sequence of oriented points to a tensor product of copies of $W$ and $W^*$ and 
	\item a tangle $\T$ to a $U_q(\mathfrak{g})$ intertwiner $\RT(\T)$ between the representations associated to its incoming and outgoing boundaries.
\end{enumerate}
The map associated to a tangle is an invariant of the tangle.

Special cases include the Jones polynomial ($\mathfrak{g} = \sltwo$, $W$ the vector representation $U$) and the Alexander polynomial ($\mathfrak{g} = \gloneone$, $W$ the vector representation $V$).  As interesting as these invariants are, more interesting still are their lifts---\emph{categorifications}---to more complicated invariants.

Khovanov homology is the poster child for categorification.  In its formulation for tangles \cite{chkh,brst}, it takes:
\begin{enumerate}
	\item a sequence of $n$ points to a graded ring $H_n$ and
	\item an $(m,n)$-tangle $\T$ to a complex $C_{\mathit{Kh}}(\T)$ of bimodules over $(H_m,H_n)$.
\end{enumerate}
(See also the earlier work in \cite{kh3} and the more geometrically flavored \cite{BN} when $m$ and $n$ are even.)
The homotopy equivalence class of $C_{\mathit{Kh}}(\T)$ is an invariant of the tangle.  Identifying a basis for the (complexified) Grothendieck group of the category of such complexes with a basis for $\Hom_{U_q(\sltwo)}(U^{\otimes m},U^{\otimes n})$, the functor given by tensor product with $C_{\mathit{Kh}}(\T)$ acts by $\RT(\T)$.

Furthermore, Khovanov homology is functorial: with some adjustments, one can associate a homotopy class of homomorphisms of complexes of bimodules over $(H_m,H_n)$ that is an invariant of tangle cobordisms \cite{KhCob,Jacobsson,MSW}.  The total package, then, is an extended 2D TQFT which categorifies the Reshetikhin-Turaev 1D TQFT.

\scalebox{.975}[1.0]{A more recent approach, initiated by Khovanov-Lauda \cite{KL1,KL3} and Rouquier \cite{Rou2KM},} seeks to categorify an even wider swath of quantum algebra: quantized enveloping algebras themselves, tensor products of their integrable highest weight representations, the Reshetikhin-Turaev intertwiners, and more.  Webster has used this approach to construct link homology theories which categorify the Reshetikhin-Turaev invariant for all representations and Kac-Moody types \cite{Webster} (without the maps for cobordisms).

None of these constructions, however, extend to the case of the Alexander polynomial.  There has been some work on the categorification of $U_q(\gloneone)$ \cite{Khgl12,Tian1,Tian2} (see also an approach via $u_{\sqrt{-1}}(\sltwo)$ \cite{v1,eq}), but so far, the only categorification of the Alexander polynomial has a rather different, non-representation theoretic origin.

Knot Floer homology, introduced by Ozsv\'ath-Szab\'o \cite{hfk} and Rasmussen \cite{jrth}, associates a bigraded chain complex $\CFKh(\mathcal H)$ to a Heegaard diagram $\mathcal H$ for a link $L$.  The differential on $\CFKh(\mathcal H)$ counts pseudoholomorphic curves with prescribed boundary conditions in an almost complex manifold defined in terms of $\mathcal H$.  The homology of $\CFKh(\mathcal H)$ is an invariant of $L$ denoted $\HFKh(L)$.

Like its distant cousin Khovanov homology, $\HFKh(L)$ has proven to be a powerful invariant.  Unfortunately, despite a completely combinatorial description of $\CFKh(\mathcal H)$ \cite{mos}, the invariant is still global in nature; local modifications are only partly understood \cite{hfk,oszskein,mskein}.  In order to fit $\HFKh(L)$ into the general pattern of Reshetikhin-Turaev invariants, then, two initial hurdles must be addressed: \emph{locality} and the \emph{relation to $U_q(\gloneone)$}.

\subsection{Jumping hurdles}

In 2014, the second and third authors introduced a local construction of knot Floer homology \cite{pv}.  A dg algebra $\AA(\PP)$ is associated to each oriented $0$-manifold $\PP$, and a dg bimodule $\ctt(\T)$ is associated to each oriented tangle $\T$.  For a closed tangle, i.e. a link, this bimodule agrees with knot Floer homology tensored with a $2$-dimensional vector space. The general structure should feel familiar to any bordered Heegaard Floer homologist: the bimodules in question are type $\tDA$ structures in the sense of \cite{bimod}, composition of these bimodules is via the box tensor product, and while the algebras and bimodules of \cite{pv} admit combinatorial descriptions inspired by Heegaard diagrams, the proof of invariance is  topological and analytic.

The Alexander polynomial, then, admits a categorification with local pieces very much like its construction as a Reshetikhin-Turaev invariant.  The wildly optimistic reader will expect these local pieces to categorify their Reshetikhin-Turaev counterparts.

Fortunately, this optimism is rewarded.

In Subsection \ref{subsec-gloneone}, to a sign sequence $\PP \in \{\pm1\}^n$ we associate the $U_q(\gloneone)$-representation $V_{\PP} \otimes L(\lambda_{\PP})$, where $V_{\PP}$ is a tensor product of copies of $V$ and $V^*$ ($V$ for plus and $V^*$ for minus) and $L(\lambda_{\PP})$ is an appropriately chosen $2$-dimensional representation depending on $\PP$, and a basis $B$ for $V_{\PP} \otimes L(\lambda_{\PP})$ whose vectors are in bijection with subsets of $[n]= \{0,1,\ldots,n\}$.

The dg algebra $\AA(\PP)$ has primitive idempotents in bijection with subsets of the set $[n]$.  Write $\ee{\sss}$ for the primitive idempotent corresponding to $\sss \subseteq [n]$.

\newtheorem*{thmA}{Theorem A}
\newtheorem*{thmB}{Theorem B}

\begin{thmA}
Let $\PP =(P_1,\ldots,P_n) \in \{\pm1\}^n$ be a sign sequence. Then the Grothendieck group of dg modules over the dg algebra $\AA(\PP)$ is a free $\Z[q^{\pm1}]$-module with basis $\{[\AA(\PP)\ee{\sss}] \mid \sss \subseteq [n]\}$.
   Identifying the basis vector $[\AA(\PP)\ee{\sss}]$ with the basis vector in $B$ associated to the subset $\sss$ determines an isomorphism of vector spaces
\begin{equation*}
K_0(\AA(\PP)) \otimes_{\Z[q^{\pm1}]} \C(q) \cong V_{\PP} \otimes L(\lambda_{\PP}).
\end{equation*}
Let $\T$ be a tangle and color each strand of $\T$ by the vector representation $V$.  Under the identification above, up to an overall factor of a positive integer power of \scalebox{.98}[1.0]{$(1-q^{-2})$}, box tensor product with the type $\tDA$ bimodule $\ctt(\T)$ acts on $K_0(\AA(\PP))$ as the Reshetikhin-Turaev intertwiner associated to the colored  \scalebox{.993}[1.0]{tangle $\T$ tensored with $\id_{L(\lambda_{\PP})}$ (see Figure \ref{fig:RTtangle}). The precise statement is given in Equation \eqref{eqn:thmA}. }\footnote{The appearance of the extra representation $L(\lambda_P)$ is possibly related to the fact that tangle Floer homology recovers knot Floer homology tensored with a $2$-dimensional vector space, rather than knot Floer homology alone.}
\end{thmA}

\begin{figure}[h]
\centering
  \labellist
            \pinlabel \textcolor{red}{$\T$} at 128 90
            \pinlabel  \textcolor{red}{$\rotatebox{270}{$\T$}$} at 360 80
            \pinlabel $K_0(\AA(-\bdy^0\T))\xleftarrow{[\ctt]}K_0(\AA(\bdy^1\T))$ at 126 25
            \pinlabel $=$ at 282 80
            \pinlabel $V_{\bdy^1\T}$ at 432 30
            \pinlabel $\Big\uparrow$ at 432 80
            \pinlabel $V_{-\bdy^0\T}$ at 432 130
            \pinlabel $\RT$ at 418 80
            \pinlabel $\otimes$ at 484 80
            \pinlabel $L(\lambda_{\bdy^1\T})$ at 544 30
            \pinlabel $\Big\uparrow$ at 544 80
            \pinlabel $L(\lambda_{\bdy^1\T})$ at 544 130
            \pinlabel $\id$ at 556 80
       \endlabellist
\includegraphics[scale = .67]{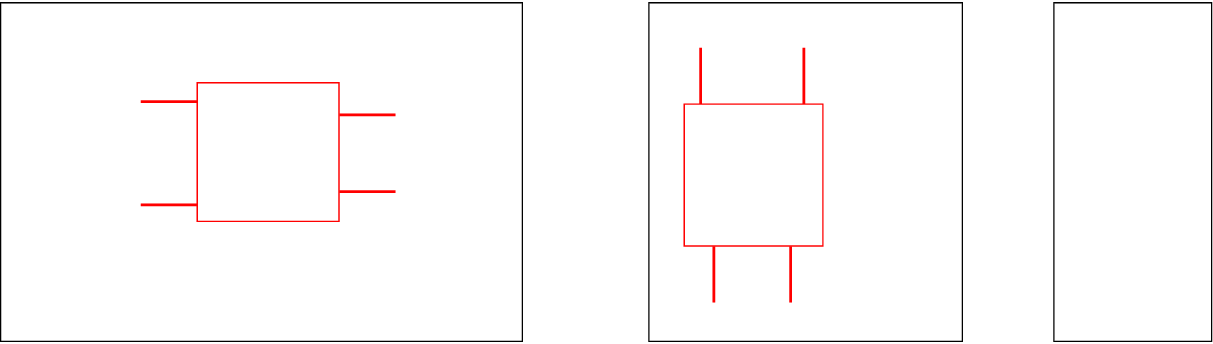} 
      \vskip .2 cm
       \caption{The Reshetikhin-Turaev invariant $\RT(\T)$ for the $U_q(\gloneone)$-representation $V$, later denoted by $Q(\T)$,  corresponds to the action of $\ctt(\T)$ on $K_0(\AA(\PP))$.}
       \label{fig:RTtangle}
\end{figure}

Theorem A is proved in Subsections \ref{subsec-A-P} and \ref{subsec-K0-CT}.

We also introduce dg bimodules $\EE(\PP)$ and $\FF(\PP)$ over $(\AA(\PP), \AA(\PP))$ which act on $K_0(\AA(\PP))$ as the elements $E$ and $F$ of $U_q(\gloneone)$.  In disanalogy with other categorifications of quantized enveloping algebras, these dg bimodules do not arise from induction and restriction with respect to a tower of algebras comprising the dg algebras $\AA(\PP)$.

\begin{thmB}
For any sign sequence $\PP$, under the identification of the elementary basis with the basis $B$ from the previous theorem, the actions of the dg bimodules $\EE(\PP)$ and $\FF(\PP)$ on $K_0(\AA(\PP))$ equal the actions of $E,F \in U_q(\gloneone)$ on $V_{\PP} \otimes L(\lambda_{\PP})$.

There are quasi-isomorphisms
\begin{equation*}
\EE(\PP) \dtensor{\AA(\PP)} \EE(\PP) \simeq 0, \quad
\FF(\PP) \dtensor{\AA(\PP)} \FF(\PP) \simeq 0.
\end{equation*}
Furthermore, there is a distinguished triangle
\begin{equation*}
\xymatrix{
\EE(\PP) \dtensor{\AA(\PP)} \FF(\PP) \ar[r] & \AA(\PP) \ar[r] & \FF(\PP)\dtensor{\AA(\PP)} \EE(\PP) \ar[r] & \EE(\PP) \dtensor{\AA(\PP)} \FF(\PP)[1],
}
\end{equation*}
in $\dcat{\AA(\PP)}$.
\indent For any tangle $\T$,
\begin{equation*}\begin{split}
& \EE(-\bdy^0\T) \boxtimes \ctt(\T) \simeq \AA(-\bdy^0\T) \boxtimes \ctt(\T) \dtensor{\AA(\bdy^1\T)} \EE(\bdy^1\T), \\
& \FF(-\bdy^0\T) \boxtimes \ctt(\T) \simeq \AA(-\bdy^0\T) \boxtimes \ctt(\T) \dtensor{\AA(\bdy^1\T)} \FF(\bdy^1 \T)
\end{split}\end{equation*}
as type $\tAA$ bimodules over $(\AA(-\bdy^0\T), \AA(\bdy^1\T))$.
\end{thmB}

Theorem B is proved in Subsection \ref{subsec-add-black}.

\subsection{Outline}

The reader already conversant with $\ainf$ algebra in the context of bordered Heegaard Floer homology is encouraged to skip Subsections \ref{subsec-ainf} and \ref{subsec:dg-algebra}, which review dg and $\ainf$ algebras, including derived categories of type $\tA$ and type $\tD$ structures.

In Subsection \ref{subsec-gloneone}, we review the decategorified setting: $U_q(\gloneone)$, the particular representations we will be concerned with, a canonical basis for these representations, and the Reshetikhin-Turaev maps on these representations.  This subsection concludes with a description of how to construct the Alexander polynomial in this language.

Section \ref{sec:tangle} is an exposition of the tangle Floer package in the language of strand diagrams (see also \cite[Section 3]{pv}).  Since we only need a special case of the general construction, we are able to make several simplifications; see the dictionary in Subsection \ref{ssec:dictionary}.  Subsections \ref{ssec:alg} through \ref{ssec:DAmodule} review the construction, and Subsection \ref{ssec:CTthms} recalls the main theorems of \cite{pv}.

Theorems A and B are proved in the aptly named Section \ref{sec:results}.  Subsection \ref{subsec-A-P} computes the Grothendieck group of $\AA(\PP)$.  Subsection \ref{subsec-K0-CT} computes the action of $\ctt(\T)$ on this group and relates the result to the representations $V_{\PP} \otimes L(\lambda_{\PP})$ and the Reshetikhin-Turaev maps of Subsection \ref{subsec-gloneone}.  Subsection \ref{subsec-add-black} gives the categorical $U_q(\gloneone)$-module structure on the derived category of compact dg modules over $\AA(\PP)$.

For the reader uninterested in the details, Section \ref{sec:results} can be understood at a purely formal level after giving Section \ref{sec:tangle} not much more than a good skim.

\subsection*{Acknowledgments}

We thank Mikhail Khovanov, Anthony Licata, Robert Lipshitz, Andy Manion, You Qi, and Joshua Sussan for many helpful comments and conversations. We thank the referees for many useful suggestions. A.P.E. and I.P. each received support from an AMS-Simons travel grant. V.V. was supported by ERC Geodycon, OTKA grant number NK81203 and NSF
grant number 1104690.



\section{Preliminaries} 
\label{sec:preliminaries}

\subsection{$\ainf$ algebras and modules}\label{subsec-ainf}

In this and the following subsection, we review some definitions and constructions from \cite{bfh2, bimod}.  This subsection will review modules and bimodules over dg and $\ainf$ algebras; the next will discuss categories of modules and their Grothendieck groups.  For more details, see \cite[Section 2]{bimod}.  Throughout, we work over a ground ring $\bfk$ that is assumed to be a direct sum of a finite number of copies of the two-element field $\Ft$.  Although the $\ainf$ algebras in our main construction are all dg algebras, the general $\ainf$ context is more natural, so we will work in that generality.  Our review will be brief and mostly serves to establish notation; see \cite{Keller,bimod} for more details.

Throughout, we will need to distinguish between homological (later, ``Maslov'')  and inner (later, ``(twice) Alexander'') gradings.  The former is the grading which interacts with differentials, $\ainf$ structure maps, and so forth.  Structure maps will preserve inner gradings.  For a bigraded chain complex $V$, we write $V[k]$ for the complex obtained from $V$ by decreasing the homological grading by $k$.  It has graded pieces $V[k]_i=V_{k+i}$.  To shift inner rather than homological gradings, we write $V\{k\}$.

Throughout, all algebras and modules will be assumed to be finitely generated.

Write $T^*(V)=\bigoplus_{k=0}^\infty V^{\otimes k}$ for the tensor algebra of $V$ and $\overline{T}^*(V)=\prod_{k=0}^\infty V^{\otimes k}$ for its completion.  The $k$-th graded piece of $T^*(V[1])$ is $V^{\otimes k}[k]$.

\begin{definition} An \emph{$\ainf$ algebra} over $\bfk$ is a (homologically) $\Z$-graded $\bfk$-bimodule $A$ equipped with degree $0$ $\bfk$-linear maps
\begin{equation*}
\mu_i: A^{\otimes i} \to A[2-i]\text{ for }i=1,2,\ldots
\end{equation*}
satisfying a certain compatibility condition.  To state this condition, we first define a degree $1$ $\bfk$-linear map $D^A:T^*(A[1])\to T^*(A[1])$.  On $(A[1])^{\otimes n}$, $D^A$ acts as
\begin{equation}
D^A|_{(A[1])^{\otimes n}} = \sum_{j=1}^n \sum_{\ell=1}^{n-j+1} \id_A^{\otimes(j-1)} \otimes \mu_\ell \otimes \id_A^{\otimes(n-\ell-j+1)}.
\end{equation}
The compatibility condition is that
\begin{equation}
D^A \circ D^A = 0
\end{equation}
(or, equivalently, that the part of $D^A \circ D^A$ with image in $A[1]$ is zero).
\end{definition}

Intuitively, the condition says that the sum over all ways to apply two of the $\mu_i$'s in sequence is zero.  There is also a graphical formulation in terms of trees (see, for instance, \cite{bimod}).

\begin{definition} We say an $\ainf$ algebra $A$ is \emph{(strictly) unital} if there is an element $1\in A$ such that $\mu_2(1,a)=\mu_2(a,1)=a$ for all $a\in A$ and $\mu_k(a_1,\ldots,a_k)=0$ if $k>2$ and $a_j=1$ for some $j$.  A strictly unital $\ainf$ algebra $A$ is \emph{augmented} if it is equipped with a $\bfk$-linear map $\epsilon:A\to\bfk$  such that $\epsilon(1)=1$, $\epsilon(\mu_2(a_1,a_2))=\epsilon(a_1)\epsilon(a_2)$, and $\epsilon\circ\mu_k=0$ for $k\neq2$.  If $A$ is unital and augmented, we write $A_+$ for $\ker(\epsilon)$.\end{definition}

\emph{From now on, we assume all $\ainf$ algebras to be strictly unital and augmented.}  Let $A$ be an $\ainf$ algebra over $\bfk$.  Undecorated tensor products are assumed to be taken over $\bfk$.

\begin{definition} A \emph{right $\ainf$ module over $A$} (also called a \emph{right type $\tA$ structure over $A$}) is a graded $\bfk$-module $M$ equipped with degree $1$ $\bfk$-linear maps 
\begin{equation}
m_{i+1}: M\otimes (A[1])^{\otimes i}\to M\text{ for }i=0,1,2,\ldots,
\end{equation}
satisfying a certain compatibility condition.  Assemble all the $m_i$'s to form a single degree $1$ map $m:M\otimes T^*(A[1])\to M$, as we did for $D^A$ above.  Then the compatibility condition is that 
\begin{equation}\label{eqn-type-A-structure-eqn}
m_1 \circ m + m \circ (\id_M \otimes D^A) = 0.
\end{equation}
  A type $\tA$ structure is \emph{strictly unital} if $m_2(x, 1) = x$ and $m_i(x, a_1,\cdots, a_{i-1})=0$ if $i>2$ and some $a_j \in \bfk$.  If $M$ is strictly unital, then equivalently, we can restrict $m$ to $M \otimes T^*(A_+[1])$ and impose the same condition \eqref{eqn-type-A-structure-eqn} on this restriction.
\end{definition}

In the case that $A$ is a dg algebra (meaning $\mu_i=0$ for $i>2$) with differential $d_A=\mu_1$, the type $\tA$ module condition simplifies to
\begin{equation*}\begin{split}
0 &= \sum_{i+j = n+1}m_i(m_j(x, a_1, \cdots, a_{j-1}), \cdots , a_{n-1}) \\
& \qquad + \sum_{i=1}^{n-1} m_n(x, a_1,\cdots, a_{i-1}, d_A(a_i),\cdots, a_{n-1}) \\
& \qquad + \sum_{i=1}^{n-2} m_{n-1}(x, a_1,\cdots, a_{i-1}, a_ia_{i+1},\cdots, a_{n-1})
\end{split}\end{equation*}
for all $n\geq1$.  If we further assume that $M$ is a dg module (meaning $m_i=0$ for $i>2$) with differential $d_M=m_1$, the condition simplifies to
\begin{equation*}\begin{split}
&n=1: \qquad d_M^2 = 0, \\
&n=2: \qquad m_2(d_M(x), a) + d_M(m_2(x, a)) + m_2(x, d_A(a)) = 0, \\
&n=3: \qquad m_2(m_2(x, a), b) + m_2(x, ab) = 0.
\end{split}\end{equation*}

We assume all type $\tA$ structures to be strictly unital.  Left type $\tA$ structures can be defined analogously.  We will write $\prescript{}{A}M$ (respectively $M_A$) to indicate that $M$ is a left (respectively right) module over $A$.

We say that $M$ is \emph{bounded} if  $m_i=0$ for all sufficiently large $i$.

\begin{definition} A \emph{left type $\tD$ structure}  over $A$ is a graded $\bfk$-module $N$ equipped with a degree $0$ $\bfk$-linear homogeneous map
$$\delta^1:N\to (A\otimes N)[1]$$
satisfying a certain compatibility condition.  To state this condition, define maps
$$\delta^k:N\to (A[1])^{\otimes k}\otimes N$$
for $k\geq0$ inductively by $\delta^0=\id_N$ and
\begin{equation}
\delta^k = (\id_{A^{\otimes(k-1)}} \otimes \delta^1) \circ \delta^{k-1} \text{ for }k\geq2.
\end{equation}
Define a map $\delta:N\to\overline{T}^*(A[1])$ by
\begin{equation*}
\delta(x) = \sum_{i=0}^\infty \delta^i(x).
\end{equation*}
The compatibility condition is then
\begin{equation}
(\mu \otimes \id_N) \circ \delta(x) = 0.
\end{equation}
\end{definition}

Right type $\tD$ structures can be defined analogously.  We write $\prescript{A}{}{N}$ (respectively $N^A$) to indicate that $N$ is a left (respectively right) type $\tD$ structure over $A$.

A type $\tD$ structure is \emph{bounded} if for any $x \in N$, $\delta^i(x) = 0$ for all sufficiently large $i$.
 
If $M_A$ is a right $\ainf$ module over $A$ and $\prescript{A}{}{N}$ is a left type $\tD$ structure over $A$, and at least one of them is bounded, we can define the \emph{box tensor product} $M\boxtimes N = M_A\boxtimes\prescript{A}{}{N}$ to be the  vector space $M\otimes_\bfk N$ with differential 
$$\bdy: M\otimes N \to (M\otimes N)[1]$$
defined by 
$$\bdy = \sum_{k=1}^{\infty}(m_k\otimes \id_N)\circ(\id_M\otimes \delta^{k-1}).$$
The boundedness condition guarantees that the above sum is finite. In that case $\bdy^2= 0$, and $M\boxtimes N$ is a graded chain complex.  The box tensor product is a model for the derived tensor product of two type $\tA$ modules, as we explain in the following subsection.  The idea: $\prescript{}{A}{A}_A\boxtimes\prescript{A}{}{N}$ is a left type $\tA$ module, and $M\boxtimes N$ is homotopy equivalent to the usual derived tensor product (which we define below) $M\dtensor{A}(A\boxtimes N)$.  The advantage to $M\boxtimes N$ is that it is often finite dimensional, while $M\dtensor{A}(A\boxtimes N)$ is always infinite dimensional.

Given dg algebras $A$ and $B$ over  $\bfk$ and  $\bfj$ with differentials and multiplications $d_A$, $d_B$, $\mu_A$, and $\mu_B$, respectively, four types of bimodules can be defined in a similar way: types $\tDD$, $\tAA$, $\tDA$, and $\tAD$. See \cite[Section 2.2.4]{bimod} for details; we will review them briefly.

An \emph{$\ainf$ bimodule} or \emph{type $\tAA$ bimodule over $(A,B)$} is a graded  $(\bfk, \bfj)$-bimodule $M$, together with degree $0$  maps
$$m_{i,1,j}:(A[1])^{\otimes i}\otimes M\otimes (B[1])^{\otimes j}\to M$$
subject to compatibility conditions analogous to those for type $A$ structures, see \cite[Equation 2.2.38]{bimod}.

We assume all type $\tAA$ bimodules to be \emph{strictly unital}: $m_{1,1,0}(1,x) = x = m_{0,1,1}(x,1)$ and $m_{i,1,j}(a_1, \ldots, a_i, x, b_1, \ldots, b_j) = 0$ if $i+j>1$ and some $a_i$ or $b_j$ lies in $\bfk$ or $\bfj$.

A \emph{type $\tDA$ bimodule over $(A,B)$} is a graded  $(\bfk, \bfj)$-bimodule $M$, together with degree $0$, $(\bfk, \bfj)$-linear maps
$$\delta^1_{1+j}: M\otimes B[1]^{\otimes j}\to A\otimes M[1],$$
satisfying a compatibility condition combining those for type $\tA$ and $\tD$ structures, see \cite[Definition 2.2.42]{bimod}. 
 
 A \emph{type $\tAD$ structure} can be defined similarly, with the roles of $A$ and $B$ interchanged. 
 
A \emph{type $\tDD$ structure over $(A,B)$} is a type $\tD$ structure over $A\otimes_{\Ft}B^{\mathrm{op}}$. In other words, it is a graded $(\bfk, \bfj)$-bimodule $M$ and a degree $0$ map $\delta^1: M \to A\otimes M\otimes B[1]$, again with an appropriate compatibility condition.

When $A$ is the trivial algebra $\{1\}$, a type $\tAD$ structure over $(A,B)$ is the same thing as a right type $\tD$ structure over $B$.  Similar statements hold for other bimodule structure types, \emph{mut. mut.}

There are notions of boundedness for bimodules similar to those for one-sided modules. For each compatible pair of bimodule types, there is a corresponding box tensor product. When forming a box tensor product, we always assume that one of the factors is bounded.  We illustrate the idea below for the box tensor product of two type $\tDA$ structures; for details, see \cite[Section 2.3.2]{bimod}. 

Let $M$ be a type $\tDA$ module over $(A,B)$ and $N$ a type $\tDA$ bimodule over $(B,C)$.  As a chain complex, define
\begin{equation}
\prescript{A}{}{M}_B \boxtimes \prescript{B}{}{N}_C = \mathcal{F}_1(\prescript{A}{}{M}_B) \boxtimes \mathcal{F}_2(\prescript{B}{}{N}_C),
\end{equation}
where $\mathcal{F}_1(\prescript{A}{}{M}_B)$ is the right type $\tA$ structure over $B$ obtained from $\prescript{A}{}{M}_B$ by forgetting its left type $\tD$ structure over $A$ (and analogously for $\mathcal{F}_2(\prescript{B}{}{N}_C)$).  This chain complex can be given the structure of a type $\tDA$ bimodule over $(A,C)$ in a natural way.

\subsection{Categories and Grothendieck groups of $\ainf$ modules}\label{subsec:dg-algebra}

Besides \cite[Section 2]{bimod}, helpful references for the material in this section include \cite{BL,Keller,Khgl12}.  We will work in the language of $\ainf$ and dg categories.  While we recommend \cite{Keller} as a reference for these, the reader is invited to not worry about the details and instead let the analogy ``$\ainf$ categories over $\bfk$ are to $\ainf$ algebras over $\bfk$ as $\bfk$-linear categories are to $\bfk$-algebras'' be their guide.  In other words, $\ainf$ categories have higher compositions, and compositions hold up to homotopies of homotopies of\ldots \emph{ad. inf.}

\subsubsection{Categories of type $\tA$ modules}

Let $A$ be an $\ainf$ algebra and $M_A,N_A$ right type $\tA$ modules over $A$.  Define
\begin{equation*}
\MorA{A}(M,N) = \Hom_\bfk (M \otimes T^*(A_+[1]), N),
\end{equation*}
the \emph{morphism space} of type $\tA$ maps from $M$ to $N$.  The usual differential
\begin{equation*}
d(f) = d_N \circ f + f \circ d_M
\end{equation*}
makes $\MorA{A}(M,N)$ into a chain complex.  We use these as the morphism spaces to define the $\ainf$ category $\ModA{A}$ of right type $\tA$ modules over $A$.  The cycles of $\MorA{A}(M,N)$ are the \emph{$\ainf$ homomorphisms} from $M_A$ to $N_A$, and the boundaries are called \emph{null-homotopic} morphisms.  Explicitly, a morphism $f:M\to N$ is null-homotopic if there is a degree $-1$ map $h$ such that
\begin{equation}
h \circ \mu_1^M + \mu_1^N \circ h = f.
\end{equation}
Define the dg category $\AMod{A}$ of left type $\tA$ modules over $A$ analogously.

When $A$ is an $\ainf$ algebra, there are two models for the derived category of $A$-modules: 
\begin{itemize}
    \item the $0$-th homology category $H(\AMod{A})$ (same objects as $\AMod{A}$, quotient each morphism space by the subspace of null-homotopic morphisms);
    \item the localization of $\AMod{A}$ at the class of quasi-isomorphisms (morphisms which induce an isomorphism on homology).
\end{itemize}
These two categories are equivalent as triangulated categories \cite[Proposition 2.4.1]{bimod}.

If $A$ is a dg algebra, then there are three more models for the derived category:
\begin{itemize}
    \item dg modules and homotopy classes of dg module homomorphisms, localized at the class of quasi-isomorphisms;
    \item dg modules and homotopy classes of $\ainf$ homomorphisms;
    \item the localization of the previous at the class of quasi-isomorphisms.
\end{itemize}
All five models are equivalent as triangulated categories \cite[Proposition 2.4.1]{bimod}.  By abuse of notation, write $\dcat{A}$ for any of these triangulated categories.

For the rest of this subsection, we restrict to the case where $A$ is a dg algebra and use the dg modules, dg homomorphisms model of $\dcat{A}$.  In our main construction later, all $\ainf$ algebras will be dg algebras.

In this case, the shift functor $\dcat{A}\to\dcat{A}$ is given on objects by $(M,d_M)\mapsto(M[1],-d_M)$.  The distinguished triangles are those isomorphic to triangles of the form
\begin{equation*}
\xymatrix{M \ar[r]^-f & N \ar[r] & C(f) \ar[r] & M[1],}
\end{equation*}
where $C(f)$ is the cone of $f$, and the maps in and out of $C(f)$ are the evident inclusion and projection maps.  This triangulated structure can also be obtained by viewing the category of dg modules and dg module homomorphisms as a Frobenius category and identifying $\dcat{A}$ with its stable category \cite{Keller}.

Let $\hcat{A}$ be the usual homotopy category of dg modules over $A$: objects are left dg modules over $A$, and morphisms spaces are the quotient of all dg homomorphisms by the null-homotopic homomorphisms.

In the derived category of an ordinary algebra, morphism spaces can be computed by taking projective resolutions.  The analogous notion in dg algebra is that of a cofibrant module (also called ``projective'' in \cite{Khgl12}, ``K-projective'' in \cite{BL}).  We say a dg module $P$ over $A$ is \emph{cofibrant} if whenever we are given a surjective quasi-isomorphism $L\to M$ and a morphism $P\to M$, we can factor the latter through $L$,
\begin{equation*}
\xymatrix{& P \ar[d] \\ L \ar[r]^-{\simeq} \ar@{<--}[ur]^-{\exists} & M.}
\end{equation*}
Let $\hcatcf{A}$ be the smallest subcategory of $\hcat{A}$ containing all cofibrant modules and closed under arbitrary direct sums.  Then the restriction of the localization functor $\hcat{A}\to\dcat{A}$ to $\hcatcf{A}$ is an equivalence of triangulated categories,
\begin{equation*}
\xymatrix{\hcat{A} \ar[r] & \dcat{A}. \\ \hcatcf{A} \ar@{^{(}->}[u] \ar[ur]_-{\simeq}}
\end{equation*}
There is a standard way to choose a cofibrant replacement for a dg (or $\ainf$) module $M$.  There is a type $\tDD$ bimodule $\prescript{A}{}\BBar(A)^A$ (see \cite[Definition 2.3.16]{bimod}) such that
\begin{equation}
\BBar(M_A) = M_A \boxtimes \prescript{A}{}\BBar(A)^A \boxtimes \prescript{}{A}A_A
\end{equation}
is cofibrant for any $M_A$.  This \emph{bar resolution} is functorial in $M_A$ and comes with a canonical map $\BBar(M_A) \to M_A$.  There is the obvious definition for left modules as well.

\begin{remark*} \emph{We do not need to localize at the class of quasi-isomorphisms when we use the ``dg modules and $\ainf$ homomorphisms modulo homotopy'' model because the definition of $\AMor{A}$ already includes the bar resolution.}\end{remark*}

We now turn to Grothendieck groups.  If $\Cc$ is a triangulated category (assumed to be essentially small) with shift functor $X\mapsto X[1]$, then its \emph{Grothendieck group} $K_0(\Cc)$ is the quotient of the free abelian group on the set of symbols
\begin{equation*}
\{[X] \mid X\text{ is an isomorphism class of objects of }\Cc\}
\end{equation*}
by the following two relations:
\begin{enumerate}
    \item $[Y] = [X] + [Z]$ for every triangle $X\to Y\to Z\to X[1]$ isomorphic to a distinguished triangle;
    \item $[X[1]] = -[X]$ for every object $X$.
\end{enumerate}

As with ordinary algebras, we must restrict to a class of suitably small modules in order to get a nontrivial Grothendieck group (recall the ``Eilenberg swindle'', $A^{\oplus\infty}\oplus A\cong A^{\oplus\infty}$).  A useful condition is the following: we say a dg module $M$ over $A$ is \emph{compact} (also called ``small'' in \cite{Keller}) if the functor $\Hom_{\dcat{A}}(M,-)$ commutes with arbitrary direct sums.  Let $\dcatc{A}$ be the smallest full triangulated subcategory of $\dcat{A}$ that contains all compact objects and is closed under isomorphisms; we call this the \emph{compact (or perfect) derived category} of $A$.  Define the Grothendieck group of the dg algebra $A$ to be
\begin{equation}
K_0(A) = K_0(\dcatc{A}).
\end{equation}
This is not necessarily the same as the Grothendieck group of the category of finitely generated (non-dg) projective modules over $A$.

A dg bimodule $M$ over $(A,B)$ gives rise to a \emph{derived tensor functor}
\begin{equation*}\begin{split}
M \dtensor{B} (-): \dcat{B} & \to \dcat{A}, \\
N & \mapsto M \otimes_B \BBar(N).
\end{split}\end{equation*}
There is also a derived hom functor, right adjoint to the derived tensor functor.  If $M\dtensor{B}(-)$ sends compact objects to compact objects, then there is an induced homomorphism
\begin{equation*}
[M \dtensor{B} -] : K_0(B) \to K_0(A).
\end{equation*}

\subsubsection{Categories of type $\tD$ modules; bimodules and functors}

If $\prescript{A}{}{M}$ and $\prescript{A}{}{N}$ are left type $\tD$ structures over $A$ with respective structure maps $\delta^M$ and $\delta^N$, let
\begin{equation}
\DMor{A}(\prescript{A}{}{M},\prescript{A}{}{N}) = \Hom_\bfk(M,A\otimes N).
\end{equation}
Write $h^1$ for an element of this morphism space, and for each $i\geq2$, define $h^i:M\to A^{\otimes i}\otimes N$ by
\begin{equation}
h^i = \sum_{j=0}^{i-1} (\id_{A^{\otimes(j-1)}} \otimes (\delta^N)^{i-j-1}) \circ (\id_{A^{\otimes j}} \otimes h^1) \circ (\delta^M)^j
\end{equation}
and $h:M\to\overline{T}^*(A)\otimes N$ by $h=\sum_{i=1}^\infty h^i$.  We give this space a differential $\partial$ defined by
\begin{equation}
(\partial h)^1 = (\mu \otimes \id_N)\circ h.
\end{equation}
Let $\DuMod{A}$ be the $\ainf$ category of left type $\tD$ structures over $A$ with these morphism complexes (see \cite[Lemma 2.2.27]{bimod} for the definition of the higher composition maps).  When $A$ is a dg algebra, $\DuMod{A}$ will be a dg category.

Cycles in $\DuMor{A}$ are called \emph{homomorphisms of type $\tD$ structures}.  Since any morphism between bounded type $\tD$ structures is bounded \cite[2.2.30]{bimod}, the subcategory of bounded type $\tD$ structures and bounded morphisms in $\DuMod{A}$ is a full $\ainf$ subcategory.  Let $\DMod{A}$ be the full $\ainf$ subcategory of left type $\tD$ structures homotopy equivalent to a bounded type $\tD$ structure.

Functors between categories of left type $\tA$ or type $\tD$ structures can be constructed using bimodules of appropriate types.  For example, suppose $\prescript{}{A}{M}^B$ is a type $\tAD$ bimodule over $(A,B)$, bounded as a type $\tD$ module.  The assignment
\begin{equation*}
\prescript{}{B}{N} \mapsto \prescript{}{A}{M}^B \boxtimes \prescript{}{B}{N}
\end{equation*}
induces functors
\begin{equation*}
\AMod{B} \to \AMod{A}
\end{equation*}
and
\begin{equation*}
\dcat{B} \to \dcat{A}.
\end{equation*}

If we were to define functors on the derived category using dg bimodules (rather than type $\AA$ or $\tD$ bimodules), we would have to use the derived tensor product.  This requires taking a cofibrant replacement of one of the factors; in practice, cofibrant replacements of finite dimensional modules are often infinite dimensional (functorial ones, always so).  As we presently explain, the category of type $\tD$ structures can be used as a model of the category of cofibrant dg modules.  Box tensor product with a type $\tDA$ bimodule takes the place of derived tensor product with a dg (or type $\tAA$) bimodule.  In our main construction, this technique will allow us to compute classes in $K_0(A)$ by dimension counting.

Given a left type $\tD$ structure $\prescript{A}{}{N}$ over a dg algebra $A$, the left box tensor product with the dg bimodule $A$ over $(A,A)$ makes $\prescript{}{A}{A}_A^{}\boxtimes\prescript{A}{}{N}$ a cofibrant left dg module over $A$.  If the resulting module is compact, then there is a corresponding class in $K_0(A)$.  In fact, up to quasi-isomorphism, we get all dg modules over $A$ in this way:

\begin{proposition}[\cite{bimod}, Proposition 2.3.18]\label{prop-A-D-equivalence} Let $A$ be a dg algebra and $\prescript{A}{}{N}$ a left type $\tD$ structure over $A$.  The $\ainf$ functors
\begin{equation*}\begin{split}
\DMod{A} & \to \AMod{A} \\
\prescript{A}{}{N} & \mapsto \prescript{}{A}{A}_A^{} \boxtimes \prescript{A}{}{N}
\end{split}\end{equation*}
and
\begin{equation*}\begin{split}
\AMod{A} & \to \DMod{A} \\
\prescript{}{A}{M} & \mapsto \prescript{A}{}{\BBar(A)}_{}^A \boxtimes \prescript{}{A}{M}
\end{split}\end{equation*}
are homotopy inverse via canonical homotopies.

They intertwine the tensor products $\boxtimes$ and $\dtensor{}$ in the sense that there is a canonical homotopy equivalence
\begin{equation*}
M_A^{} \boxtimes \prescript{A}{}{N} \simeq M_A^{} \dtensor{} ( \prescript{}{A}A_A^{} \boxtimes \prescript{A}{}{N})
\end{equation*}
for any right type $\tA$ module $M$ and left type $\tD$ module $N$ over $A$.  In particular, the categories $\DMod{A}$ and $\AMod{A}$ are quasi-equivalent.  Hence their derived categories are equivalent, and their Grothendieck groups are isomorphic.

Corresponding statements hold with left and right exchanged.
\end{proposition}

We will use Proposition \ref{prop-A-D-equivalence} in the proof of Proposition \ref{prop-K0-tangle-floer} below.

\begin{remark*} \emph{If a dg algebra $A$ is considered as a dg category $A_c$ (objects: primitive idempotents $e\in A$; morphism complexes: primitive idempotented pieces $e'Ae$), then (bounded) type $\tD$ structures over $A$ are equivalent to one-sided twisted complexes over $A_c$.  We have chosen to use the language of type $\tD$ structures because it is more standard in the low dimensional topology literature.  For more on twisted complexes, see \cite{BondalKapranov}.}\end{remark*}

\subsection{Quantum $\gloneone$ and bases for tensor powers of $V$ and $V^*$}\label{subsec-gloneone}

 Much of the exposition of this subsection follows the development of \cite{s1}.

\subsubsection{Quantum $\gloneone$}

Let
\begin{equation}
\gloneone = \End_\C(\C^{1|1})
\end{equation}
be the Lie superalgebra of superlinear endomorphisms of a $(1,1)$ dimensional supervector space.  Its Cartan subalgebra is the supervector space
\begin{equation*}
\hfrak=\mathrm{span}(h_1,h_2)\subset\gloneone,
\end{equation*}
where $h_1$ (respectively $h_2$) is projection onto the subspace $\C^{1|0}$ (respectively $\C^{0|1}$).  Let $\ve_1,\ve_2$ be the basis vectors of $\hfrak^*$ dual to $h_1,h_2$ respectively.  We have weight and coweight lattices
\begin{equation*}
\Pb=\Z\ve_1\oplus\Z\ve_2\subset\hfrak^*,\qquad\Pb^*=\Z h_1\oplus\Z h_2\subset\hfrak,
\end{equation*}
respectively.  The unique positive root is
\begin{equation*}
\alpha=\ve_1-\ve_2.
\end{equation*}
Let $\langle\cdot,\cdot\rangle$ denote the canonical pairing between $\hfrak$ and $\hfrak^*$,
\begin{equation}
\langle h_i,\ve_j\rangle=\delta_{i,j}.
\end{equation}
We define a non-degenerate symmetric bilinear form on $\hfrak^*$ by
\begin{equation}
(\ve_i,\ve_j) = \begin{cases}
1 & \text{if }i=j=1, \\
-1 & \text{if }i=j=2, \\
0 & \text{if }i\neq j.
\end{cases}
\end{equation}

We give a $\Z/2\Z$-grading to the additive group $\Pb$ by
\begin{equation*}
p(\ve_1)=0,\quad p(\ve_2)=1,
\end{equation*}
so that $p(n_1\ve_1+n_2\ve_2)=n_2$.

The Hopf superalgebra $U_q = U_q(\gloneone)$ is defined as follows. The underlying superalgebra is generated over $\C(q)$ by $E$ (odd), $F$ (odd), and $\lbrace\qb^h:h\in\Pb^*\rbrace$ (all even) subject to the relations
\begin{equation*}\begin{split}
& \qb^0=1, \\
& \qb^h\qb^{h'}=\qb^{h+h'}\text{ for all }h,h'\in\Pb^*, \\
& \qb^h E=q^{\langle h,\alpha\rangle}E\qb^h,\text{ for all }h\in\Pb^*, \\
& \qb^h F=q^{-\langle h,\alpha\rangle}F\qb^h\text{ for all }h\in\Pb^*, \\
& E^2=F^2=0, \\
& EF+FE=\frac{K-K^{-1}}{q-q^{-1}}.
\end{split}\end{equation*}
We have used the shorthand definition $K=\qb^{h_1+h_2}$.  Since $\langle h_1+h_2,\alpha\rangle=0$, it follows that $K\in Z(U_q(\gloneone))$. 
The coproduct on $U_q(\gloneone)$ is defined by
\begin{equation*}\begin{split}
& \Delta(E)=E\otimes K^{-1}+1\otimes E, \\
& \Delta(F)=F\otimes1+K\otimes F, \\
& \Delta(\qb^h)=\qb^h\otimes\qb^h.
\end{split}\end{equation*}
The antipode of $U_q(\gloneone)$ is defined by
\begin{equation*}\begin{split}
& S(E)=-EK, \\
& S(F)=-K^{-1}F, \\
& S(\qb^h)=\qb^{-h}.
\end{split}\end{equation*}
Recall that a Hopf superalgebra $A$ over $\C(q)$ acts on a tensor product of $A$-modules $M\otimes N$  by $a\cdot (m\otimes n) = \Delta(a)\cdot (m\otimes n)$, where $A\otimes A$ acts on $M\otimes N$ by $(a_1\otimes a_2)\cdot (m\otimes n) = (-1)^{p(a_2)p(m)}(a_1\cdot m)\otimes (a_2 \cdot n)$. Also recall that $A$ acts on the dual $M^*=\Hom_{\C(q)}(M;\C(q))$ of an $A$-module $M$ by $(a\cdot \varphi)(m) = (-1)^{p(\varphi)p(a)}\varphi(S(a)\cdot m)$.

Some more notation: if $n\in\Z$, we set
\begin{equation*}
[n]_q=\frac{q^n-q^{-n}}{q-q^{-1}}\in\Z[q,q^{-1}].
\end{equation*}
For $\lambda \in \Pb$, define
\begin{equation*}
[\lambda]_q = [\langle h_1+h_2, \lambda \rangle]_q.
\end{equation*}
For a sequence $\undlam=(\lambda_1,\ldots,\lambda_n)\in\Pb^n$, set $[\undlam]_q = [\lambda_1 + \ldots + \lambda_n]_q$ and $\langle h, \undlam \rangle = \sum_j \langle h, \lambda_j \rangle$.

\subsubsection{The representations $L(\lambda)$}\label{ssec:reps}

For each $\lambda\in\hfrak^*$, let $L(\lambda)$ be a two dimensional vector space over $\C(q)$ with basis $\{v_0^\lambda,v_1^\lambda\}$.  We give this the structure of a $U_q$-module by declaring $v_0^\lambda$ and $v_1^\lambda$ to have weights $\lambda$ and $\lambda-\alpha$ and gradings $p(\lambda)$ and $p(\lambda)-1$, respectively, and setting
\begin{equation*}\begin{split}
& E(v_0^\lambda) = 0, \quad F(v_0^\lambda) = [\lambda]_q v_1^\lambda, \\
& E(v_1^\lambda) = v_0^\lambda, \quad F(v_1^\lambda) = 0, \\
& \qb^h(v_0^\lambda) = q^{\langle h,\lambda \rangle}v_0^\lambda, \quad \qb^h(v_1^\lambda) = q^{\langle h,\lambda-\alpha \rangle}v_1^\lambda.\\
\end{split}\end{equation*}
The representation $L(\lambda)$ is irreducible if and only if $\lambda$ does not annihilate $h_1+h_2$.  For a sequence $\undlam=(\lambda_1,\ldots,\lambda_n)$ of weights, let $L(\undlam) = L(\lambda_1) \otimes \ldots \otimes L(\lambda_n)$.

For a binary sequence $\unda=(a_1,\ldots,a_n)$, let
\begin{equation*}
v_{\unda} = v_{a_1}^{\lambda_1} \otimes \cdots \otimes v_{a_n}^{\lambda_n} \in L(\lambda_1) \otimes \cdots \otimes L(\lambda_n).
\end{equation*}
An example of the action of $F$ on a tensor product: on $L(\ve_1)^{\otimes3}$,
\begin{equation}\label{eg:tensor}\begin{split}
F(v_{100}) &= (F\otimes1\otimes1 + K\otimes F\otimes1 + K\otimes K\otimes F)(v_{100})\\
&= -qv_{110} -q^2v_{101}.
\end{split}\end{equation}

There is a handy exterior algebra model for $L(\undlam)$.  Let $W$ be a vector space over $\C(q)$ with basis $e_1,\ldots,e_n$.  If $\undlam\in\Pb^n$, then for each $1\leq j\leq n$, let
\begin{equation*}
\undlam_{<j}=(\lambda_1,\ldots,\lambda_{j-1}),\quad\undlam_{>j}=(\lambda_{j+1},\ldots,\lambda_n).
\end{equation*}
{and similarly let
\begin{equation*}
\undlam_{\leq j}=(\lambda_1,\ldots,\lambda_{j}),\quad\undlam_{\geq j}=(\lambda_{j},\ldots,\lambda_n).
\end{equation*}}
For $\lambda \in \Pb$, define
\begin{equation*}
q^\lambda=q^{\langle h_1+h_2,\lambda\rangle},
\end{equation*}
and for any sequence $\undlam\in\Pb^n$, let
\begin{equation*}
p(\undlam) = \sum_{j=1}^np(\lambda_j), \quad q^{\undlam} = q^{\sum_{j=1}^n\lambda_j}.
\end{equation*}

Define the vectors
\begin{equation}
\ell_E = \sum_{j=1}^n (-1)^{p(\undlam_{<j})} q^{-\undlam_{>j}} e_j, \quad
\ell_F = \sum_{j=1}^n (-1)^{p(\undlam_{<j})} q^{\undlam_{<j}} [\lambda_j]_q e_j
\end{equation}
and the operators
\begin{equation}
E = \ell_E \lrcorner, \quad F = \ell_F \wedge
\end{equation}
on $\Lambda^*W$.  Here $\lrcorner$ is the contraction operator,
\begin{equation*}
e_j\lrcorner e_k=\delta_{j,k}.
\end{equation*}
Note that
\begin{equation*}
(e_i\lrcorner)(e_j\wedge)+(e_j\wedge)(e_i\lrcorner)=\delta_{i,j}\id.
\end{equation*}
It is easy to check that
\begin{equation*}
EF+FE=[\undlam]_q\id.
\end{equation*}
In fact, it is not hard to check that we have an isomorphism of $U_q(\gloneone)$-modules
\begin{equation}\begin{split}
\Lambda^*W&\to L(\undlam)\\
e_{i_1}\wedge\cdots\wedge e_{i_k}&\mapsto v_{S(\undi)},
\end{split}\end{equation}
where $\undi=(i_1,\ldots,i_k)$, and $S(\undi)$ is the binary sequence of length $n$ with a $1$ in the $j$-th slot if and only if $j\in\undi$.

The subspace of highest weight vectors of $L(\undlam)$ (i.e., those $v\in L(\undlam)$ such that $E(v)=0$) has half the dimension of $L(\undlam)$.  For each $i=1,2,\ldots,n-1$, let
\begin{equation}
\ell_i = -(-1)^{p(\lambda_i)} e_i + q^{-\lambda_{i+1}} e_{i+1}.
\end{equation}
For $I\subseteq\{1,2,\ldots,n-1\}$ with elements $i_1<\ldots<i_k$, let
\begin{equation}
\ell_I=\ell_{i_1}\wedge\cdots\wedge \ell_{i_k}.
\end{equation}
It is easy to check that $\{\ell_I \mid I\subseteq\{1,2,\ldots,n-1\}\}$ is a basis for the space of highest weight vectors of $L(\undlam)$.  If $\langle h_1+h_2, \undlam\rangle\neq0$, then $L(\undlam)$ splits into a direct sum of two-dimensional simple modules, so that 
\begin{equation*}
\{\ell_I, \ell_F \wedge \ell_I \mid I\subseteq\{1,2,\ldots,n-1\}\}
\end{equation*}
is a basis of $L(\undlam)$.  If $\langle h, \undlam\rangle=0$, however, $F(\ell_I)$ is still a highest weight vector, so the highest weight vectors do not generate $L(\undlam)$.  In this case, $B_{nss}=\{ \ell_I, \ell_E \wedge \ell_I\}$ is a basis with several desirable properties.  This is a basis in both the semisimple and non-semisimple cases, and the matrix coefficients of $E$ and $F$ with respect to this basis admit a uniform description across all cases.

Our main construction in Section \ref{sec:tangle}, however, categorifies a different basis than $B_{nss}$ and only makes use of the semisimple case.  We now turn to the representations and bases we will categorify.

\subsubsection{The representations $V_{\PP} \otimes L(\lambda_{n+1})$}\label{subsubsec-our-rep}

For short, set $V=L(\ve_1)$ (the \emph{vector representation}) and $V^*=L(-\ve_2)$ (the \emph{dual vector representation}).  We also write $V_1$ for $V$ and  $V_{-1}$ for $V^*$.  For a sign sequence $\PP=(P_1,\ldots,P_n)\in\{\pm1\}^n$, let
\begin{equation*}
V_{\PP} = V_{P_1} \otimes V_{P_2} \otimes \cdots \otimes V_{P_n}.
\end{equation*}
 Define a weight sequence $\undlam=(\lambda_1, \ldots, \lambda_{n+1})$ by setting
\begin{equation}
\lambda_i = 
\begin{cases}
\ve_1 & \text{if }1\leq i\leq n\text{ and }\PP_i = 1, \\
-\ve_2 & \text{if }1\leq i\leq n\text{ and }\PP_i = -1,
\end{cases}
\end{equation}
and $\lambda_{n+1} = \ve_1-\sum_{i=1}^n \lambda_i$ (in fact, any weight $\lambda_{n+1}$ such that $\langle h_1+h_2, \lambda_{n+1} \rangle = 1-\sum_{i=1}^n \PP_i$ would work in this paper).  It follows that 
\begin{equation*}
[\lambda_{n+1}]_q = \left[ 1 - \sum_i \PP_i \right]_q, \quad q^{\lambda_{n+1}} = q^{1-\sum_i \PP_i }.
\end{equation*}
For the remainder of this subsection we work with the tensor product $L(\undlam) = V_{\PP} \otimes L(\lambda_{n+1})$. Since $\lambda_{n+1}$ depends on the sum of the elements in $P$, rather than on the length of $P$, we use the notation $\lambda_P$ for $\lambda_{n+1}$ in all other sections.

When regarding the formulas below, bear in mind that $L(\undlam)$ is a tensor product of $n+1$ two dimensional representations, instead of $n$ as in Section~\ref{ssec:reps}.
The representation $L(\undlam)$ has exterior-generating highest weight vectors $\ell_1,\ldots,\ell_n$ as described above:
\begin{equation*}
\ell_i = -(-1)^{p(\lambda_{i})} e_i + q^{-\lambda_{i+1}} e_{i+1}.
\end{equation*}
{Note that $\langle h_1+h_2, \lambda_{i} \rangle = \PP_i$  and $(-1)^{p(\lambda_{i})}=P_i$ for $1\leq i\leq n$  , so  we can write
\begin{equation*}
\ell_i = -\PP_i e_i + q^{-\lambda_{i+1}} e_{i+1}
\end{equation*}
and 
\begin{equation*}
\ell_E = \sum_{j=1}^{n+1} (\prod_{i=1}^{j-1}\PP_iq^{\PP_i}) q^{\lambda_j - 1} e_j, \quad
\ell_F = \sum_{j=1}^{n+1} (\prod_{i=1}^{j-1}\PP_iq^{\PP_i})  [\lambda_j]_q e_j
\end{equation*}
Let
\begin{equation}\label{eqn-defn-ell-zero}
\ell_0  =  \ell_F - \sum_{j=1}^n (\prod_{i=1}^j \PP_iq^{-\PP_i}) \ell_j .
\end{equation}
Note that 
\begin{equation*}
E(\ell_0) = E(\ell_F) = EF(1) = 1.
\end{equation*}
Using the derivation property of contraction, it follows that 
\begin{equation}\label{eqn-E-ell-zero}
E(\ell_0 \wedge \ell_I) = \ell_I
\end{equation}
for any $I \subseteq \{1,\ldots,n\}$.  We also have
\begin{equation}\label{eqn-F-ell-zero}\begin{split}
& F(\ell_I)= \ell_0 \wedge \ell_I + \sum_{j=1}^n (\prod_{i=1}^j \PP_iq^{-\PP_i}) \ell_j \wedge \ell_I , \\
& F(\ell_0 \wedge \ell_I) = - \sum_{j=1}^n (\prod_{i=1}^j \PP_iq^{-\PP_i})  \ell_0 \wedge \ell_j \wedge \ell_I .
\end{split}\end{equation}
It follows that $\{\ell_0,\ell_1,\ldots,\ell_n\}$ is a basis for $W$, so that
\begin{equation}
B = \{ \ell_I, \ell_0\wedge\ell_I \mid I \subseteq\{1,\ldots,n\} \}
\end{equation}
is a basis for $\Lambda^*W$.  We will always consider $B$ with respect to the \emph{complement reverse lexicographic order}: let $\ell_{\sss}$, $\ell_{\ttt}$ be two elements of $B$; $\sss,\ttt \subseteq \{0,\ldots,n\}$.  Let $\sss',\ttt'$ be the complements of $\sss,\ttt$ respectively.  Let $w(\sss'),w(\ttt')$ be the words in the alphabet $\{0,\ldots,n\}$ obtained respectively by reading the elements of $\sss',\ttt'$ backwards.  We say that $\ell_{\sss}$ precedes $\ell_{\ttt}$ if $w(\sss')$ precedes $w(\ttt')$ in the lexicographic (``alphabetical'') order.  A word precedes any of its initial subwords.  For example, if $n = 2$, the lexicographic order on the subsets $\sss'$ is
\begin{equation*}
\emptyset \enskip<\enskip \{0\} \enskip<\enskip \{1\} \enskip<\enskip \{0,1\} \enskip<\enskip \{2\} \enskip<\enskip \{0,2\} \enskip<\enskip \{1,2\} \enskip<\enskip \{0,1,2\}.
\end{equation*}
The induced order on $B$ is
\begin{equation*}
\ell_0 \wedge \ell_1 \wedge \ell_2 \enskip<\enskip \ell_1 \wedge \ell_2 \enskip<\enskip \ell_0 \wedge \ell_2 \enskip<\enskip \ell_2 \enskip<\enskip \ell_0 \wedge \ell_1 \enskip<\enskip \ell_1 \enskip<\enskip \ell_0 \enskip<\enskip 1.
\end{equation*}

Zhang has constructed a canonical basis for the representations $V^{\otimes n}$ ($n>0$) coming from a super Howe duality with the Hecke algebra in type $A$ \cite{Zhang}.  Since we study a different representation (except in the trivial case $P = ()$), we cannot directly compare our basis with that of \cite{Zhang}.  Zhang's basis also appears in \cite{Sartori2}, the results of which we expect to be related to ours.

An easy inductive argument computes the matrices $[E]_B,[F]_B$ of $E$ and $F$ with respect to $B$.  In the base case $n=0$, $\ell_E=\ell_F=\ell_0=e_1$, so $E(1)=e_1\lrcorner 1 = 0$, $E(l_0)= e_1\lrcorner e_1 = 1$, $F(1) = e_1=l_0$, and $F(l_0)=e_1\wedge e_1 = 0$. The ordering on the basis is $\ell_0<1$, so the matrices are
\begin{equation*}
[E]_B = \begin{pmatrix} 0 & 0 \\ 1 & 0 \end{pmatrix}, \quad 
[F]_B = \begin{pmatrix} 0 & 1 \\ 0 & 0 \end{pmatrix}.
\end{equation*}
The inductive step: let $\PP=(P_1,\ldots,P_n)$ and $\PP' = (P_1,\ldots,P_{n-1})$.  The subsets sequence for the ordered basis $B'$ for $V_{\PP'} \otimes L(\lambda_n)$ is just the first half of the corresponding sequence for the ordered basis $B$ for $V_{\PP} \otimes L(\lambda_{n+1})$. Looking at \eqref{eqn-E-ell-zero} and \eqref{eqn-F-ell-zero}, we see that $[E]_B$ and $[F]_B$ have block forms
\begin{equation}\label{eqn-E-F-matrix-induction}
[E]_B = \begin{pmatrix} [E]_{B'} & 0 \\ 0 & [E]_{B'} \end{pmatrix}, \quad 
[F]_B = \begin{pmatrix} [F]_{B'} & D_B \\ 0 & [F]_{B'} \end{pmatrix},
\end{equation}
where $D_B$ is a diagonal matrix.  For a subset $\sss\subseteq\{0,1,\ldots,n-1\}$, the $\sss,\sss\sqcup\{n\}$ matrix entry (which is in the $D_B$ part) is
\begin{equation*}
(-1)^{|\sss|} \prod_{i=1}^n \left( -(-1)^{p(\lambda_i)} q^{-\lambda_i} \right) = (-1)^{|\sss| + n + p(\undlam_{\leq n})} q^{-\undlam_{\leq n}}.
\end{equation*}

\subsubsection{Ribbon category structure}\label{subsubsec-rib}

Although $U_q$ is not a ribbon Hopf superalgebra, its module category is a ribbon category \cite{s1}.  One way of specifying this extra data is to give a functor from oriented framed tangles to the category of $U_q$-modules.  So we will need to write down maps of $U_q$-modules that are the images of ribbon twists, crossings,  caps, and cups under this functor.

We begin with crossings.  If $W_i, W_j$ are $U_q$-modules, let $\Rv=\Rv_{W_i,W_j}:W_i\otimes W_j\to W_j\otimes W_i$ be the intertwiner coming from the braiding structure, so that
\begin{equation*}\begin{split}
& (\Rv_{W_2,W_3}\otimes\id_{W_1})(\id_{W_2}\otimes\Rv_{W_1,W_3})(\Rv_{W_1,W_2}\otimes\id_{W_3}) \\
& \qquad =
(\id_{W_3}\otimes\Rv_{W_1,W_2})(\Rv_{W_1,W_3}\otimes\id_{W_2})(\id_{W_1}\otimes\Rv_{W_2,W_3}).
\end{split}\end{equation*}
On the representations $V_{\PP} \otimes L(\lambda_{n+1})$ of our main construction, we will only need to consider crossings among the factors coming from $V_{\PP}$, all of which are isomorphic to either $V$ or $V^*$.  For $1\leq i\leq n-1$, define
\begin{equation*}\begin{split}
& \Rv_i: V_{\PP} \otimes L(\lambda_{n+1}) \to V_{s_i(\PP)} \otimes L(\lambda_{n+1}), \\
& \Rv_i = \id_{V_{\PP_{<i}}} \otimes \Rv_{V_{P_i,P_{i+1}}} \otimes \id_{V_{\PP_{>(i+1)}} \otimes L(\lambda_{n+1})}.
\end{split}\end{equation*}
(Here, the simple transposition $s_i$ acts on $\{\pm1\}^n$ in the obvious way.)  We never apply the $R$-matrix that crosses $V_{P_{n}}$ with $L(\lambda_{n+1})$.  Denote the weight sequences for $V_{\PP} \otimes L(\lambda_{n+1})$ and $V_{s_i(\PP)} \otimes L(\lambda_{n+1})$ by $\undlam$ and $\undlam'$, the vector spaces for the exterior algebra models by $W$ and $W'$, the generating vectors by $e_i$ and $e_i'$ (or $\ell_i$ and $\ell_i'$), and the ordered bases from Section~\ref{subsubsec-our-rep} by $B$ and $B'$, respectively. Define a linear map $f_i:W'\to W$ by
\begin{equation}
f_i(e_j') = \begin{cases}
e_j & \text{if }j\neq i, i+1, \\
P_i P_{i+1} (1-q^{2P_{i+1}})e_i + P_i q^{P_i} e_{i+1} & \text{if }j = i, \\
P_{i+1} q^{P_{i+1}} e_i & \text{if }j = i+1.\end{cases}
\end{equation}
If we extend $f_i$ to a map from  $\Lambda^*W'$ to $\Lambda^*W$  by defining
\begin{equation}\label{eqn-f-wedge}
f_i(x \wedge y) = f_i(x) \wedge f_i(y),
\end{equation}
then the inverse of $\Rv_i$ is given by
\begin{equation}\label{eqn-R-f}
\Rv_i^{-1} = \left(P_i q^{-P_i}\right)^{\delta_{P_i,P_{i+1}}} f_i.
\end{equation}

From this description it follows that $\Rv_i$ intertwines the action of $E$ and $F$ on $V_{\PP} \otimes L(\lambda_{n+1})$ and $V_{s_i(\PP)} \otimes L(\lambda_{n+1})$.  Thus
\begin{equation}\label{eqn-R-ell-F-ell-I}
\Rv_i^{-1}(\ell_F' \wedge \ell_I') = \ell_F \wedge \Rv_i^{-1}(\ell_I'),
\end{equation}
and in particular,
\begin{equation*}
\Rv_i^{-1}(\ell_F') = \left(P_i q^{-P_i}\right)^{\delta_{P_i,P_{i+1}}} \ell_F.
\end{equation*}
It is easy to check that the action of $\Rv_i^{-1}$ on $\ell_0, \ldots, \ell_n$ is given by
\begin{equation}\label{eqn-R-ell-i}\begin{split}
\Rv_i^{-1}(\ell_j')  & = \left(P_i q^{-P_i}\right)^{\delta_{P_i,P_{i+1}}} \ell_j \quad\textrm{ if }j\neq i-1, i, i+1\\
\Rv_i^{-1}(\ell_{i-1}') & = \left(P_i q^{-P_i}\right)^{\delta_{P_i,P_{i+1}}} \ell_{i-1} + \left(P_i q^{P_i}\right)^{\delta_{P_i,-P_{i+1}}} \ell_i, \\
\Rv_i^{-1}(\ell_i') & = \left(-P_i q^{P_i}\right)^{\delta_{P_i,P_{i+1}}} \ell_i, \\
\Rv_i^{-1}(\ell_{i+1}') & = \left(P_{i+1} q^{P_{i+1}}\right)^{\delta_{P_i,-P_{i+1}}} \ell_i + \left(P_i q^{-P_i}\right)^{\delta_{P_i,P_{i+1}}} \ell_{i+1}.
\end{split}\end{equation}

Equations \eqref{eqn-R-ell-F-ell-I} and \eqref{eqn-R-ell-i} and the exterior homomorphism property \eqref{eqn-f-wedge} suffice to compute the matrix $[\Rv_i^{-1}]_{B,B'}$ (where rows correspond to the basis $B$ and columns correspond to $B'$).

The image of $\Rv_i^{-1}$ restricted to the subspace spanned by $\ell_{i-1}', \ell_i', \ell_{i+1}'$ is contained in the subspace spanned by $\ell_{i-1}, \ell_i, \ell_{i+1}$. Below, subsets of basis elements are listed in the order we get if we first sort by weight, and then use the induced ordering from Section~\ref{subsubsec-our-rep} within each weight. With columns corresponding to  $\ell_{i-1}'\wedge \ell_i'\wedge \ell_{i+1}'$, $\ell_{i}' \wedge \ell_{i+1}'$, $\ell_{i-1}' \wedge \ell_{i+1}'$, $\ell_{i-1}' \wedge \ell_{i}'$, $\ell_{i+1}'$, $\ell_i'$, $\ell_{i-1}'$, $1$ and rows corresponding to $\ell_{i-1}\wedge \ell_i\wedge \ell_{i+1}$, $\ell_{i} \wedge \ell_{i+1}$, $\ell_{i-1} \wedge \ell_{i+1}$, $\ell_{i-1} \wedge \ell_{i}$,  $\ell_{i+1}$, $\ell_i$, $\ell_{i-1}$, $1$, the submatrix for $\Rv_i^{-1}$ is

\begin{center}
\begin{tabular}{l l }
 $\left(\begin{array}{c| c c c | c c c | c} 
      -q & & & & & & & \\
            \hline
       &  -q & 1 & 0 & & & &\\
       &  0& q^{-1} &0 & & & &\\
       &  0 & 1 &-q & & & &\\
      \hline
      & & & &q^{-1} & 0 & 0 &  \\
      & & & & 1 & -q & 1&  \\
       & & & & 0 & 0 & q^{-1}&  \\
       \hline
       & & & & & & & q^{-1} 
   \end{array}\right)$ & if $(P_i, P_{i+1}) = (+, +)$,\\
   &  \\
       $\left(\begin{array}{c| c c c | c c c | c} 
      1 & & & & & & & \\
      \hline
       & 1 & q & 0 & & & &\\
       & 0& 1 &0 & & & &\\
       & 0 & -q^{-1} & 1 & & & &\\
       \hline
      & & & &1 & 0 & 0 & \\
      & & & & -q^{-1} & 1 & q & \\
       & & & & 0 & 0 & 1& \\
       \hline
       & & & & & & & 1
   \end{array}\right)$ & if $(P_i, P_{i+1}) = (+,-)$,\\
   & \\
    $\left(\begin{array}{c| c c c | c c c | c} 
      1 & & & & & & & \\
      \hline
       & 1 & -q^{-1} & 0 & & & &\\
       & 0& 1 &0 & & & & \\
       & 0 & q & 1 & & & &\\
       \hline
      & & & &1 & 0 & 0 & \\
      & & & & q & 1 & -q^{-1} &  \\
       & & & & 0 & 0 & 1& \\
       \hline
       & & & & & & & 1
   \end{array}\right)$ & if $(P_i, P_{i+1}) = (-,+)$,\\
   & \\
    $\left(\begin{array}{c| c c c | c c c | c} 
      q^{-1} & & & & & & & \\
       \hline
       &   q^{-1} & 1 & 0 & & & &\\
       &  0& -q &0 & & & &\\
       &  0 & 1 &  q^{-1} & & & &\\
       \hline
      & & & & -q  & 0 & 0 & \\
      & & & & 1 &  q^{-1} & 1& \\
      & & & & 0 & 0 & -q & \\
       \hline
       & & & & & & &  -q
   \end{array}\right)$ & if $(P_i, P_{i+1}) = (-, -)$,
\end{tabular}
\end{center}
where all omitted entries are zeros, since they correspond to incompatible pairs of weights. For $I\subseteq \{0,\ldots, n\}\setminus \{i-1, i, i+1\}$ and $J\subseteq  \{i-1, i, i+1\}$, we  have $\Rv_i^{-1}(\ell_I'\wedge \ell_J')  = \ell_I\wedge \Rv_i^{-1}(\ell_J)$. Note that $P_i$ and $P_{i+1}$ refer to the codomain of $\Rv_i^{-1}$. For example, the third matrix covers the case $(P_i, P_{i+1})=(-,+)$ and $(s_i(P_i), s_i(P_{i+1}))=(+,-)$, i.e. locally $\Rv_i^{-1}$ is applied to $V\otimes V^*$. 

The maps assigned to left-oriented caps and cups are just the canonical evaluation and coevaluation maps.  With label $V$, in the standard basis we have been using for $V$ and $V^*$, they are
\begin{equation}\begin{split}
\lcap=\ev: V^* \otimes V &\to \C(q) \\
v_{00},v_{11} &\mapsto 0, \\
v_{10} &\mapsto 1, \\
v_{01} &\mapsto -q,
\end{split}\end{equation}
\begin{equation}\begin{split}
\lcup=\coev: \C(q) &\mapsto V \otimes V^*\\
1 &\mapsto -q^{-1}v_{10} + v_{01}.
\end{split}\end{equation}
For any two representations $W_i$ and $W_j$, let $\sigma_{W_iW_j}$ be the super-flip map which takes $w_i\otimes w_j\mapsto(-1)^{p(w_i)p(w_j)}w_j\otimes w_i$.  The adjusted evaluation and coevaluation maps
\begin{equation*}
\widehat{\ev}=\ev\circ\sigma_{VV^*}, \quad \widehat{\coev}=\sigma_{VV^*}\circ\coev
\end{equation*}
are intertwiners.  To account for framing, though, we will adjust these by a ribbon element which is defined in a variant $U_\hbar$ of $U_q$ (see \cite[Section 2]{s1} for details as well as definitions of the notation used below).  Let
\begin{equation}\begin{split}
u & = \left(1+(q-q^{-1})EFK\right)e^{\hbar(H_2^2-H_1^2)}, \\
v & = K^{-1}u = uK^{-1}
\end{split}\end{equation}
in $U_\hbar$.  The elements $u$ and $v$ are both central; $v$ acts by $1$ on both $V$ and $V^*$, and $u$ acts by $q$ on $V$ and by $q^{-1}$ on $V^*$.  Define the right-oriented cap and cup maps to be $\rcap=\widehat{\ev}\circ(uv^{-1}\otimes\id)$, $\rcup=(\id\otimes vu^{-1})\circ\widehat{\coev}$.  For $V$, these are
\begin{equation}\begin{split}
\rcap: V \otimes V^* &\to \C(q) \\ 
v_{00},v_{11} & \mapsto 0, \\
v_{10} &\mapsto q^2, \\
v_{01} &\mapsto q,
\end{split}\end{equation}
\begin{equation}\begin{split}
\rcup: \C(q) &\to V^* \otimes V, \\
1 &\mapsto q^{-1}v_{10} + q^{-2}v_{01}.
\end{split}\end{equation}

As with crossings, we can define $\lcap_i$ to act as $\lcap$ on $V_{P_i} \otimes V_{P_{i+1}}$ and the identity elsewhere; and analogously for $\rcap_i,\lcup_i,\rcup_i$.  We only do this for $1 \leq i \leq n-1$ (we never cap or cup with the $L(\lambda_{n+1})$ factor).  
For a sign sequence $\PP=(P_1, \ldots, P_n)$, we compute the cap and cup maps on $V_{\PP} \otimes L(\lambda_{n+1})$ with respect to the bases of interest.  Denote the weight sequences for the codomain and domain by $\undlam$ and $\undlam'$, the vector spaces for the exterior algebra models by $W$ and $W'$, the generating vectors by $e_i$ and $e_i'$ (or $\ell_i$ and $\ell_i'$), and the ordered bases from Section~\ref{subsubsec-our-rep} by $B$ and $B'$, respectively.

The image of $\lcap_i$ (and similarly the image of $\rcap_i$) restricted to the subspace spanned by $\ell_{i-1}', \ell_i', \ell_{i+1}'$ is contained in the subspace spanned by $\ell_{i-1}$. With columns corresponding to   $\ell_{i-1}'\wedge \ell_i'\wedge \ell_{i+1}', \ell_{i}' \wedge \ell_{i+1}', \ell_{i-1}' \wedge \ell_{i+1}', \ell_{i-1}' \wedge \ell_{i}' ,  \ell_{i+1}', \ell_i', \ell_{i-1}', 1$ and rows corresponding to $\ell_{i-1}, 1$ (again we use the induced ordering from Section~\ref{subsubsec-our-rep}, but sorted by weight), the submatrices for $\lcap_i$ and $\rcap_i$ are the same, given by
\[
\left(
 \begin{array}{c| c c c | c c c | c} 
       & 0 & q & 0 & & & & \\
       \hline
       & & & & q & 0 & q & 
   \end{array}\right).
\]
For $I_1\subseteq \{0,\ldots, i-2\}$, $I_2\subseteq \{i+2, \ldots, n\}$, and $J\subseteq  \{i-1, i, i+1\}$, we  have 
\begin{align*}
\lcap_i(\ell_{I_1}'\wedge \ell_J'\wedge\ell_{I_2})  &= \ell_{I_1}\wedge \lcap_i(\ell_J)\wedge \ell_{I_2(-2)},\\
\rcap_i(\ell_{I_1}'\wedge \ell_J'\wedge\ell_{I_2})  &= \ell_{I_1}\wedge \rcap_i(\ell_J)\wedge \ell_{I_2(-2)},
\end{align*}
where $I_2(-2)$ is the set obtained from $I_2$ by subtracting $2$ from each element. 

Similarly, the image of $\lcup_i$ (and similarly the image of $\rcup_i$) restricted to the subspace spanned by $\ell_{i-1}'$ is contained in  the subspace spanned by $\ell_{i-1}, \ell_i, \ell_{i+1}$. With columns corresponding to $\ell_{i-1}', 1$  and rows corresponding to $\ell_{i-1}\wedge \ell_i\wedge \ell_{i+1}, \ell_{i} \wedge \ell_{i+1},  \ell_{i-1} \wedge \ell_{i+1}, \ell_{i-1} \wedge \ell_{i} ,  \ell_{i+1}, \ell_i, \ell_{i-1}, 1$, the submatrices for $\lcap_i$ and $\rcap_i$ are the same, given by
\[
\left(
 \begin{array}{c| c}
       & \\
       \hline
      q^{-1} & \\
      0 & \\
      q^{-1} & \\
      \hline
       & 0 \\
       & q^{-1} \\
       & 0 \\
       \hline
       & 
          \end{array}\right).
\] 
For $I_1\subseteq \{0,\ldots, i-2\}$, $I_2\subseteq \{i, \ldots, n\}$, and $J\subseteq  \{i-1\}$, we  have 
\begin{align*}
\lcup_i(\ell_{I_1}'\wedge \ell_J'\wedge_{I_2})  &= \ell_{I_1}\wedge \lcup_i(\ell_J)\wedge \ell_{I_2(2)},\\
\rcup_i(\ell_{I_1}'\wedge \ell_J'\wedge_{I_2})  &= \ell_{I_1}\wedge \rcup_i(\ell_J)\wedge \ell_{I_2(2)},
\end{align*}
where $I_2(2)$ is the set obtained from $I_2$ by adding $2$ to each element.

\subsubsection{The Alexander polynomial from $U_q(\gloneone)$}\label{subsubsec-Q}

Let $\mathcal{OTAN}$ be the monoidal category of oriented tangles.  The crossing, cap, and cup maps from Section~\ref{subsubsec-rib} give a monoidal functor
\begin{equation*}
Q: \mathcal{OTAN} \to \text{$U_q$--mod}
\end{equation*}
as follows.

\noindent
\begin{tabular}{@{} c c c c}
$Q\left(
  \mathcenter{
    \begin{tikzpicture}
      \node[inner sep=1pt] at (-.3,-.5) (bl) {};
      \node[inner sep=1pt] at (.3,-.5) (br) {};
      \node[inner sep=1pt] at (-.3,.5) (tl) {};
      \node[inner sep=1pt] at (.3,.5) (tr) {};
      \node at (0,0) (center) {};
      \draw[->, thick, red] (bl) to (tr);
      \draw[-, thick, red] (br) to (center);
      \draw[->, thick, red] (center) to (tl);
    \end{tikzpicture}
  }
  \right) =    
    \mathcenter{
     \begin{tikzpicture}
      \node[inner sep=2pt]  at (0,-.7) (b) {$V\otimes V$};
      \node[inner sep=2pt]  at (0,.7) (t) {$V\otimes V$};
      \draw[->, thick] (b) to node[left, inner sep=2pt]{$\Rv$} (t);
    \end{tikzpicture}
  }$
  &
  $Q\left(
  \mathcenter{
    \begin{tikzpicture}
      \node[inner sep=1pt] at (-.3,-.5) (bl) {};
      \node[inner sep=1pt] at (.3,-.5) (br) {};
      \node[inner sep=1pt] at (-.3,.5) (tl) {};
      \node[inner sep=1pt] at (.3,.5) (tr) {};
      \node at (0,0) (center) {};
      \draw[->, thick, red] (tr) to (bl);
      \draw[-, thick, red] (br) to (center);
      \draw[->, thick, red] (center) to (tl);
    \end{tikzpicture}
  }
  \right) =    
    \mathcenter{
     \begin{tikzpicture}
      \node[inner sep=2pt] at (0,-.7) (b) {$V^*\otimes V$};
      \node[inner sep=2pt] at (0,.7) (t) {$V\otimes V^*$};
      \draw[->, thick] (b) to node[left]{$\Rv$} (t);
    \end{tikzpicture}
  }$
&
  $Q\left(
  \mathcenter{
    \begin{tikzpicture}
      \node[inner sep=1pt] at (-.3,-.5) (bl) {};
      \node[inner sep=1pt] at (.3,-.5) (br) {};
      \node[inner sep=1pt] at (-.3,.5) (tl) {};
      \node[inner sep=1pt] at (.3,.5) (tr) {};
      \node at (0,0) (center) {};
      \draw[->, thick, red] (bl) to (tr);
      \draw[->, thick, red] (center) to (br);
      \draw[-, thick, red] (tl) to (center);
    \end{tikzpicture}
  }
  \right) =    
    \mathcenter{
     \begin{tikzpicture}
      \node[inner sep=2pt] at (0,-.7) (b) {$V^*\otimes V^*$};
      \node[inner sep=2pt] at (0,.7) (t) {$V^*\otimes V^*$};
      \draw[->, thick] (b) to node[left]{$\Rv$} (t);
    \end{tikzpicture}
  }$
  &
  $Q\left(
  \mathcenter{
    \begin{tikzpicture}
      \node[inner sep=1pt] at (-.3,-.5) (bl) {};
      \node[inner sep=1pt] at (.3,-.5) (br) {};
      \node[inner sep=1pt] at (-.3,.5) (tl) {};
      \node[inner sep=1pt] at (.3,.5) (tr) {};
      \node at (0,0) (center) {};
      \draw[->, thick, red] (tr) to (bl);
      \draw[->, thick, red] (center) to (br);
      \draw[-, thick, red] (tl) to (center);
    \end{tikzpicture}
  }
  \right) =    
    \mathcenter{
     \begin{tikzpicture}
      \node[inner sep=2pt] at (0,-.7) (b) {$V\otimes V^*$};
      \node[inner sep=2pt] at (0,.7) (t) {$V^*\otimes V$};
      \draw[->, thick] (b) to node[left]{$\Rv$} (t);
    \end{tikzpicture}
  }$\\
  & & & \\
    $Q\left(
  \mathcenter{
    \begin{tikzpicture}
      \node[inner sep=1pt] at (-.3,-.5) (bl) {};
      \node[inner sep=1pt] at (.3,-.5) (br) {};
      \node[inner sep=1pt] at (-.3,.5) (tl) {};
      \node[inner sep=1pt] at (.3,.5) (tr) {};
      \node at (0,0) (center) {};
      \draw[->, thick, red] (br) to (tl);
      \draw[-, thick, red] (bl) to (center);
      \draw[->, thick, red] (center) to (tr);
    \end{tikzpicture}
  }
  \right) =    
    \mathcenter{
     \begin{tikzpicture}
      \node[inner sep=2pt] at (0,-.7) (b) {$V\otimes V$};
      \node[inner sep=2pt] at (0,.7) (t) {$V\otimes V$};
      \draw[->, thick] (b) to node[left, inner sep=.5pt]{$\Rv^{-1}$} (t);
    \end{tikzpicture}
  }$
  &
    $Q\left(
  \mathcenter{
    \begin{tikzpicture}
      \node[inner sep=1pt] at (-.3,-.5) (bl) {};
      \node[inner sep=1pt] at (.3,-.5) (br) {};
      \node[inner sep=1pt] at (-.3,.5) (tl) {};
      \node[inner sep=1pt] at (.3,.5) (tr) {};
      \node at (0,0) (center) {};
      \draw[->, thick, red] (br) to (tl);
      \draw[->, thick, red] (center) to (bl);
      \draw[-, thick, red] (center) to (tr);
    \end{tikzpicture}
  }
  \right) =    
    \mathcenter{
     \begin{tikzpicture}
      \node[inner sep=2pt] at (0,-.7) (b) {$V^*\otimes V$};
      \node[inner sep=2pt] at (0,.7) (t) {$V\otimes V^*$};
      \draw[->, thick] (b) to node[left, inner sep=.5pt]{$\Rv^{-1}$} (t);
    \end{tikzpicture}
  }$
&
    $Q\left(
  \mathcenter{
    \begin{tikzpicture}
      \node[inner sep=1pt] at (-.3,-.5) (bl) {};
      \node[inner sep=1pt] at (.3,-.5) (br) {};
      \node[inner sep=1pt] at (-.3,.5) (tl) {};
      \node[inner sep=1pt] at (.3,.5) (tr) {};
      \node at (0,0) (center) {};
      \draw[->, thick, red] (tl) to (br);
      \draw[-, thick, red] (bl) to (center);
      \draw[->, thick, red] (center) to (tr);
    \end{tikzpicture}
  }
  \right) =    
    \mathcenter{
     \begin{tikzpicture}
      \node[inner sep=2pt] at (0,-.7) (b) {$V\otimes V^*$};
      \node[inner sep=2pt] at (0,.7) (t) {$V^*\otimes V$};
      \draw[->, thick] (b) to node[left, inner sep=.5pt]{$\Rv^{-1}$} (t);
    \end{tikzpicture}
  }$
  &
    $Q\left(
  \mathcenter{
    \begin{tikzpicture}
      \node[inner sep=1pt] at (-.3,-.5) (bl) {};
      \node[inner sep=1pt] at (.3,-.5) (br) {};
      \node[inner sep=1pt] at (-.3,.5) (tl) {};
      \node[inner sep=1pt] at (.3,.5) (tr) {};
      \node at (0,0) (center) {};
      \draw[->, thick, red] (tl) to (br);
      \draw[->, thick, red] (center) to (bl);
      \draw[-, thick, red] (center) to (tr);
    \end{tikzpicture}
  }
  \right) =    
    \mathcenter{
     \begin{tikzpicture}
      \node[inner sep=2pt] at (0,-.7) (b) {$V^*\otimes V^*$};
      \node[inner sep=2pt] at (0,.7) (t) {$V^*\otimes V^*$};
      \draw[->, thick] (b) to node[left, inner sep=.5pt]{$\Rv^{-1}$} (t);
    \end{tikzpicture}
  }$\\
  & & & \\
  $Q\left(
  \mathcenter{
    \begin{tikzpicture}
      \node[inner sep=1pt] at (-.25,.5) (l) {};
      \node[inner sep=1pt] at (.25,.5) (r) {};
      \draw[->, bend left=90, looseness=3.5, thick, red] (r) to (l);
    \end{tikzpicture}
  }
  \right) =    
    \mathcenter{
     \begin{tikzpicture}
      \node[inner sep=2pt] at (0,-.7) (b) {$\C(q)$};
      \node[inner sep=2pt] at (0,.7) (t) {$V\otimes V^*$};
      \draw[->, thick] (b) to node[left, inner sep=1pt]{$\lcup$} (t);
    \end{tikzpicture}
  }$
& 
  $Q\left(
  \mathcenter{
    \begin{tikzpicture}
      \node[inner sep=1pt] at (-.25,.5) (l) {};
      \node[inner sep=1pt] at (.25,.5) (r) {};
      \draw[->, bend right=90, looseness=3.5, thick, red] (l) to (r);
    \end{tikzpicture}
  }
  \right) =    
    \mathcenter{
     \begin{tikzpicture}
      \node[inner sep=2pt] at (0,-.7) (b) {$\C(q)$};
      \node[inner sep=2pt] at (0,.7) (t) {$V^*\otimes V$};
      \draw[->, thick] (b) to node[left, inner sep=1pt]{$\rcup$} (t);
    \end{tikzpicture}
  }$
  &
    $Q\left(
  \mathcenter{
    \begin{tikzpicture}
      \node[inner sep=1pt] at (-.25,-.5) (l) {};
      \node[inner sep=1pt] at (.25,-.5) (r) {};
      \draw[->, bend right=90, looseness=3.5, thick, red] (r) to (l);
    \end{tikzpicture}
  }
  \right) =    
    \mathcenter{
     \begin{tikzpicture}
      \node[inner sep=2pt] at (0,-.7) (b) {$V^*\otimes V$};
      \node[inner sep=2pt] at (0,.7) (t) {$\C(q)$};
      \draw[->, thick] (b) to node[left, inner sep=1pt]{$\lcap$} (t);
    \end{tikzpicture}
  }$
   &
     $Q\left(
  \mathcenter{
    \begin{tikzpicture}
      \node[inner sep=1pt] at (-.25,-.5) (l) {};
      \node[inner sep=1pt] at (.25,-.5) (r) {};
      \draw[->, bend left=90, looseness=3.5, thick, red] (l) to (r);
    \end{tikzpicture}
  }
  \right) =    
    \mathcenter{
     \begin{tikzpicture}
      \node[inner sep=2pt] at (0,-.7) (b) {$V\otimes V^*$};
      \node[inner sep=2pt] at (0,.7) (t) {$\C(q)$};
      \draw[->, thick] (b) to node[left, inner sep=1pt]{$\rcap$} (t);
    \end{tikzpicture}
  }$
  \\
    & & & \\
  $Q\left(\hspace{2pt}
  \mathcenter{
    \begin{tikzpicture}
      \node[inner sep=1pt] at (0,-.5) (b) {};
      \node[inner sep=1pt] at (0,.5) (t) {};
      \draw[->, thick, red] (b) to (t);
    \end{tikzpicture}
  }
  \hspace{2pt}
  \right) =    
    \mathcenter{
     \begin{tikzpicture}
      \node at (0,-.7) (b) {$V$};
      \node at (0,.7) (t) {$V$};
      \draw[->, thick] (b) to node[left]{$\id$} (t);
    \end{tikzpicture}
  }$
& 
  $Q\left(
  \hspace{2pt}
  \mathcenter{
    \begin{tikzpicture}
      \node[inner sep=1pt]  at (0,-.5) (b) {};
      \node[inner sep=1pt]  at (0,.5) (t) {};
      \draw[->, thick, red] (t) to (b);
    \end{tikzpicture}
  }
  \hspace{2pt}
  \right) =    
    \mathcenter{
     \begin{tikzpicture}
      \node at (0,-.7) (b) {$V^*$};
      \node at (0,.7) (t) {$V^*$};
      \draw[->, thick] (b) to node[left]{$\id$} (t);
    \end{tikzpicture}
  }$ &
& 
  \end{tabular}

Applying this to a closed link, we get an element of the ground field $\C(q)$.  (The framing is irrelevant because with label $V$ on all strands, the ribbon element acts as $1$.)

\begin{figure}[h]
\centering
     \labellist
        \pinlabel \textcolor{red}{$L$} at 25 50
        \pinlabel \textcolor{red}{$T_L$} at 202 50
        \pinlabel $=$ at 115 50
     \endlabellist
	\includegraphics[scale=.75]{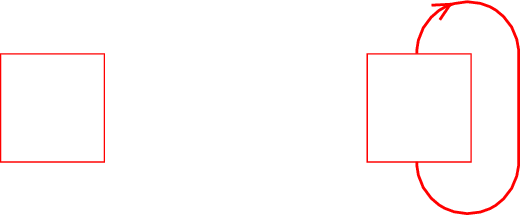}
      \vskip .2 cm
       \caption{Left: A link $L$. Right: The same link $L$ seen as the closure of a $(1,1)$-tangle $T_L$.}
       \label{fig:rt-alex}
\end{figure}

This invariant vanishes on any closed link diagram.  So to get an interesting invariant, we must modify this construction.  It turns out that if we cut, say, the topmost strand of a link $L$ at some horizontal coordinate and isotope it to get a $(1,1)$-tangle (see Figure~\ref{fig:rt-alex}), then the value of $Q$ on the resulting $(1,1)$-tangle $T_L$ is an invariant of the original link.  That is, this assignment is unchanged if we apply Reidemeister moves at the beginning and cut a different strand.  See \cite[Proposition 4.4, Proposition 4.5, Theorem 4.6]{s1} or \cite{v1}.
It is easy to show that
\begin{equation*}
\Rv=\Rv^{-1}+(q-q^{-1})\id,
\end{equation*}
which matches the skein relation
\begin{equation*}
\Delta\left(
  \mathcenter{
    \begin{tikzpicture}
      \node at (-.5,-.5) (bl) {};
      \node at (.5,-.5) (br) {};
      \node at (-.5,.5) (tl) {};
      \node at (.5,.5) (tr) {};
      \node at (0,0) (center) {};
      \draw[->, thick, red] (bl) to (tr);
      \draw[-, thick, red] (br) to (center);
      \draw[->, thick, red] (center) to (tl);
      \draw[dashed] (0,0) circle (.52);
    \end{tikzpicture}
  }
\right)-\Delta\left( 
\mathcenter{
    \begin{tikzpicture}
      \node at (-.5,-.5) (bl) {};
      \node at (.5,-.5) (br) {};
      \node at (-.5,.5) (tl) {};
      \node at (.5,.5) (tr) {};
      \node at (0,0) (center) {};
      \draw[->, thick, red] (br) to (tl);
      \draw[-, thick, red] (bl) to (center);
      \draw[->, thick, red] (center) to (tr);
      \draw[dashed] (0,0) circle (.52);
    \end{tikzpicture}
  }
\right)=(t^{1/2}-t^{-1/2})\Delta\left(
 \mathcenter{
    \begin{tikzpicture}
      \node at (-.5,-.5) (bl) {};
      \node at (.5,-.5) (br) {};
      \node at (-.5,.5) (tl) {};
      \node at (.5,.5) (tr) {};
      \draw[al, red] (bl) to (tl);
      \draw[ar, red] (br) to (tr);
      \draw[dashed] (0,0) circle (.52);
    \end{tikzpicture}
  }
\right)
\end{equation*}
for the Alexander-Conway polynomial if we let $q^2=t$.  Since $Q(T_-)$ sends the unknot to $\id_V$ and $\Delta$ sends the unknot to $1$, it follows that 
\begin{equation*}
Q(T_L) = \Delta(L)\id_V.
\end{equation*}

Note that since there is no monoidal structure on tangle Floer homology, here we think of $Q$ as the regular functor coming from the maps  $\Rv_i$, $\lcap_i$, $\rcap_i$, $\lcup_i$, and $\rcup_i$.




\section{Tangle Floer homology} 
\label{sec:tangle}


In this section, we give a brief review of Tangle Floer homology; a similar summary can be found in \cite{pv2}.
 An  \emph{ $(m,n)$-tangle} $\T$ is a proper, smoothly embedded oriented 1--manifold in  $[t_0,t_1]\times \R^2$, with boundary $\bdy \T = \bdy^0\T\sqcup \bdy^1\T$, where $\bdy^0\T = \{t_0\}\times\{\frac 1 2, \ldots, m - \frac 1 2\}\times \{0\}$ and $\bdy^1\T = \{t_1\}\times\{\frac 1 2, \ldots, n - \frac 1 2\}\times \{0\}$, treated as oriented sequences of points.  A planar diagram of a tangle is a projection to $[t_0,t_1]\times\R\times\{0\}$ with no triple intersections, self-tangencies, or cusps, and with over- and under-crossing data preserved (as viewed from the positive $z$ direction).  The boundaries of $\T$ can be thought of as \emph{sign sequences}
\begin{equation*}
-\bdy^0 \T \in \{\pm1\}^m, \bdy^1\T \in \{\pm1\}^n,
\end{equation*}
according to the orientation of the tangle at each point ($+$ if the tangle is oriented left-to-right, $-$ if the tangle is oriented right-to-left).

\begin{figure}[h]
\centering
 \includegraphics[scale=1.2]{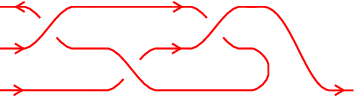} 
      \vskip .2 cm
       \caption{A projection of a $(3,1)$-tangle $\T$ to $I\times \R$. Here $-\partial^0\T=(+,+, -)$ and $\partial^1\T=(+)$.}\label{fig:tangle}
\end{figure}

In \cite{pv}, the last two authors defined: 
\begin{itemize}
  \item a dg algebra $\AA(\PP)$ for any sign sequence $\PP\in \{\pm1\}^n$ and
  \item a type $\tDA$ bimodule $\ctt(\T)$ over $(\AA(-\bdy^0\T)$, $\AA(\bdy^1\T))$ (defined up to homotopy equivalence and tensor-factors of $\Ft\oplus\Ft[1]\{2\}$).
\end{itemize}

In this section, we will give a combinatorial description of $\AA(\PP)$ and $\ctt(\T)$.  We start by describing $\AA(\PP)$ and a certain subalgebra $\II(\PP)$ in Subsection \ref{ssec:alg}.  The subalgebra $\II(\PP)$ will serve as our ground ring.  After a digression in Subsection \ref{ssec:dec} on elementary decompositions of tangles, we describe in Subsection \ref{ssec:gen} a homogeneous $\Ft$-basis for $\ctt(\T)$.  In Subsection \ref{ssec:DAmodule}, we endow $\ctt(\T)$ with the  structure of a $\tDA$ bimodule.  Subsection \ref{ssec:CTthms} is a summary of the main results from \cite{pv} about invariance, pairing, and the relation to knot Floer homology.  Lastly, Subsection \ref{ssec:dictionary} compares the notation used here with that of \cite{pv}.

\subsection{The dg algebra of a sign sequence}\label{ssec:alg}

Let $\PP\in\{\pm1\}^n$ be a sign sequence and let $[n]=\{0,1,\ldots, n\}$.

\begin{definition} A \emph{generator} associated to $\PP$ is a partial bijection $[n]\to[n]$ (that is, a bijection from a subset $\sss\subseteq[n]$ to a possibly different subset $\ttt\subseteq[n]$).  If $x$ is a generator which is the identity function on some subset of $[n]$, we say $x$ is an \emph{idempotent generator}.  The \emph{weight} $|x|$ of a generator $x$ is the number of elements in the subset on which it is defined.\end{definition}

Let $x:[n]\to[n]$ be a generator for $\PP$ with underlying bijection $\sss\to\ttt$.  The diagram associated to $x$ is drawn in $[0,1]\times[0,n]$ as follows:
\begin{itemize}
  \item draw red strands from $(0,i-\frac{1}{2})$ to $(1,i-\frac{1}{2})$ for $i=1,2,\ldots,n$;
  \item orient the red strands according to the sign sequence $\PP$ (right for $+$, left for $-$);
  \item draw black dots at the $2(n+1)$ points $\{0,1\}\times[n]$;
  \item draw a black strand from $j$ to $x(j)$ for each $j\in\sss$, such that:
  \begin{itemize}
    \item black strands have no critical points with respect to the horizontal coordinate (``don't turn back''),
    \item there are no triple intersection points among red and/or black strands, and
    \item there are a minimal number of intersection points between strands (subject to the above conditions).
  \end{itemize}
\end{itemize}
Up to the evident notion of equivalence fixing boundaries (allowing ambient isotopies and Reidemeister III moves among black and red strands), there is exactly one such diagram for each generator $x$.

\begin{figure}[h]
\centering
 \includegraphics[scale=.8]{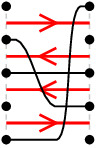} 
      \vskip .2 cm
       \caption{The diagram of the generator of $\AA(+,-,-,+)$ of weight $3$ which sends $0\mapsto4$, $2\mapsto2$, and $3\mapsto1$.}\label{fig:alggen}
\end{figure}

\begin{definition} Let $\AA(\PP)$ be the $\Ft$-span of all generators for $\PP$, and let $\II(\PP)$ be the subspace of all idempotent generators for $\PP$. Below, we will give both of these the structure of bigraded algebras and $\AA(\PP)$ the structure of a dg algebra over $\II(\PP)$.\end{definition}

We define the \emph{Alexander grading} $A(x)$ and \emph{Maslov grading} $M(x)$ of a generator $x$ as follows:
\begin{eqnarray*}
2A(x) &=& \diagup \hspace{-.375cm}{\color{red}{\nwarrow}}(x) + \diagdown \hspace{-.375cm}{\color{red}{\swarrow}}(x)
-\diagup \hspace{-.37cm}{\color{red}{\searrow}}(x) - \diagdown \hspace{-.37cm}{\color{red}{\nearrow}}(x),\\
M(x) & =& \diagup \hspace{-.35cm}\diagdown(x) - \diagup \hspace{-.37cm}{\color{red}{\searrow}}(x) -\diagdown \hspace{-.37cm}{\color{red}{\nearrow}}(x).
\end{eqnarray*}
Here, $\diagup \hspace{-.375cm}{\color{red}{\nwarrow}}(x)$ means the number of crossings between a left-oriented red strand and a black strand passing from below-left to above-right with respect to that red strand; and analogously for the other terms.  The Maslov grading will be the homological grading and $2$ times the Alexander grading will be the internal grading.

Let $x$ and $y$ be two generators for $\PP$ with underlying bijections $\sss_1\to\ttt_1$ and $\sss_2\to\ttt_2$ respectively.  If $\ttt_1\neq\sss_2$, define their product in $\AA(\PP)$ to be $0$.  If $\ttt_1=\sss_2$, consider the diagram obtained from diagrams for $x$ and $y$ by concatenating them with $x$ on the left, $y$ on the right.  If the resulting diagram has a minimal number of crossings, then define $xy=y\circ x$, so that the diagram of the product $xy$ is obtained from the concatenated diagram by horizontal scaling by $\frac{1}{2}$.  If not, then the product is defined to be $0$.  See Figure \ref{fig:multiplication} for examples.

\begin{figure}[h]
\centering
   \includegraphics[scale=0.8]{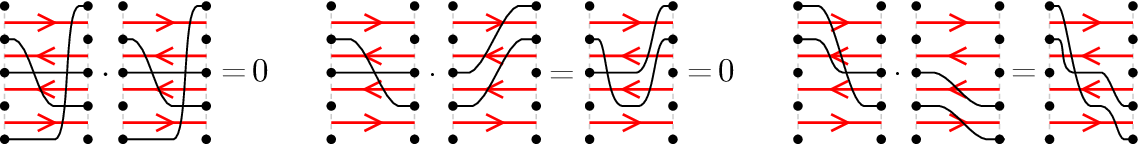}
         \vskip .2 cm
         \caption{Examples of the multiplication. Left: the two diagrams cannot be concatenated. Middle: the concatenation does not have a minimal number of crossings. Right: the product is the concatenation followed by horizontal scaling by $\frac{1}{2}$.}\label{fig:multiplication}
\end{figure}

It is clear that $\II(\PP)\cong\Ft^{\oplus2^{n+1}}$, so $\II(\PP)$ is a suitable ground ring of the form discussed in Section \ref{sec:preliminaries}.

Given a black-black crossing in a diagram for a generator $x$, there is (up to the evident equivalence) a unique picture in which that crossing is locally resolved, as in Figure \ref{fig:resolution}.  The differential on $\AA(\PP)$ sends a generator $x$ to the sum of all generators $y$ whose diagrams can be obtained from that of $x$ by resolving one crossing \emph{in such a way that the result has a minimal number of crossings}.  Figure \ref{fig:diffalg} computes the differential of the generator of Figure \ref{fig:alggen}.

\begin{figure}[h]
\centering
\includegraphics[scale = .8]{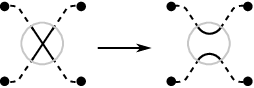} 
         \caption{The resolution of a crossing. The diagram is unchanged outside the grey circle.}\label{fig:resolution}
\end{figure}

\begin{figure}[h]
\centering
  \labellist
         \pinlabel  $\partial$ at 64 57
       \endlabellist
\includegraphics[scale = .8]{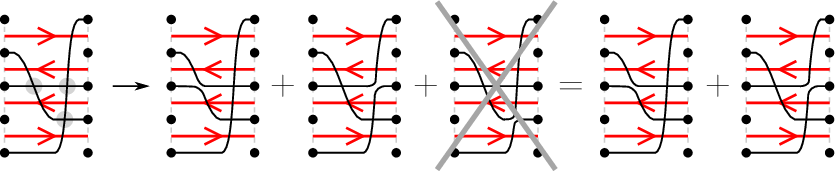} 
         \caption{The differential of a generator. The three diagrams after the arrow are those obtained from the leftmost diagram by resolving a crossing. The last one does not have minimal intersection.}\label{fig:diffalg}
\end{figure}

The following lemma is an easy consequence of the definitions:
\begin{lemma}[\cite{pv}, Theorem 3.9]
 $\AA(\PP)$ is a differential graded algebra over the graded algebra $\II(\PP)$ with respect to the Maslov grading.  The Alexander grading is preserved by the differential and the multiplication.  The primitive idempotents are precisely the idempotent generators.\qed
\end{lemma}

For a subset $\sss\subseteq[n]$, we write $\ee{\sss}$ for the corresponding primitive idempotent.  If $x:[n]\to[n]$ is a generator defined on $\sss$ with $\ttt=x(\sss)$, let
\begin{equation*}
\sssA{0}(x)=\sss, \quad \sssD{0}=[n]\setminus\sss, \quad \sssA{1}(x)=\ttt, \quad \sssD{1}(x)=[n]\setminus\ttt.
\end{equation*}
For short, we define
\begin{equation*}
\eeA{i}(x) = \ee{\sssA{i}(x)}, \quad \eeD{i}(x) = \ee{\sssD{i}(x)}
\end{equation*}
for $i=0,1$, so that $x$ is an element of the idempotented piece $\eeA{0}(x)\AA(\PP)\eeA{1}(x)$.

Note that
\begin{equation*}
\AA(\PP) = \bigoplus_{\ell=0}^{n} \AA_\ell(\PP)
\end{equation*}
as dg algebras, where $\AA_\ell(\PP)$ is the subspace spanned by all generators of weight $\ell$.

\subsection{Decompositions of tangles}\label{ssec:dec} 

The notion of decomposing a tangle into elementary pieces is standard, and we don't make use of any non-standard results in this direction.  For convenience, we will insist on arranging these rather specifically, as we presently explain.

\begin{definition}\label{def:elemtangle}
An \emph{elementary tangle} is a tangle of one of the following five types:
\renewcommand{\theenumi}{\roman{enumi}}
\begin{enumerate}
\item\label{it:triv} an $(n,n)$-tangle consisting of only straight strands is \emph{trivial};
\item\label{it:cap} an $(n+1,n-1)$-tangle consisting of a single cup and straight strands is a \emph{cup};
\item\label{it:cup} an $(n-1,n+1)$-tangle consisting of a single cap and straight strands is a \emph{cap};
\item\label{it:pos} an $(n,n)$-tangle consisting of straight strands and a single crossing where the strand with the higher slope is over the strand with the lower slope is an \emph{\texttt{e}-crossing};
\item\label{it:neg} and an $(n,n)$-tangle consisting of straight strands and a single crossing where the strand with the higher slope is under the strand with the lower slope is an \emph{\texttt{o}-crossing}. 
\end{enumerate}
In all of the above cases we assume $n>0$. 
\end{definition}

The choice of ``cup" an ``cap" terminology is made to match our view (for this paper) of tangles as running right-to-left, opposite to the indexing in the decompositions below. See for example the direction of the map $[\ctt]$ in Figure~\ref{fig:RTtangle}.

Any $(m,n)$-tangle diagram $\T$ admits a decomposition $\Tdec=(\T_1,\ldots,\T_k)$ such that $\T=\T_1\circ\cdots\circ \T_k$ with each $\T_i$ an elementary $(n_{i-1},n_i)$-tangle. 
By an ``inactive strand" of an elementary tangle we mean any strand other than one of the two crossing strands in a crossing, or the strand which forms the semi-circle part in a cap or a cup.

After possibly performing ambient isotopy rel boundary, we can and will always assume the following:
\begin{itemize}
  \item $\T_i\subset [i-1, i]\times \R$, $\bdy^0\T_i = \{i-1\} \times \{\frac 1 2, \ldots, n_{i-1} - \frac 1 2\}$, $\bdy^1\T_i = \{i\} \times  \{\frac 1 2, \ldots, n_{i} - \frac 1 2\}$ for $i=1,2,\ldots,k$;
  \end{itemize}
We will refer to the pieces of the diagram lying above intervals of form $[i, i+\frac 1 2]$ as the \emph{odd pieces/halves}, and the pieces lying above intervals of form $[i+\frac 1 2, i+1]$ as the \emph{even pieces/halves}. We further assume:
\begin{itemize}
  \item if $\T_i$ is an \texttt{e}-crossing (respectively  an \texttt{o}-crossing), then all inactive strands are horizontal, and the two strands that cross do so in the even piece (respectively odd piece), and are horizontal in the other piece, whence the terms \texttt{e}-crossing and \texttt{o}-crossing;
  \item if $\T_i$ is a cup (respectively cap), then the curved strand is a  left-opening (respectively right-opening) semi-circle of radius $\frac{1}{2}$ centered at $\{i-1\}\times\{r\}$ (respectively $\{i\}\times \{r\}$) for some non-negative integer $r$. In particular, cups occur in odd halves, caps occur in even halves. Strands below the ``curved" strand are horizontal; strands above the curved strand are horizontal in the odd (respectively even) half, then connect to points $2$ lower on the right (respectively left) boundary of $\T_i$.
\end{itemize}
A cap or a cup whose semi-circle is centered at $\{i\}\times\{r\}$ is said to be ``at height $r$.'' Note that our cup/cap terminology is consistent with ``reading tangles right-to-left''. 

Figure \ref{fig:tdec} is an example of a decomposition of the form we will consider.

\begin{figure}[h]
\centering
    \includegraphics[scale=.8]{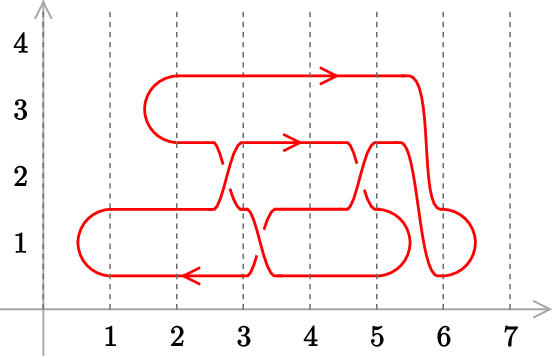} 
      \vskip .2 cm
       \caption{The trefoil, decomposed as $\Tdec = (\T_1, \ldots, \T_7)$. The elementary tangles $\T_1$ and $\T_2$ are examples of caps, $\T_3$ and $\T_5$ are \texttt{e}-crossings, $\T_4$ is an \texttt{o}-crossing, $\T_6$ and $\T_7$ are cups. }\label{fig:tdec}
\end{figure}

In the literature, trivial tangles are not usually considered to be elementary.  For us, it is convenient to consider them to be.

\subsection{Generators associated to tangles} \label{ssec:gen}

Let $\T$ be a $(m,n)$-tangle and fix a decomposition $\Tdec=(\T_1,\T_2,\dots,\T_k)$ of the form described in Subsection \ref{ssec:dec}; $\T_i$ is a $(n_{i-1},n_i)$-tangle.  To such a decomposition, we will associate a type $\tDA$ bimodule $\ctt(\Tdec)$ over $(\AA(-\bdy^0\T), \AA(\bdy^1\T))$.  In this subsection, we will describe $\ctt(\Tdec)$ as a bigraded bimodule over $(\II(-\bdy^0\T), \II(\bdy^1\T))$, leaving the rest of the structure  for the following subsections. 

For each $i=1,2,\ldots,k$, let $N_i=\max(n_{i-1},n_i)$, and define a subset $B_i\subseteq[N_i]$ by
\begin{equation*}
B_i = \begin{cases}
[N_i] & \text{if }\T_i\text{ is trivial or a crossing}, \\
[N_i] \setminus \{r\} & \text{if }\T_i\text{ is a cap or cup at height }r. \\
\end{cases}
\end{equation*}
For $i=0,1,\ldots,k$, let
\begin{equation*}
V_i = \{i\} \times [n_i],
\end{equation*}
and for $i=1,2,\ldots,k$, let
\begin{equation*}
W_i = \{i-\frac{1}{2}\} \times B_i.
\end{equation*}

In Figures \ref{fig:tdecVW} and \ref{fig:trefdecVW}, the dark red dots are the sets $V_i$ and the dark blue dots are the sets $W_i$. The vertical lines through the sets $V_i$ and $W_j$ cut the diagram for $\Tdec$  into its even and odd pieces. 
\begin{figure}[h]
\centering
    \includegraphics[scale=.8]{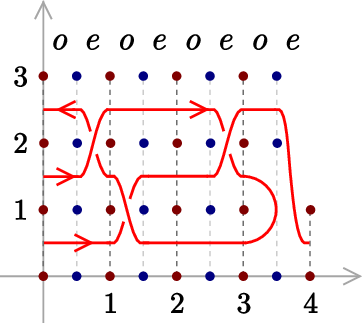} 
      \vskip .2 cm
       \caption{The tangle from Figure \ref{fig:tangle} decomposed as $\Tdec = (\T_1,\T_2,\T_3,\T_4)$, along with the sets of points $V_0, \ldots, V_4$ (in dark red) and $W_1, \ldots, W_4$ (in dark blue). The even and the odd  pieces are marked with an $e$ and an $o$, respectively. The elementary tangles $\T_1$ and $\T_3$ are examples of an \texttt{e}-crossing, $\T_2$ is an \texttt{o}-crossing, and $\T_4$ is a cup.}\label{fig:tdecVW}
\end{figure}
\begin{figure}[h]
\centering
    \includegraphics[scale=.8]{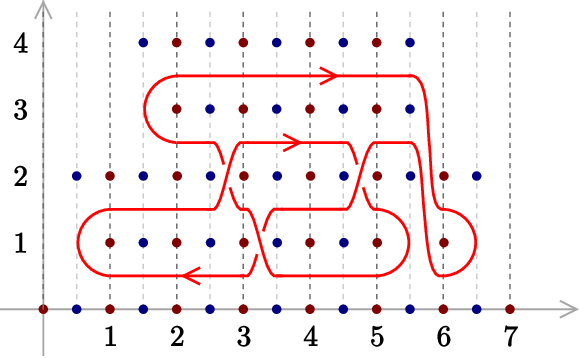} 
      \vskip .2 cm
       \caption{The trefoil decomposition from Figure \ref{fig:tdec}, along with with the sets of points $V_0, \ldots, V_7$ (in dark red) and $W_1, \ldots, W_7$ (in dark blue).}\label{fig:trefdecVW}
\end{figure}
Since \texttt{e}-crossings only occur inside even pieces, and \texttt{o}-crossings only occur inside odd pieces, we will work with singular projections of tangles from now on and not necessarily draw diagrams to scale. See, for example,  Figure \ref{fig:tangen}.

\begin{definition} A \emph{generator} associated to the decomposition $\Tdec$ of a tangle $\T$ is a choice of partial bijections $V_{i-1}\to W_i$ and $W_i\to V_i$ for each $i=1,2,\ldots,k$ such that every point of each $V_i$ ($i=1,2,\ldots,k-1$) and each $W_i$ ($i=1,2,\ldots,k$) is either in the domain or the range of one of the chosen partial bijections, but not both.  Note that there is no restriction on the points of $V_0$ and $V_k$.\end{definition}

We define diagrams of generators associated to $\Tdec$ in analogy with our definition of those of algebra generators (no triple intersections, minimal number of intersections, etc.).  In such a diagram, each point of $V_i$ or $W_i$ (except for $V_0$ and $V_k$) is the boundary of exactly one black strand. 

\begin{definition}\label{defn-cthat-vs} Let $\ctt(\Tdec)$ be the $\Ft$-vector space spanned by all generators for $\Tdec$.  Below in this subsection, we will give $\ctt(\Tdec)$ the structure of a bigraded bimodule over $(\II(-\bdy^0\T), \II(\bdy^1\T))$.  In the following subsection, we will give it the structure of a type $\tDA$ bimodule over $(\AA(-\bdy^0\T), \AA(\bdy^1\T))$.\end{definition}

If $x\in\ctt(\Tdec)$ is a generator, define $\sssA{i}(x)$, $\sssD{i}(x)$, $\eeA{i}(x)$, and $\eeD{i}(x)$ for $i=0,1$ in analogy with the corresponding definitions for algebra generators in Subsection \ref{ssec:alg}.  For example, if the $f:V_0\to W_1$ part of $x$ has domain $\ttt$, then $\sssA{0}(x)=\ttt$ and $\sssD{0}(x)=[n_0]\setminus\ttt$.  The $(\II(-\bdy^0\T), \II(\bdy^1\T))$ bimodule structure on $\ctt(\Tdec)$ is defined on generators $x$ as follows:
\begin{equation*}
\ee{\sss} x = \begin{cases}
x & \sssD{0}(x) =  \sss, \\
0 & \text{otherwise},
\end{cases} \qquad
x \ee{\sss} = \begin{cases}
x & \sssA{1}(x) = \sss, \\
0 & \text{otherwise}.
\end{cases}
\end{equation*}

Figure \ref{fig:tangen} depicts the diagram of a generator for the tangle decomposition of Figure \ref{fig:tdecVW}.
\begin{figure}[h]
\centering
    \includegraphics[scale=.8]{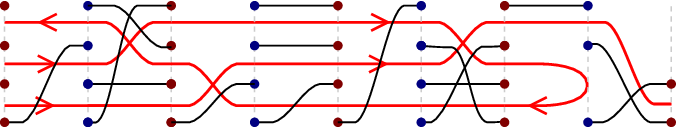} 
      \vskip .2 cm
       \caption{A diagram for a generator associated to the decomposition $\Tdec$ from Figure \ref{fig:tdecVW}.}\label{fig:tangen}
\end{figure}

The degree of a generator $x$ is a sum of the degrees of its constituent partial bijections, which are defined in terms of diagrams as follows.  The Alexander grading of a partial bijection diagram $f$ is defined by
\[
2A(f) =  \diagup \hspace{-.375cm}{\color{red}{\nwarrow}}(f) + \diagdown \hspace{-.375cm}{\color{red}{\swarrow}}(f)
-\diagup \hspace{-.37cm}{\color{red}{\searrow}}(f) - \diagdown \hspace{-.37cm}{\color{red}{\nearrow}}(f)
+{\color{red}{\searrow}} \hspace{-.38cm}{\color{red}{\nearrow}}(f)
-{\color{red}{\nwarrow}} \hspace{-.38cm}{\color{red}{\swarrow}}(f)
-{\color{red}{\leftarrow}}(f).
\]
All but the last term on the right-hand side are counts of the number of occurrences of the relevant crossing type in the diagram.  The last term counts the total number of left-oriented red strands in the diagram.  In the semi-circle part of a cap or a cup, we consider there to be one left- and one right-oriented strand.

The Maslov grading is defined differently on the partial bijections $V_{i-1}\to W_i$ and on those $W_i\to V_i$.  On the former,
\[
M(f) = -\diagup \hspace{-.34cm}\diagdown(f) + \diagup \hspace{-.375cm}{\color{red}{\nwarrow}}(f) + \diagdown \hspace{-.375cm}{\color{red}{\swarrow}}(f)
-{\color{red}{\swarrow}} \hspace{-.38cm}{\color{red}{\nwarrow}}(f)
-{\color{red}{\leftarrow}}(f),\]
while on the latter,
\[
M(f) = \diagup \hspace{-.34cm}\diagdown(f) - \diagup \hspace{-.37cm}{\color{red}{\searrow}}(f) -\diagdown \hspace{-.37cm}{\color{red}{\nearrow}}(f)
+{\color{red}{\searrow}} \hspace{-.38cm}{\color{red}{\nearrow}}(f).
\]
As with the definition of $\AA(\PP)$, $2$ times the Alexander grading will be the internal grading, and the Maslov grading will be the homological grading.

\begin{example*} \emph{The Alexander grading of the generator in Figure \ref{fig:tangen} is $-5$, and the Maslov grading of the same generator is $-1$.}
\end{example*}

\subsection{Bimodules associated to tangles}\label{ssec:DAmodule}
Let  $\Tdec$ be a decomposition of a tangle $\T$. In this subsection we endow $\ctt(\Tdec)$ with the structure of a $\tDA$ bimodule.

We define an algebra action 
\begin{equation*}\begin{split}
\ctt(\Tdec)\otimes A(\bdy^1\T) & \to \ctt(\Tdec) \\
x \otimes a & \mapsto x \cdot a.
\end{split}\end{equation*}
similarly to how we defined the multiplication on the algebras. For generators $x\in \ctt(\Tdec)$ and $a\in \AA(\bdy^1\T)$, if $\sssA{1}(x) \neq \sssA{0}(a)$, define the product to be $x \cdot a = 0$. If $\sssA{1}(x) = \sssA{0}(a)$ consider the diagram obtained from diagrams for $x$ and $a$ by concatenating them with $x$ on the left and $a$ on the right. If the resulting diagram has a minimal number of crossings, define $x \cdot a = a\circ x$, where the generator $a\circ x$ is obtained from $x$ by composing the rightmost partial bijection for $x$ with the partial bijection $a$,  so that the diagram for $x \cdot a$ is obtained from the concatenated diagram by horizontal scaling of the rightmost piece for $x$ along with the diagram for $a$ by $\frac 1 2$. If not, then define $x \cdot a = 0$. 

We sometimes write $xa$ for $x \cdot a$. Note that whether $xa$ vanishes only depends on $a$ and on the rightmost piece for $x$. See Figure \ref{fig:typeA} for an example where $\Tdec = (\T_2, \T_3)$ for the elementary tangles $\T_2$ and $\T_3$ from Figure \ref{fig:tdecVW}.

\begin{figure}[h]
\centering
\includegraphics[scale = .8]{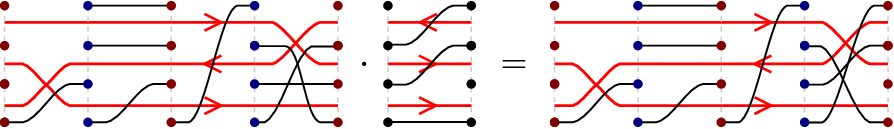} 
\vskip .2cm
         \caption{A non-zero product. The diagram for the result is the concatenation of  diagrams for the input, followed by the appropriate scaling.}\label{fig:typeA}
\end{figure}

Next, we define a map 
\[\bdy:\ctt(\Tdec)\to \ctt(\Tdec)\]
as the sum $\bdy_+ + \bdy_- + \bdy_m$, with the three components defined on generators below.

The map $\bdy_+$ modifies even pieces analogously to the differential on the algebra---it sends a generator $x$  to the sum of all generators $y$ whose diagram can be obtained from that of $x$ by resolving one black-black crossing in an even piece in such a way that the result has a minimal number of crossings. 

The map $\bdy_-$ modifies odd pieces in a ``dual'' way---it sends a generator $x$ to the sum of all generators $y$ whose diagram can be obtained from that of $x$ by \emph{introducing} a crossing between black-black strands that do not cross, in such a way that the total intersection number (taking into account both black and red strands) increases by one. This means we can introduce a crossing exactly when the change can be made local, as in Figure \ref{fig:resolution_dual}.
 \begin{figure}[h]
\centering
\includegraphics[scale = .8]{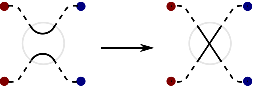} 
         \caption{Introducing a crossing. The diagram is unchanged outside the grey circle.}\label{fig:resolution_dual}
\end{figure}
Observe that if we think of a map on $\ctt(\Tdec)$ which sends a generator $x$ to a sum of generators $\sum y_i$ graphically by drawing an arrow from $x$ to each $y_i$, then the graph representing  $\bdy_-$ is obtained from the graph for $\bdy_+$ by changing the direction of each arrow.

The map $\partial_m$ modifies pairs of adjacent pieces of $\Tdec$. It sends a generator $x$ to the sum of all generators $y$ whose diagram can be obtained from that of $x$ by \emph{exchanging endpoints of black strands} as follows. Given a pair of endpoints $p$ and $q$  that lie on one and the same vertical line, we allow the two respective black strands to exchange these endpoints if:
\begin{itemize}
\item  both strands are in an odd piece, they cross, and each black or red strand that ends between $p$ and $q$ is on the odd side and crosses both of the given strands;
\item  both strands are in an even piece, they don't cross, and all black or red strands that end between $p$ and $q$ are on the even side, and do not cross either of the given strands; or
\item the two strands are in two adjacent pieces, and the exchange does not gain new crossings with black or red strands on the odd side, or lose such crossings on the even side. 
\end{itemize}
See Figure \ref{fig:dmix}. 

\begin{figure}[h]
\centering
\includegraphics[scale = .65]{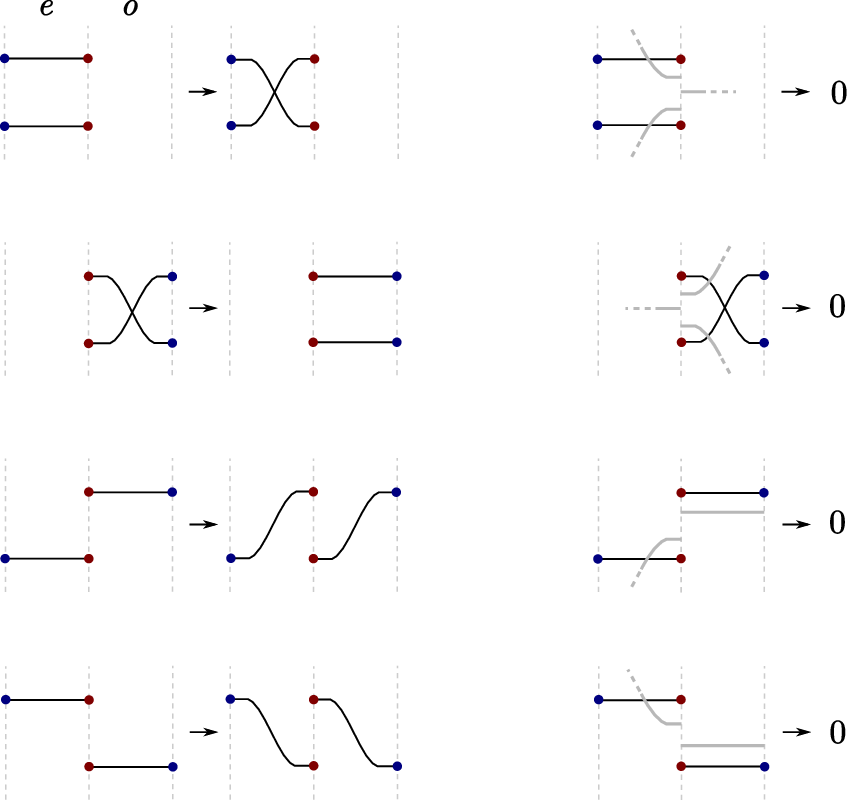} 
\vskip .2cm
         \caption{The map $\partial_m$ along a vertical line with an even piece to the left and an odd piece to the right. For the opposite case, one can simply reflect all diagrams about a vertical line and swap the colors of the dots. 
Left: The map $\bdy_m$ counts four types of exchanges (along the vertical line separating two adjacent pieces) of pairs of black strands, subject to a condition drawn to the right. Right: For a fixed pair of black strands, if there is a  black or a red strand in the same relative position to the pair as one of the displayed grey strands, we do not exchange. 
        }\label{fig:dmix}
\end{figure}

Last, we define a map 
\[\delta^L:\ctt(\Tdec)\to\AA(-\bdy^0\T)\otimes \ctt(\Tdec)\]
 as follows. 
 First, we represent a non-zero generator $a\otimes x$ for $\AA(-\bdy^0\T)\otimes \ctt(\Tdec)$ diagrammatically by gluing a diagram for $a$ to the left of a diagram for $x$. Left (resp.~right) multiplication is given by concatenating with the leftmost (resp.~rightmost) piece of the diagram for $a\otimes x$ and scaling horizontally; the differential is given as the sum $\bdy_+ + \bdy_- + \bdy_m$ of maps defined as above, where we treat the leftmost piece (the diagram for $a$) as even.
 For a generator $x$, $\delta^L(x)$ is the sum of all generators  $a\otimes y$ whose diagram can be obtained by  gluing a diagram for $\eeD{0}(x)$ (treated as an even diagram) to the left of a diagram for $x$, and exchanging two endpoints of black strands that lie on the gluing line, as in the definition of $\bdy_m$. See Figure \ref{fig:deltaL} for an example.

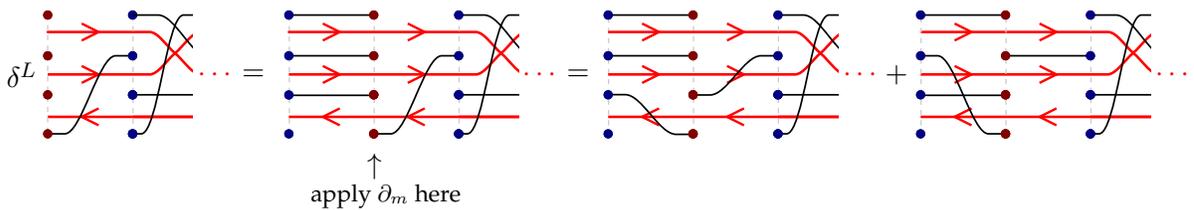
\begin{figure}[h]
\centering
  \labellist
    	 \pinlabel $\delta^L$ at -10 30
    	 \pinlabel  \textcolor{red}{$\cdots$} at 82 30
    	 \pinlabel  $=$ at 99 30
    	 \pinlabel  \footnotesize{$\textrm{ apply }\partial_m\textrm{ here}$} at 160 -27
   	 \pinlabel  $\uparrow$ at 156 -13
   	 \pinlabel  \textcolor{red}{$\cdots$} at 235 30
   	 \pinlabel  $=$ at 252 30
  	 \pinlabel  \textcolor{red}{$\cdots$} at 386 30
   	 \pinlabel  $+$ at 402 30
   	 \pinlabel  \textcolor{red}{$\cdots$} at 533 30
   \endlabellist
     \hspace{.2cm}
    \includegraphics[scale=.8]{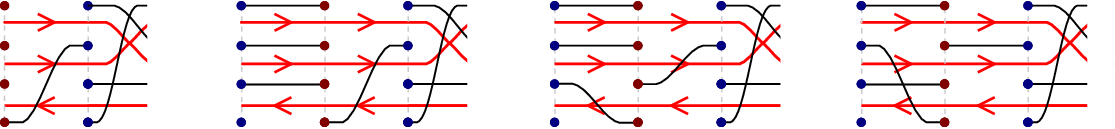} 
      \vskip .9 cm
       \caption{The map $\delta^L$ applied to the generator from Figure \ref{fig:tangen}.}\label{fig:deltaL}
\end{figure}

\begin{definition}
Given a tangle decomposition $\Tdec$, we give the $(\II(-\partial^0\T), \II(\partial^1\T))$ bimodule of Definition \ref{defn-cthat-vs} the structure of a type $\tDA$ bimodule over $(\AA(-\partial^0\T), \AA(\partial^1\T))$ using the following structure maps: let $\delta_i^1=0$ for $i>2$, define
 \[\delta_1^1:\ctt(\Tdec)\to \AA(-\partial^0\T)\otimes \ctt(\Tdec)\]
on generators by
\[\delta_1^1(x) = \eeD{0}(x)\otimes \bdy (x) + \delta^L(x),\]
and define
\[\delta_2^1:\ctt(\Tdec)\otimes \AA(\partial^1\T) \to \AA(-\partial^0\T)\otimes \ctt(\Tdec)\]
on generators by
\[\delta_2^1(x\otimes a) = \eeD{0}(x)\otimes xa.\]
\end{definition}
By Section 5.2. of \cite{pv}, the above structure is indeed a type $\tDA$ bimodule. 
Define the \emph{weight} of a generator $x$ of $\ctt(\Tdec)$ just as we did for a generator of $\AA(\PP)$, so that $|x|=|\eeA{1}(x)|$. The bimodule $\ctt(\Tdec)$ decomposes as a $\II(-\bdy^0\T)$, $\II(\bdy^1\T)$ bimodule as
$$\ctt(\Tdec) = \bigoplus_{k=0}^{\max\{|\bdy^0\T|, |\bdy^1\T|\}+1}\ctt_k(\Tdec),$$
where $\ctt_k(\Tdec)$ is the subspace of elements of weight $k$.  

Note that by forgetting the ``right map'' $\cdot$, the ``left map'' $\bdy^L$, or both,  we can view $\ctt(\Tdec)$ as a left type $D$ structure with $\delta^1=\delta_1^1$,  a right type $A$ structure  with structure maps $m_1=\bdy$, $m_2=\cdot$, and $m_i=0$ for $i>2$, or a chain complex with differential $\bdy$. 

For any of the above four types of structures, the Maslov grading $M$ is the homological grading, and $2$ times the Alexander grading $A$ is the internal grading.


\subsection{Invariance, pairing, and relation to knot Floer homology}\label{ssec:CTthms}

To state precise theorems below, we define the \emph{size} $|\Tdec|$ of a tangle decomposition $\Tdec$ as one half of the sum ${\color{red}{\leftarrow}}(p) + {\color{red}{\rightarrow}}(p)$, taken over all even and odd pieces $p$ of the decomposition.

Up to a factor that depends on the size of the chosen decomposition, tangle Floer homology is an invariant of the topological tangle $\T$.

\begin{theorem}[Invariance, {cf.~\cite[Theorem~5.3]{pv}}]\label{thm-invariance}
If $\Tdec'$ and $\Tdec''$ are two decompositions for a tangle $\T$ with $|\Tdec'|\geq |\Tdec''|$, then there is a  bigraded $\tDA$ homotopy equivalence
\[\ctt(\Tdec') \simeq \ctt(\Tdec'')\otimes (\Ft\oplus (\Ft[1]\{2\}))^{\otimes(|\Tdec'|-|\Tdec''|)}\]
\end{theorem}

Further, $\ctt$ is functorial under composition (this follows automatically from the definition of $\ctt$; see also \cite[Definition 5.2]{pv}).

\begin{theorem}[Pairing]\label{thm-pairing}
If $\Tdec_1$ and $\Tdec_2$ are decompositions for tangles $\T_1$ and $\T_2$ with $\bdy^1\T_1=-\bdy^0 \T_2$,  then $\ctt(\Tdec)$ is bigraded isomorphic to $\ctt(\Tdec_1)\boxtimes \ctt(\Tdec_2)$ as type $\tDA$ structures, where $\Tdec$ is the decomposition for $\T_1\circ \T_2$ that is the concatenation of the two decompositions $\Tdec_1$ and $\Tdec_2$.
\end{theorem}

In the case of a closed link, tangle Floer homology agrees with knot Floer homology $\HFKh(\K)$, the bigraded invariant of knots and links defined, in its various versions, in \cite{hfk, jrth, oszcomb, oszlink}. As defined here, given a link decomposition $\Tdec$, $\ctt(\Tdec)\cong \ctt_0(\Tdec)\oplus \ctt_1(\Tdec)$. The two summands $\ctt_0(\Tdec)$ and $\ctt_1(\Tdec)$ are bigraded homotopy equivalent  (for experts: each of $\ctt_0(\Tdec)$ and $\ctt_1(\Tdec)$ corresponds to a certain  Heegaard diagram for the link for $\Tdec$), and the relation to knot Floer homology can be stated in terms of either. 

\begin{theorem}[{cf.~\cite[Theorem 6.1, item (4) of Corollary 12.5]{pv}}]
Suppose that $\Tdec$ is a decomposition for a closed link $\K$ with $|\K|$ connected components.  Then there is a bigraded  homotopy equivalence of complexes
\[\ctt_0(\Tdec)\simeq \ctt_1(\Tdec)\simeq \HFKh(\K)\otimes (\Ft\oplus (\Ft[1]\{2\}))^{\otimes(|\Tdec|-|\K|)}\otimes (\Ft\oplus (\Ft[1]\{0\})).\]
\end{theorem}

\subsection{Dictionary}\label{ssec:dictionary}
Section~\ref{sec:tangle} can be thought of as a ``user's guide'' to \cite{pv};  the terminology  \scalebox{.96}[1.0]{and  notation  here are different from   \cite{pv} to ease readability. Below is a (non-comprehensive)} list of the major correspondences between this section and \cite[Section 3]{pv}. 

\begin{center}
\begin{tabular}{l l }
\toprule
{\bf Here} \hspace{5.1cm} & {\bf In \cite{pv}} \hspace{7cm} \\
\midrule
 a tangle decomposition   & an alternating sequence of shadows and mirror-\\
 $\Tdec = (\T_1, \ldots, \T_k)$ & shadows $\boldsymbol{\P} = (\P_1^*, \P_2, \ldots, \P_{2k-1}^*, \P_{2k}) $\\
\midrule
the $i$-th even piece in $\Tdec$  & the shadow $\P_{2i}$\\
 & (drawn over white background color)\\
\midrule
the  $i$-th odd piece in $\Tdec$ & the mirror-shadow $\P_{2i-1}^*$\\
 & (drawn over grey background color)\\
\midrule
\textcolor{red}{$\longrightarrow$} in an even piece & \textcolor{green}{- - - -} on $\P$\\
\midrule
\textcolor{red}{$\longleftarrow$} in  an odd piece & \textcolor{green}{- - - -} on $\P^*$\\
\midrule
\textcolor{red}{$\longleftarrow$} in an even piece & \textcolor{orange}{$=\mathrel{\mkern-3mu}=\mathrel{\mkern-3mu}=$} on $\P$\\
\midrule
\textcolor{red}{$\longrightarrow$} in an odd piece & \textcolor{orange}{$=\mathrel{\mkern-3mu}=\mathrel{\mkern-3mu}=$} on $\P^*$\\
\midrule
$\PP\subset \{\pm1\}^n$ & $\epsilon\in (2^{\{\pm\}})^{n}$\\
\midrule
$V_i$ & $\a$ for $\P_{2i}$ and $P_{2i+1}^*$ (called $\a_i$ below)\\
\midrule
$W_i$ & $\b$ for $\P_{2i-1}^*$ and $\P_{2i}$ (called $\b_i$ below)\\
\bottomrule
\end{tabular}
\end{center}

Further, fix a generator $x$ of $\Tdec$, where $\Tdec$ is as above,  and label the respective partial bijections as $x_{2i-1}:V_{i-1}\to W_i$ and $x_{2i}:W_i\to V_i$. We have the following correspondences. 

\begin{center}
\begin{tabular}{l l }
\toprule
{\bf Here}\hspace{5.1cm} & {\bf In \cite{pv}}\\
\midrule
 $x_{2i}:W_i\to V_i$ & $(S_{2i}, T_{2i}, \phi_{2i})$, where $S_{2i}\subset \b_i, T_{2i}\subset \a_i$ and  \\
 &  $\phi_{2i}:S_{2i}\to T_{2i}$ is a bijection \\
\midrule
 $x_{2i-1}:V_{i-1}\to W_i$ & $(S_{2i-1}, T_{2i-1}, \phi_{2i-1})^*$, where $S_{2i-1}\subset \b_{i-1}$, $T_{2i-1}\subset \a_i$\\
 & and  $\phi_{2i-1}:S_{2i-1}\to T_{2i-1}$ is a bijection \\
\midrule
a black strand for $x_i$ &  a pair $(s,\phi s)$ for $s\in S_i$\\
\midrule
$\sssA{1}(x)$ & $T_{2k}$\\
\midrule
$\sssD{0}(x)$ & $\a_0 \setminus T_0$\\
\bottomrule
\end{tabular}
\end{center}

\newpage 



\section{Proofs of the main results} 
\label{sec:results}


\subsection{The Grothendieck group of $\AA(\PP)$}\label{subsec-A-P}

If an element $x \in \AA(\PP)$ satisfies $\ee{\ttt} x = \delta_{\sss,\ttt} x$, we say that $x$ has \emph{left $\II(\PP)$-degree} equal to $\sss$.  We define right $\II(\PP)$-degree as well as such degrees for elements of $\AA(\PP)$-modules analogously.

Let $\PP\in\{\pm1\}^n$ be a sign sequence, $\AA(\PP)$ the corresponding tangle Floer dg algebra, and $\II(\PP)$ the ring of idempotents in $\AA(\PP)$.  For a subset $\sss\subseteq[n]$, let $\ee{\sss}$ be the corresponding primitive idempotent.  Since $\II(\PP)\cong\Ft^{\oplus 2^{n+1}}$ with zero differential, we have
\begin{equation}
K_0(\II(\PP))\cong\Z^{2^{n+1}}\otimes_\Z\Z[q^{\pm1}].
\end{equation}
A free $\Z[q^{\pm1}]$-basis for $K_0(\II(\PP))$ is given by $\{[\II(\PP)\ee{\sss}]\mid \sss\subseteq[n]\}$.  We call this the \emph{primitive basis} of $K_0(\II(\PP))$.

Let $V$ be a bigraded $\Ft$-vector space and write $V_{i,j}$ for the part of $V$ in homological (Maslov) degree $i\in\Z$ and internal (twice Alexander) degree $j\in\Z$.  Suppose $\bigoplus_{i\in\Z}V_{i,j}$ is finite dimensional for each $j\in\Z$.  The \emph{graded Euler characteristic} of $V$ is the formal series
\begin{equation}
\chi(V) = \sum_{i,j\in\Z} (-1)^i q^{j} \dim_{\Ft}(V_{i,j}).
\end{equation}
In all our cases of interest, $\chi(V)$ will be a polynomial in $q^{\pm1}$.

Let $\prescript{\AA}{}{\Ft\{\sss\}}$ be the one-dimensional left type $\tD$ structure over $\AA$ spanned by a homogeneous generator $x_\sss$ of Maslov degree $0$, Alexander degree $0$, $\II(\PP)$-degree $\sss$, and satisfying $\delta^1(x_\sss)=0$.  The corresponding left type $\tA$ module
\begin{equation}\label{eqn-one-dim-type-D}
\AA \boxtimes \prescript{\AA}{}{\Ft\{\sss\}} \cong \AA \ee{\sss}.
\end{equation}

\begin{proposition}\label{prop-K0-tangle-floer}
Let $\PP\in\{\pm1\}^n$.
\begin{enumerate}
  \item Then the Grothendieck group of the compact derived category of left dg modules over $\AA(\PP)$ is
  \begin{equation}
  K_0(\AA(\PP))\cong \Z^{2^{n+1}}\otimes \Z[q^{\pm1}].
  \end{equation}
  \item If $\prescript{}{\AA(\PP)}{M}$ is a compact cofibrant dg module over $\AA(\PP)$, then $[M] = \sum_{\sss\subseteq[n]}c_{\sss}[\AA(\PP)\ee{\sss}]$ for certain constants $c_{\sss}\in\Z[q^{\pm1}]$.  The set $\{[\AA\ee{\sss}]\mid\sss\subseteq[n]\}$ is a basis of $K_0(\AA(\PP))$ over $\Z[q^{\pm1}]$.
  \item Under the quasi-equivalence of Proposition \ref{prop-A-D-equivalence}, we can view the class of a bounded left type $\tD$ structure as an element of $K_0(\AA(\PP))$.  Under this identification, for any bounded left $D$ structure $\prescript{\AA(\PP)}{}{N}$,
  \begin{equation}\label{eqn-prop-K0-D}
  [\prescript{\AA(\PP)}{}{N}] = \sum_{\sss\subseteq[n]} \chi(\ee{\sss}N)[\AA(\PP)\ee{\sss}].
  \end{equation}
\end{enumerate}
\end{proposition}
\begin{proof} The idea of the proof is analogous to that of \cite[Theorem 21]{dec}.  We give the details here for completeness.

The dg $(\II(\PP),\AA(\PP))$-bimodule $\AA(\PP)$ induces a derived tensor functor (as in Subsection \ref{subsec:dg-algebra}) $\dcat{\AA(\PP)}\to\dcat{\II(\PP)}$.  This functor takes a dg module $M$ over $\AA(\PP)$ to $M/(\AA(\PP)_+M)$.  In particular, for a primitive idempotent $\ee{\sss}$, the induced map on Grothendieck groups takes $[\AA(\PP)\ee{\sss}]$ to the primitive basis vector $[\II(\PP)\ee{\sss}]$.  It follows that
\begin{equation*}
\{[\AA(\PP)\ee{\sss}]\mid\sss\subseteq[n]\}
\end{equation*}
is a $\Z[q^{\pm1}]$-linearly independent set in $K_0(\AA(\PP))$.

It suffices, then, to express the symbol of any compact cofibrant dg module over $\AA(\PP)$ as a $\Z[q^{\pm1}]$-linear combination of symbols of the form $[\AA(\PP)\ee{\sss}]$.  By Proposition \ref{prop-A-D-equivalence}, it suffices to consider modules of the form $\AA(\PP)\boxtimes N$, where $N$ is a bounded left type $\tD$ structure over $\AA(\PP)$.

Given such a module, choose an $\Ft$-basis $x_1,\ldots,x_r$ for $N$ such that $x_i = \eeD{0}(x_i) x_i$ for each $i$.  By boundedness, we can choose $j$ such that $\delta^1(x_j)=0$.  In the sum
\begin{equation*}
d(a\otimes x_j) = \sum_{k=1}^\infty \left[(m_k\otimes\id_N) \circ (\id_M\otimes\delta^{k-1})\right](a\otimes x_j),
\end{equation*}
only the $k=1$ term survives.  This term equals $m_1(a)\otimes x_j=d(a)\otimes x_j$.  Hence $\AA(\PP)\otimes x_j$ is a dg submodule isomorphic to $\AA(\PP)\ee{x_j}[-M(x_j)]\{-2A(x_j)\}$, where $\ee{x_j}$ is the primitive idempotent with the same $\II(\PP)$-degree as $x_j$.  Write $M_j=\AA(\PP)\ee{x_j}[-M(x_j)]\{-2A(x_j)\}$ for short, and let $\iota$ be the inclusion of this submodule.  Then the triangle
\begin{equation*}
\xymatrix{M_j \ar[r]^-\iota & \AA(\PP)\boxtimes N \ar[r] & C(\iota) \ar[r] & M_j[1]}
\end{equation*}
is isomorphic to a distinguished triangle, so 
\begin{equation*}
[\AA(\PP)\boxtimes N] = [M_j] + [(\AA(\PP)\boxtimes N)/M_j].
\end{equation*}
We have $[M_j] = (-1)^{M(x_j)}q^{-2A(x_j)}[\AA(\PP)\ee{x_j}]$.  So by induction on the dimension of $N$, $[\AA(\PP)\boxtimes N]$ is a sum of terms of the form $\pm q^k [\AA(\PP)\ee{\sss}]$.  So $K_0(\AA(\PP))$ is a free $\Z[q^{\pm1}]$-module with a basis given by the $2^{n+1}$ symbols $\{[\AA\ee{\sss}] \mid \sss\subseteq[n]\}$.

Equation \eqref{eqn-prop-K0-D} follows immediately from the above analysis and the quasi-equivalence of Proposition \ref{prop-A-D-equivalence}.
\end{proof}

Proposition \ref{prop-K0-tangle-floer} implies the first half of Theorem A of the Introduction.

We call $\{[\AA\ee{\sss}]\mid \sss\subseteq[n]\}$ the \emph{primitive basis} of $K_0(\AA)$.  When discussing elements of and linear maps between Grothendieck groups of $\AA(\PP)$'s, our default will be to use this basis.  We call $a_\sss$ the ``$\sss$ component'' of an element
\begin{equation*}
\sum_{\sss}a_\sss[\AA\ee{\sss}]\in K_0(\AA),
\end{equation*}
and if $X:K_0(\AA(\PP))\to K_0(\AA(\PP'))$ is a $\Z[q^{\pm1}]$-linear map and
\begin{equation*}
X([\AA(\PP)\ee{\ttt}]) = \sum_{\sss} X_{\sss,\ttt} [\AA(\PP')\ee{\sss}],
\end{equation*}
we call $X_{\sss,\ttt}$ the ``$(\sss,\ttt)$ matrix entry'' of $X$.  In our main construction, $X$ will usually be the linear map induced by the derived (or box) tensor product with some bimodule.  In the case of a type $\tDA$ bimodule, the computation of the components $X_{\sss,\ttt}$ is just a graded dimension count, as we presently explain.

\begin{proposition}\label{prop-matrix-entries}
Let $M$ be a bounded type $\tDA$ bimodule over $(\AA(\PP'), \AA(\PP))$ for two sign sequences $\PP,\PP'$.  $M$ induces a homomorphism
\begin{equation*}
[M\boxtimes-] : K_0(\AA(\PP)) \to K_0(\AA(\PP'))
\end{equation*}
via the rule
\begin{equation}\label{eqn-DA-homo-1}
[M\boxtimes-] : [\prescript{}{\AA(\PP)}{M'}] \mapsto [\AA(\PP') \boxtimes \prescript{\AA(\PP')}{}{M}_{\AA(\PP)} \dtensor{} \prescript{}{\AA(\PP)}{M'}].
\end{equation}
Using the quasi-equivalence of Proposition \ref{prop-A-D-equivalence}, this rule can also be stated as
\begin{equation}\label{eqn-DA-homo-2}
[M\boxtimes-] : [\prescript{\AA(\PP)}{}{N}] \mapsto [\prescript{\AA(\PP')}{}M_{\AA(\PP)} \boxtimes \prescript{\AA(\PP)}{}{N}].
\end{equation}
This homomorphism has matrix entries
\begin{equation}\label{eqn-DA-homo-matrix}
[M\boxtimes-]_{\sss,\ttt} = \chi(\ee{\sss}M\ee{\ttt}).
\end{equation}
\end{proposition}
\begin{proof} 
That $M$ induces a homomorphism via \eqref{eqn-DA-homo-1} follows from the quasi-equivalence of Proposition \ref{prop-A-D-equivalence} and the fact that the derived tensor product with a dg module induces a homomorphism of Grothendieck groups (cf. Section \ref{subsec:dg-algebra}).  Again by Proposition \ref{prop-A-D-equivalence}, the rules \eqref{eqn-DA-homo-1} and \eqref{eqn-DA-homo-2} are equivalent.   Abbreviating $\AA(\PP)$ by $\AA$ and $\AA(\PP')$ by $\AA'$, the matrix entry description is an easy computation:
\begin{equation*}
[\prescript{\AA'}{}{M}_\AA \dtensor{} \AA \boxtimes \prescript{\AA}{}{\Ft\{\sss\}}]
\refequal{\eqref{eqn-one-dim-type-D}} [\prescript{\AA'}{}{M}_\AA \dtensor{} \AA \ee{\sss}]
= \sum_\ttt [\AA' \boxtimes \ee{\ttt} M \ee{\sss}]
= \sum_\ttt \chi(\ee{\ttt}M\ee{\sss}) [\AA'\ee{\ttt}].\qedhere
\end{equation*}
\end{proof}

\begin{corollary}\label{cor:chi}
Let $\T$ be a tangle with decomposition $\Tdec$. Then the matrix entries of the homomorphism $[\ctt(\Tdec)\boxtimes-]:K_0(\AA(\bdy^1\T))\to K_0(\AA(-\bdy^0\T))$ are given by
\begin{equation}
[\ctt(\Tdec)\boxtimes-]_{\sss, \ttt} = \chi(\ee{\sss} \ctt(\Tdec) \ee{\ttt}).
\end{equation}
\end{corollary}

We will abbreviate $[M\boxtimes-]_{\sss,\ttt}$ as $[M]_{\sss,\ttt}$ below.

\subsection{The action of $\ctt(\Tdec)$ on the Grothendieck group}\label{subsec-K0-CT}

In this subsection, we compute the symbol $[\ctt(\Tdec)]$ for any tangle decomposition by induction.  Our base case is an explicit computation for the smallest elementary tangles.  Then we induct by computing the result of adding a horizontal strand above or below a decomposition $\Tdec$.

\begin{remark*} \emph{All matrices below are with respect to the primitive basis $\{[\AA(\PP)\ee{\sss}] \mid I \subseteq[n]\}$ of $K_0(\AA(\PP))$ as defined in the previous subsection.  We give this basis the \emph{reverse lexicographic order} with respect to the alphabet $0<1<\ldots <n$: a subset $\sss$ is treated as a word $w(\sss)$ spelled in decreasing order, and we say $\sss < \ttt$ if $w(\sss) < w(\ttt)$ in the lexicographic (``alphabetical'') order.  For example, when $|\PP|=2$, the ordering is
\begin{equation*}
\ee{\emptyset} \enskip<\enskip \ee{\{0\}} \enskip<\enskip \ee{\{1\}} \enskip<\enskip \ee{\{0,1\}} \enskip<\enskip \ee{\{2\}} \enskip<\enskip \ee{\{0,2\}} \enskip<\enskip \ee{\{1,2\}} \enskip<\enskip \ee{\{0,1, 2\}}.
\end{equation*}
Suppose $|\PP|=n$ and $\PP'$ is the sign subsequence which omits the last entry of $\PP$.  The ordered primitive basis for $\PP$ has length $2^{n+1}$.  Its first $2^n$ elements are exactly the ordered primitive basis for $\PP'$, and the $(2^n+i)$-th element is $\sss\sqcup\{n\}$, where $\sss$ is the $i$-th element of either ordered basis.}
\end{remark*}

From now on, using the previous remark and Proposition \ref{prop-K0-tangle-floer}, we will identify
\begin{equation}\label{eqn-identification}\begin{split}
K_0(\AA(\PP)) \otimes_{\Z[q^{\pm1}]} \C(q) & \cong V_{\PP} \otimes L(\lambda_{\PP}) \\
[\AA(\PP)\ee{\sss}] & \leftrightarrow \ell_{j_1} \wedge \ldots \wedge \ell_{j_r},
\end{split}\end{equation}
where $\{j_1 < \ldots < j_r\} = [n] \setminus \sss$.  Note that this identification is order preserving.

Although by invariance it suffices to compute only one of the eight possible oriented crossings, we provide the matrices for all crossings (but provide details for only one).

In the discussion below, let $\PP=(P_1, P_2)$ be a sign sequence of length $2$.

\scalebox{.957}[1.0]{Let $\mathtt e_{\PP}$ denote the elementary tangle that is an  \texttt{e}-crossing of two strands such that  $\bdy^1(\mathtt e_{\PP}) = {\PP}$.} \scalebox{.964}[1.0]{Let  $\mathtt o_{\PP}$ denote the elementary tangle that is an \texttt{o}-crossing of two strands such that $-\bdy^0(\mathtt o_{\PP}) = \PP$.} With this notation, $\Tdec=(\mathtt e_{\PP}, \mathtt o_{\PP})$ is a decomposition for the tangle $\T=I\times (P_2, P_1)$, and $(\mathtt o_{\PP}, \mathtt e_{\PP})$ is a decomposition for the tangle $I\times {\PP}$.

To compute the action of $\ctt(\Tdec)$ on $K_0$, we just need to compute the bigradings of the generators, and take the graded Euler characteristic. Recall that the primitive idempotents are of form $\ee{\sss}$, where $\sss\subseteq \{0,1,2\}$, and that a generator $x$ for $\ctt(\Tdec)$ occupies the black dot at height $h$ on the left (respectively right) exactly when $h\notin \sssD{0}(x)$ (respectively $h\in \sssA{1}(x)$).  For ease of reading, we give the matrices for each weight separately.

We discuss the computation of  $[\ctt_2(\mathtt e_{++})]_{\{0,2\}, \{1,2\}}$ below. Since the induced order on subsets of size $2$ is $\enskip \{0,1\}\enskip<\enskip \{0,2\} \enskip<\enskip \{1,2\}$, this corresponds to the $(2,3)$ entry of the matrix for $[\ctt_2(\mathtt e_{++})]$.

The $\tDA$ bimodule $\ctt(\mathtt e_{++})$ for the crossing $\mathtt e_{++}$ (depicted on the left diagram below) is generated by pairs of partial bijections associated to the right diagram below. 

\begin{center}
\includegraphics[scale = .75]{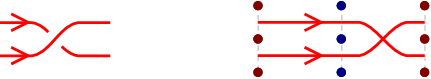} 
\end{center}

The vector space $\ee{\{0,2\}} \ctt(\mathtt e_{++}) \ee{\{1,2\}}$ is generated over $\Ft$ by the strand diagrams in Figure \ref{fig:e++02to12}. The formulas from Section \ref{ssec:gen} yield the bigradings listed in Figure \ref{fig:e++02to12}.
   \begin{figure}[h]
 \centering
 \labellist
             \pinlabel  $(-1,-2)$ at 40 -14
             \pinlabel  $(-2,-4)$ at 140 -14
             \pinlabel  $(-1,-2)$ at 238 -14
             \pinlabel  $(0,0)$ at 335 -14
             \pinlabel  $(-1,-2)$ at 435 -14
             \pinlabel  $(0,-2)$ at 530 -14
       \endlabellist
 \includegraphics[scale = .75]{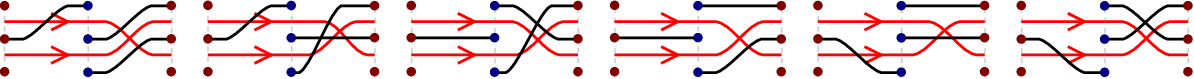} 
 \vskip .3 cm
       \caption{The generators of $\ee{\{0,2\}} \ctt(\mathtt e_{++}) \ee{\{1,2\}}$ and their $(M, 2A)$ bigradings.}\label{fig:e++02to12}
\end{figure}

   Thus, 
      \[[\ctt_2(\mathtt e_{++})]_{\{0,2\}, \{1,2\}} = \chi(\ee{\{0,2\}} \ctt(\mathtt e_{++}) \ee{\{1,2\}}) = \sum_{x \textrm{ a generator}}(-1)^{M(x)}q^{2A(x)} =  (1-q^{-2})^2.\]
   
   Analogous computations for the remaining weights of $\ctt(\mathtt e_{++})$ and for the remaining \texttt{e}-crossings yield the matrices in Table \ref{tbl:crossings}.

\begin{table}[p]
     \begin{center}
     \begin{tabular}{  p{2.5cm} p{4.5cm}  p{7.5cm} }
     \toprule
      {\bf Tangle} & {\bf Diagram} & {\bf Action on $K_0$} \\ 
      \midrule
     \labellist   
        \pinlabel $\mathtt e_{++}$ at 28 40 
     \endlabellist
     \raisebox{-.5\totalheight}{\includegraphics[scale = .7]{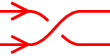}}
      & 
     \raisebox{-.5\totalheight}{\includegraphics[scale = .7]{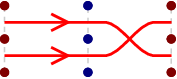}}
      & 
{$\!\begin{aligned}
[\ctt_0(\mathtt e_{++})] &=  (1-q^{-2})^2  \begin{pmatrix}
-q
   \end{pmatrix}\\
[\ctt_1(\mathtt e_{++})] &=  (1-q^{-2})^2  \begin{pmatrix}
      -q & 1 & 0 \\
      0 & q^{-1} & 0\\
      0 & 1 & -q
   \end{pmatrix}\\
[\ctt_2(\mathtt e_{++})] &=  (1-q^{-2})^2  \begin{pmatrix} 
      q^{-1} & 0 & 0 \\
      1 & -q & 1\\
      0 & 0 & q^{-1}
   \end{pmatrix}\\
   [\ctt_3(\mathtt e_{++})] &=  (1-q^{-2})^2  \begin{pmatrix} 
        q^{-1}
   \end{pmatrix}
   \end{aligned}$}
  \\
        \midrule
     \labellist   
        \pinlabel $\mathtt e_{-+}$ at 28 40 
     \endlabellist
     \raisebox{-.5\totalheight}{\includegraphics[scale = .7]{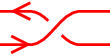}}
      & 
     \raisebox{-.5\totalheight}{\includegraphics[scale = .7]{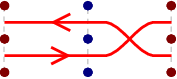}}
      & 
{$\!\begin{aligned}
[\ctt_0(\mathtt e_{-+})] &=  (1-q^{-2})^2  \begin{pmatrix}
1
   \end{pmatrix}\\
[\ctt_1(\mathtt e_{-+})] &=  (1-q^{-2})^2  \begin{pmatrix}
      1 & q & 0 \\
      0 & 1 & 0\\
      0 & -q^{-1} & 1
   \end{pmatrix}\\
[\ctt_2(\mathtt e_{-+})] &=  (1-q^{-2})^2  \begin{pmatrix} 
      1 & 0 & 0 \\
      -q^{-1} & 1 & q\\
      0 & 0 & 1
   \end{pmatrix}\\
   [\ctt_3(\mathtt e_{-+})] &=  (1-q^{-2})^2  \begin{pmatrix} 
        1
   \end{pmatrix}
   \end{aligned}$}
  \\
         \midrule
     \labellist   
        \pinlabel $\mathtt e_{+-}$ at 28 40 
     \endlabellist
        \raisebox{-.5\totalheight}{\includegraphics[scale = .7]{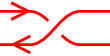}}
      & 
     \raisebox{-.5\totalheight}{\includegraphics[scale = .7]{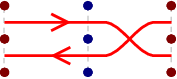}}
      & 
{$\!\begin{aligned}
[\ctt_0(\mathtt e_{+-})] &=  (1-q^{-2})^2  \begin{pmatrix}
1
   \end{pmatrix}\\
[\ctt_1(\mathtt e_{+-})] &=  (1-q^{-2})^2  \begin{pmatrix}
      1 & -q^{-1} & 0 \\
      0 & 1 & 0\\
      0 & q & 1
   \end{pmatrix}\\
[\ctt_2(\mathtt e_{+-})] &=  (1-q^{-2})^2  \begin{pmatrix} 
      1 & 0 & 0 \\
      q & 1 & -q^{-1} \\
      0 & 0 & 1
   \end{pmatrix}\\
   [\ctt_3(\mathtt e_{+-})] &=  (1-q^{-2})^2  \begin{pmatrix} 
1
   \end{pmatrix}
   \end{aligned}$}
      \\
              \midrule
     \labellist   
        \pinlabel $\mathtt e_{--}$ at 28 40 
     \endlabellist
          \raisebox{-.5\totalheight}{\includegraphics[scale = .7]{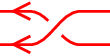}}
      & 
     \raisebox{-.5\totalheight}{\includegraphics[scale = .7]{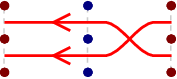}}
      & 
{$\!\begin{aligned}
[\ctt_0(\mathtt e_{--})] &=  (1-q^{-2})^2  \begin{pmatrix}
q^{-1}
   \end{pmatrix}\\
[\ctt_1(\mathtt e_{--})] &=  (1-q^{-2})^2  \begin{pmatrix}
      q^{-1} & 1 & 0 \\
      0 & -q & 0\\
      0 & 1 & q^{-1}
   \end{pmatrix}\\
[\ctt_2(\mathtt e_{--})] &=  (1-q^{-2})^2  \begin{pmatrix} 
      -q & 0 & 0 \\
      1 & q^{-1} & 1\\
      0 & 0 & -q
   \end{pmatrix}\\
   [\ctt_3(\mathtt e_{--})] &=  (1-q^{-2})^2  \begin{pmatrix} 
-q
   \end{pmatrix}
   \end{aligned}$}
      \\
      \bottomrule
      \end{tabular}
      \vspace{.2cm}
      \caption{The action of the crossing bimodules on the Grothendieck group}
      \label{tbl:crossings}
      \end{center}
      \end{table}

  An elementary tangle $\triv_{\PP} = I\times {\PP}$ acts as a scalar multiple of the identity on $K_0$,
  \begin{equation}
  [\ctt(\triv_{\PP})]_{\sss, \ttt} = (1-q^{-2})^2 \id.
  \end{equation}
  One can see this by computing $ [\ctt(\id_{()})]$, and then applying Proposition \ref{prop:addstrand} (below).
   
   By invariance (Theorem \ref{thm-invariance}),  since $(\mathtt o_{\PP},  \mathtt e_{\PP})$ is isotopic to   $\triv_{\PP}$, we have 
   \[\ctt(\mathtt o_{\PP},  \mathtt e_{\PP})\simeq \ctt (\triv_{\PP})\otimes (\Ft\oplus \Ft[1]\{2\})^{\otimes 2},\] 
   so $[\ctt(\mathtt o_{\PP},  \mathtt e_{\PP})] = (1-q^{-2})^4\id$. Then 
   \[ [\ctt_i(\mathtt o_{\PP})] = (1-q^{-2})^4[\ctt_i(\mathtt e_{\PP})]^{-1}.\]
    Alternatively, one can compute the Euler characteristic explicitly. 
   
   Let $\mathtt{cup}_{+-}$  or $\mathtt{cup}_{-+}$ denote the elementary tangle that is a single cup such that $\bdy^0(\mathtt {cup}_{+-}) = (+,-)$ or $\bdy^0(\mathtt {cup}_{-+}) = (-, +)$, respectively.    Let $\mathtt{cap}_{+-}$  or $\mathtt{cap}_{-+}$ denote the elementary tangle that is a single cup such that $\bdy^1(\mathtt {cap}_{+-}) = (+,-)$ or $\bdy^1(\mathtt {cap}_{-+}) = (-, +)$, respectively. 
      
 The only nonzero summands of $\ctt$ for the two cups are those of weight $0$ and  $1$, and the only nonzero summands of $\ctt$ for the two caps are those of weight $1$ and $2$. Their action on $K_0$ is given by the matrices in Table \ref{tbl:cupcap}. 

   \begin{table}[h!]
     \begin{center}
     \begin{tabular}{  p{2.5cm} p{4.5cm}  p{7.5cm} }
     \toprule
      {\bf Tangle} & {\bf Diagram} & {\bf Action on $K_0$} \\ 
      \midrule
           \labellist   
        \pinlabel $\mathtt{cup}_{-+}$ at 22 38 
     \endlabellist
           \raisebox{-.5\totalheight}{\includegraphics[scale = .7]{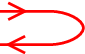}}
      & 
     \raisebox{-.5\totalheight}{\includegraphics[scale = .7]{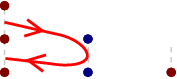}}
      & 
     {$\!\begin{aligned}
[\ctt_0(\mathtt{cup}_{-+})] =  (1-q^{-2})  \begin{pmatrix} 
      1\\
      0\\
      1
   \end{pmatrix}\\
[\ctt_1(\mathtt{cup}_{-+})] =  (1-q^{-2})  \begin{pmatrix} 
      0\\
      1\\
      0
   \end{pmatrix}
   \end{aligned}$}
  \\
      \midrule
           \labellist   
        \pinlabel $\mathtt{cup}_{+-}$ at 22 38 
     \endlabellist
      \raisebox{-.5\totalheight}{\includegraphics[scale = .7]{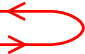}}
      & 
     \raisebox{-.5\totalheight}{\includegraphics[scale = .7]{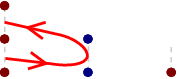}}
      & 
           {$\!\begin{aligned}
[\ctt_0(\mathtt{cup}_{+-})] =  (1-q^{-2})  \begin{pmatrix} 
      1\\
      0\\
      1
   \end{pmatrix}\\
[\ctt_1(\mathtt{cup}_{+-})] =  (1-q^{-2})  \begin{pmatrix} 
      0\\
      1\\
      0
   \end{pmatrix}
   \end{aligned}$}
  \\
      \midrule
            \addlinespace[0.4em]
                \labellist   
        \pinlabel $\mathtt{cap}_{+-}$ at 22 38 
     \endlabellist
         \raisebox{-.6\totalheight}{\includegraphics[scale = .7]{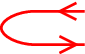}}
      & 
     \raisebox{-.4\totalheight}{\includegraphics[scale = .7]{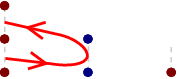}}
      & 
           {$\!\begin{aligned}
[\ctt_1(\mathtt{cap}_{+-})] =  (1-q^{-2})  \begin{pmatrix} 
      0 & 1 & 0
   \end{pmatrix}\\
[\ctt_2(\mathtt{cap}_{+-})] =  (1-q^{-2})  \begin{pmatrix} 
      1 & 0 & 1
   \end{pmatrix}
   \end{aligned}$}
      \\
      \addlinespace[0.4em]
         \midrule
          \addlinespace[0.4em]
             \labellist   
        \pinlabel $\mathtt{cap}_{-+}$ at 22 38 
     \endlabellist
            \raisebox{-.6\totalheight}{\includegraphics[scale = .7]{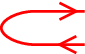}}
      & 
     \raisebox{-.4\totalheight}{\includegraphics[scale = .7]{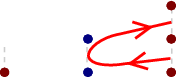}}
      & 
           {$\!\begin{aligned}
[\ctt_1(\mathtt{cap}_{-+})] =  (1-q^{-2})  \begin{pmatrix} 
      0 & 1 & 0
   \end{pmatrix}\\
[\ctt_2(\mathtt{cap}_{-+})] =  (1-q^{-2})  \begin{pmatrix} 
      1 & 0 & 1
   \end{pmatrix}
   \end{aligned}$}
      \\
            \addlinespace[0.4em]
      \bottomrule
      \end{tabular}
      \vspace{.2cm}
      \caption{The action of the cup and cap bimodules on the Grothendieck group}
      \label{tbl:cupcap}
      \end{center}
      \end{table}
      \vskip -.3cm
      
Any elementary tangle can be obtained from one of the small elementary tangles above by adding straight strands above and/or below it. We now discuss the effect that adding a strand has on $[\ctt(\Tdec)]$.

 If $\Tdec$ is a decomposition for a tangle $\T$, let $\Tdec^+$ be the decomposition (for another tangle) obtained by adding a right-oriented horizontal strand  above $\Tdec$, see Figure \ref{fig:addstrand}. 
Let $\Tdec^-$ be the decomposition obtained  by adding a left-oriented horizontal strand above $\Tdec$.  Similarly, let $\Tdec_+$ and $\Tdec_-$ be the decompositions obtained  by adding a right- or left-oriented horizontal strand below $\Tdec$, respectively. 
\begin{figure}[h]
\begin{center}
     \labellist   
        \pinlabel \textcolor{red}{$\Tdec$} at 60 35 
        \pinlabel \textcolor{red}{$\rotatebox{90}{$\cdots$}$} at 15 35
        \pinlabel \textcolor{red}{$\rotatebox{90}{$\cdots$}$} at 110 35
     \endlabellist
\includegraphics[scale = .75]{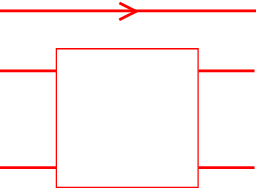}
\vspace{.1cm}
\caption{The decomposition $\Tdec^+$, obtained by adding a strand oriented to the right above $\Tdec$.}
 \label{fig:addstrand}
\end{center}
\end{figure}

\begin{proposition}\label{prop:addstrand}
Let $\Tdec$ be a decomposition for an $(m,n)$-tangle into $k$ elementary tangles, and let $\Tdec'$ be one of $\Tdec^+, \Tdec^-, \Tdec_+, \Tdec_-$. For any  $\sss\subseteq [m]$ and  $\ttt\subseteq [n]$, let  $\sss'=\sss, \ttt'=\ttt, \sss'' = \sss\sqcup \{m+1\}, \ttt'' = \ttt\sqcup \{n+1\}$ if $\Tdec'=\Tdec^{\pm}$, and $\sss' = \{s+1\mid s\in \sss\}$, $\ttt' = \{t+1 \mid t\in \ttt\}$, $\sss'' = \sss\sqcup \{0\}$, $\ttt'' = \ttt\sqcup \{0\}$ if $\Tdec'=\Tdec_{\pm}$. Then 
\begin{align*}
[\ctt(\Tdec')]_{\sss', \ttt'}  &= (1-q^{-2})^k [\ctt(\Tdec)]_{\sss', \ttt'}, \\
[\ctt(\Tdec')]_{\sss'', \ttt''}  &= (1-q^{-2})^k[\ctt(\Tdec)]_{\sss', \ttt'}, \\
[\ctt(\Tdec')]_{\sss'', \ttt'}  &= 0, \\
[\ctt(\Tdec')]_{\sss', \ttt''}  &= 0.
\end{align*}
\end{proposition}

\begin{proof}
We provide the proof for the case $\Tdec' = \Tdec^+$. The other three cases are analogous.

It suffices to prove the proposition for an elementary tangle. {The general case then follows from Theorem~\ref{thm-pairing} and from the fact that tensoring $\tDA$ bimodules corresponds to composing homomorphisms on $K_0$ (i.e. multiplying matrices).}

 Let $\T$ be an elementary tangle, and let $V_0 = \{a_0, \ldots, a_m\}$, $W_1 = \{b_0\ldots, b_l\}$, $V_1= \{c_0\ldots c_n\}$ be the three sets of points for $\T$ as in Section \ref{ssec:gen}, with each of $a_i, b_i, c_i$ indexed by its height. 

Recall that the generators for $\T$ are pairs of partial bijections $V_0\to W_1, W_1\to V_1$, such that each $b_i$ is either in the range of the former, or in the domain of the latter, but not both. The generators for $\T'$ are then pairs of partial bijections $V_0'\to W_1', W_1'\to V_1'$, where  $V_0' = V_0\sqcup \{a_{m+1}\}$, $W_1' = W_1\sqcup \{b_{l+1}\}$, $V_1' = V_1\sqcup \{c_{n+1}\}$.
We think of a generator $x$ for $\T'$  as $x = \{(p_1, q_1),\ldots, (p_{l+1}, q_{l+1})\}$, where $p_i$ and $q_i$ are matched by one of the two bijections for $x$, i.e. they are connected by a black strand in a diagram for $x$. We enumerate pairs that cancel out in Euler characteristic, and list the remaining ones and their bigradings.

Generators for which $b_l$ and $b_{l+1}$ do not connect to $a_{m+1}$ or $c_{n+1}$ cancel out in pars -- if $x$ is such a generator with $(b_l, p), (b_{l+1}, q) \in x$, then $x$ cancels out  $y = x\setminus \{(b_l, p), (b_{l+1}, q)\} \cup \{(b_l, q), (b_{l+1}, p)\}$. See Figure \ref{fig:cancel_2}. One can explicitly compute the bigradings of $x$ and $y$, or just observe that one generator is in the differential of the other. 

\begin{figure}[h]
 \centering
  \labellist
         \pinlabel  $\to$ at 100 35
         \pinlabel  $\to$ at 345 35
         \pinlabel  $\to$ at 590 35
       \endlabellist
 \includegraphics[scale = .65]{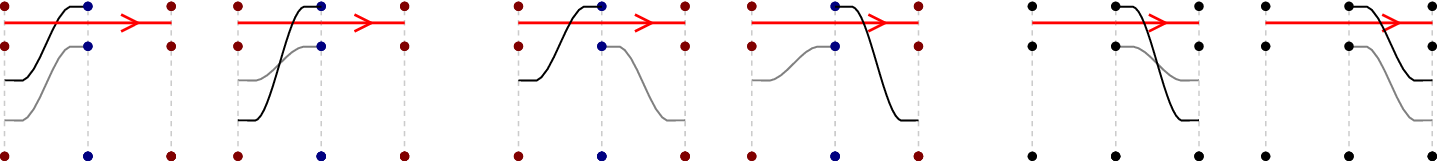} 
 \vskip .1 cm
       \caption{Generators for which $b_l$ and $b_{l+1}$ do not connect to $a_{m+1}$ or $c_{n+1}$ cancel out in pairs. Only the two strands at which a pair differs are shown.}\label{fig:cancel_2}
\end{figure}

Generators for which  $b_{l+1}$ connects to a point in $\{p,q\} = \{a_{m+1}, c_{n+1}\}$, and the other point in $\{a_{m+1}, c_{n+1}\}$ is also an endpoint of a strand also cancel out in pairs -- if  $(b_{l+1}, p), (q, t)\in x$, then $x$ cancels out $y = x\setminus \{(b_{l+1}, p), (q, t)\}\cup \{(b_{l+1}, q), (p, t)\}$. See Figure \ref{fig:cancel_1}.

\begin{figure}[h]
 \centering
 \labellist
             \pinlabel  $\to$ at 115 35
       \endlabellist
 \includegraphics[scale = .65]{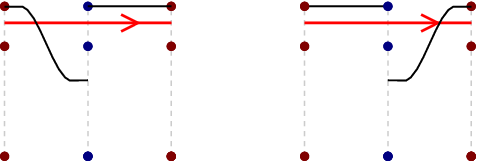} 
 \vskip .1 cm
       \caption{Generators for which $b_l$ and $b_{l+1}$ connect to points in $V_0\cup V_1\setminus \{a_{m+1}, c_{n+1}\}$ cancel out in pars. Only the two strands at which a pair differs are shown.}\label{fig:cancel_1}
\end{figure}

We pair up the remaining generators as $(x,y)$, so that $M(x)-M(y)=1$ and $A(x)-A(y)=1$, see Figure \ref{fig:cancel_3} (calculating the relative bigrading for each pair is an exercise left to the reader). 
\begin{figure}[h]
 \centering
  \labellist
           \pinlabel  \small{$[1]\{2\}$} at 114 185
         \pinlabel  \small{$[1]\{2\}$} at 114 45
                  \pinlabel  \small{$[1]\{2\}$} at 434 185
         \pinlabel  \small{$[1]\{2\}$} at 434 45
         \pinlabel  $\to$ at 115 170
         \pinlabel  $\to$ at 116 30
                  \pinlabel  $\to$ at 435 170
         \pinlabel  $\to$ at 435 30
       \endlabellist
 \includegraphics[scale = .65]{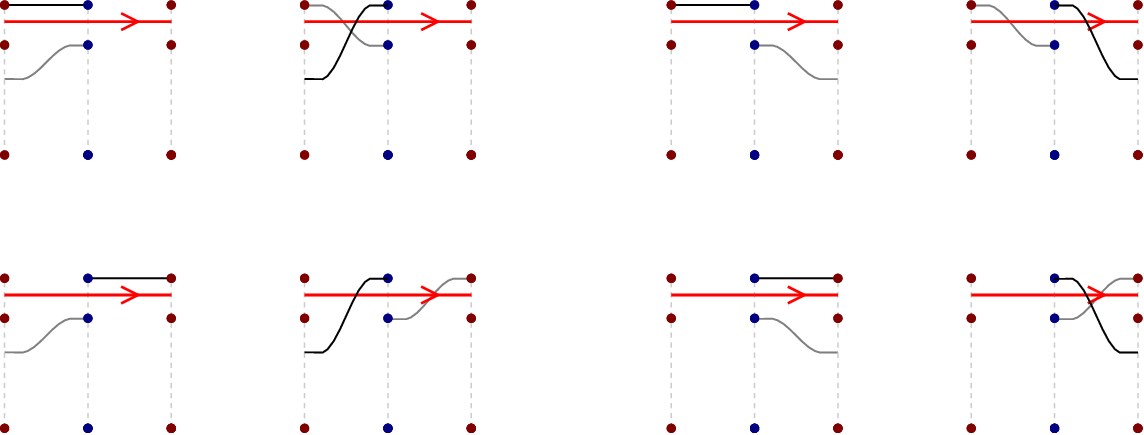} 
 \vskip .1 cm
       \caption{Remaining generators that cannot be cancelled, grouped into pairs $(x,y)$. Only the two strands at which $x$ and $y$ differ are shown.}\label{fig:cancel_3}
\end{figure}
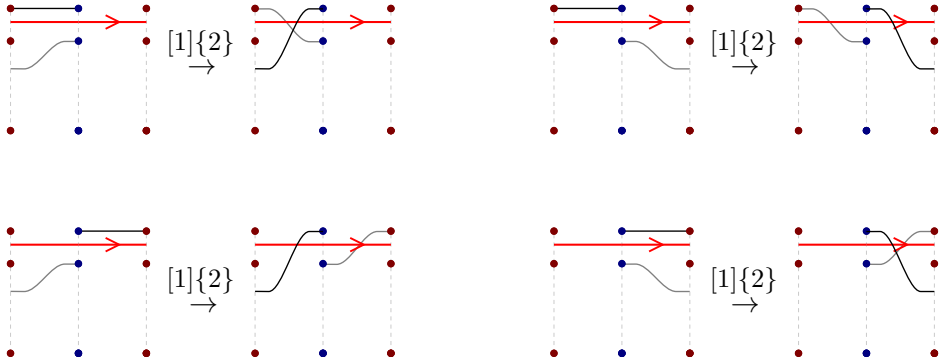
The $x$ generators containing the strand $(a_{m+1}, b_{l+1})$ do not have a strand ending at $c_{n+1}$, so they correspond bijectively with the generators of $\ctt(\Tdec)$, by removing the top of the diagram, namely, the strand  $(a_{m+1}, b_{l+1})$, the points $a_{m+1}, b_{l+1}, c_{n+1}$, and the top (red) strand of $\Tdec'$. Thus, 
\[[\ctt(\Tdec')]_{\sss', \ttt'} = (1-q^{-2}) [\ctt(\Tdec)]_{\sss', \ttt'}.\]
Similarly, the $x$ generators containing the strand $(c_{n+1}, b_{l+1})$ correspond bijectively with the generators of  $\ctt(\Tdec)$, so 
\[[\ctt(\Tdec')]_{\sss'', \ttt''} = (1-q^{-2}) [\ctt(\Tdec)]_{\sss', \ttt'}.\qedhere\]
\end{proof}

To summarize, the computations from this section agree with the computations from Section~\ref{subsubsec-rib} in the following way. Let $\Tdec$ be a decomposition for an $(m,n)$-tangle $\T$. Let $\T^{rot}$ be the tangle obtained by rotating $\T$ clockwise by $\pi/2$ and reversing the orientation. Then, under the identification \eqref{eqn-identification},
\begin{equation}\label{eqn:thmA}
[\ctt(\Tdec)] = q^{\frac{m-n}{2}}(1-q^2)^{|\Tdec|} Q(\T^{rot})\otimes \id_{L(\lambda_{\bdy^1\T})},
\end{equation} 
where $Q$ is the Reshetikhin-Turaev invariant from Section~\ref{subsubsec-Q}. (Note that $\lambda_{\bdy^1\T} = \lambda_{-\bdy^0\T}$.) See Figure~\ref{fig:thma}.

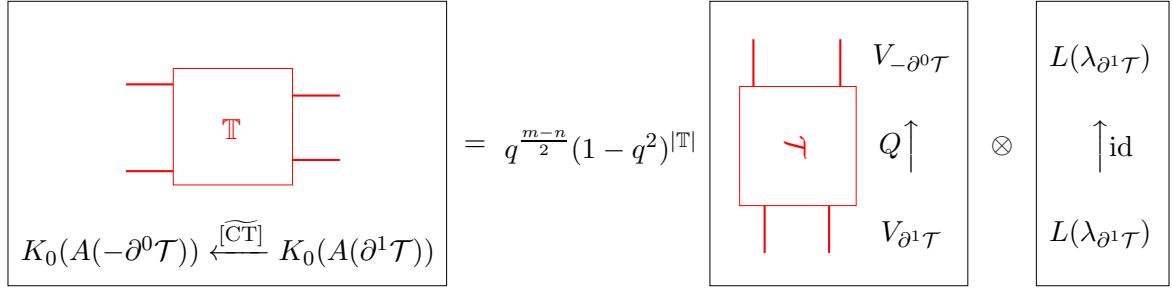
\begin{figure}[h]
 \centering
  \labellist
            \pinlabel \textcolor{red}{$\Tdec$} at 128 90
            \pinlabel  \textcolor{red}{$\rotatebox{270}{$\T$}$} at 450 80
            \pinlabel $K_0(\AA(-\bdy^0\T))\xleftarrow{[\ctt]}K_0(\AA(\bdy^1\T))$ at 126 25
            \pinlabel $=$ at 267 80
            \pinlabel $V_{\bdy^1\T}$ at 516 30
            \pinlabel $\Big\uparrow$ at 516 80
            \pinlabel $V_{-\bdy^0\T}$ at 516 130
            \pinlabel $Q$ at 504 80
            \pinlabel $\otimes$ at 568 80
            \pinlabel $L(\lambda_{\bdy^1\T})$ at 624 30
            \pinlabel $\Big\uparrow$ at 624 80
            \pinlabel $L(\lambda_{\bdy^1\T})$ at 624 130
            \pinlabel $\id$ at 636 80
            \pinlabel $q^{\frac{m-n}{2}}(1-q^2)^{|\Tdec|}$ at 339 80
       \endlabellist
\includegraphics[scale = .67]{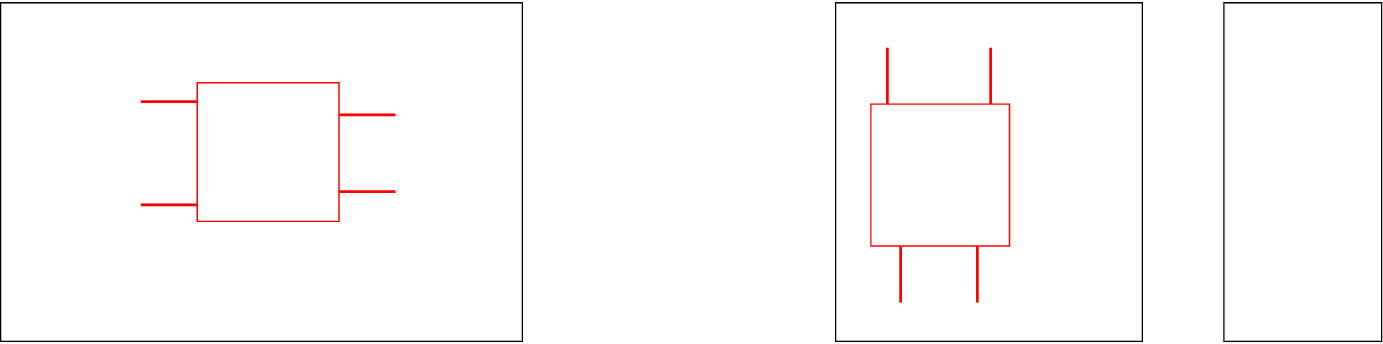} 
 \vskip .5 cm
       \caption{The more precise version of Figure~\ref{fig:RTtangle}, showing the relation between tangle Floer homology and the Reshetikhin-Turaev invariant.}\label{fig:thma}
\end{figure}

This concludes the proof of Theorem A.

\subsection{Adding and removing black strands: $\EE$ and $\FF$}\label{subsec-add-black}

Fix a sign sequence $\PP\in\{\pm1\}^n$. In this subsection, we describe dg bimodules which induce maps on $K_0(\AA(\PP))$ which correspond to the action of $E,F \in U_q$ via the identification \eqref{eqn-identification}.

\begin{definition} A \emph{lowering generator} of weight $k$ associated to $\PP$ is a partial bijection $x:\{-1\}\sqcup[n]\to[n]$ defined on a subset with $k$ elements which contains $-1$.\end{definition}

The diagram associated to a lowering generator $x$ of weight $k$ is drawn as follows.  Draw $n$ horizontal red strands from $(0,i-\frac{1}{2})$ to $(1,i-\frac{1}{2})$ for $1\leq i\leq n$ (oriented according to $\PP$) and a black strand connecting $(0,i)$ to $(1,x(i))$ for each $i$ on which $x$ is defined.  We constrain such diagrams to rules analogous to those in Subsection \ref{ssec:alg} (no horizontal coordinate critical points, no triple intersections, minimal number of crossings).

If $x$ is a lowering generator of weight $k$ associated to $\PP$ defined on a subset $\ttt\subseteq \{-1\}\sqcup [n]$, let
\begin{equation*}
\sssA{0}(x) = \ttt\setminus\{-1\}, \quad \sssA{1}(x) = x(\ttt).
\end{equation*}

\begin{definition} Let $\FF_k(\PP)$ be the $\Ft$-span of all lowering generators of weight $k$ associated to $\PP$.  Give it the structure of a bigraded chain complex by using the same Alexander and Maslov degree formulas as for generators of $\AA(\PP)$,
\begin{eqnarray*}
2A(x) &=& \diagup \hspace{-.375cm}{\color{red}{\nwarrow}}(x) + \diagdown \hspace{-.375cm}{\color{red}{\swarrow}}(x)
-\diagup \hspace{-.37cm}{\color{red}{\searrow}}(x) - \diagdown \hspace{-.37cm}{\color{red}{\nearrow}}(x),\\
M(x) & =& \diagup \hspace{-.35cm}\diagdown(x) - \diagup \hspace{-.37cm}{\color{red}{\searrow}}(x) -\diagdown \hspace{-.37cm}{\color{red}{\nearrow}}(x).
\end{eqnarray*}

We give $\FF_k(\PP)$ the structure of a dg bimodule over $(\AA_{k-1}(\PP), \AA_k(\PP))$ as follows.

If $a$ is a generator of $\AA_{k-1}(\PP)$, $\sssA{1}(a) = \sssA{0}(x)$, and the diagram obtained by concatenating that of $x$ with that of $a$ has a minimal number of crossings, let $a \cdot x = x \circ a'$, where $a':\sssA{0}(a) \sqcup \{-1\} \to \sssA{1}(a) \sqcup \{-1\}$ is defined by
\begin{equation*}
a'|_{\sssA{0}(a)} = a, \quad a'|_{\{-1\}} = \id_{\{-1\}}.
\end{equation*}
Otherwise, let $a \cdot x = 0$.  If $b$ is a generator of $\AA_k(\PP)$, $\sssA{1}(x) = \sssA{0}(b)$, and the relevant concatenated diagram has a minimal number of crossings, let $x \cdot b = b \circ x$; otherwise, let $x \cdot b = 0$.

The differential of a generator of $\FF_k(\PP)$ is the sum of crossing resolutions, just as for the differential on the algebra $\AA(\PP)$.

Let $\EE_k(\PP)$ be the bimodule over $(\AA_k(\PP), \AA_{k-1}(\PP))$ opposite to $\FF_k(-P)$.  Let
\begin{equation*}
\EE(\PP) = \bigoplus_{k=0}^n \EE_k(\PP), \quad
\FF(\PP) = \bigoplus_{k=0}^n \FF_k(\PP).
\end{equation*}

When the choice of $\PP$ is understood or unimportant, we will sometimes write $\AA_k, \II_k, \EE_k,\FF_k$ for $\AA_k(\PP), \II_k(\PP), \EE_k(\PP),\FF_k(\PP)$.

\end{definition}

\begin{figure}[h]
 \centering
 \includegraphics[scale = .7]{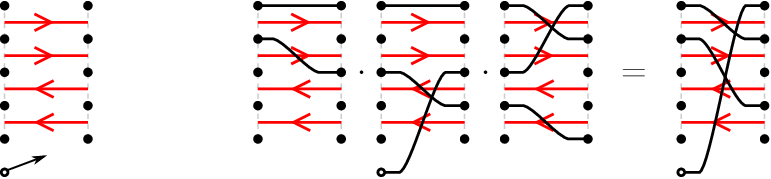} 
 \vskip .2 cm
       \caption{Left: the background diagram for the $(\AA(\PP),\AA(\PP))$ bimodule $\FF(\PP)$, when $\PP= (-, -, +, +)$. Right: an example of the dg bimodule action.}\label{fig:E_AA}
\end{figure}

\begin{figure}[h]
 \centering
 \includegraphics[scale = .7]{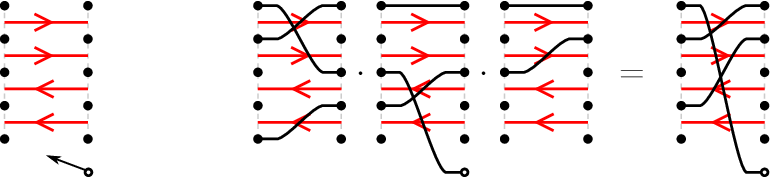} 
 \vskip .2 cm
       \caption{Left: the background diagram for the $(\AA(\PP),\AA(\PP))$ bimodule $\EE(\PP)$, when $\PP= (-, -, +, +)$. Right: an example of the dg bimodule action.}\label{fig:F_AA}
\end{figure}

\begin{proposition}\label{prop-bimodule-E}\begin{enumerate}
Fix a weight $k$, and let $G_k$ be the set of generators of $\FF_k$.
  \item For $i=1,2,\ldots,n$, let
  \begin{equation*}
  F^i = \mathrm{span}_{\Ft}\{x \in G_k \mid x(-1) \leq i\}.
  \end{equation*}
  Then $F^\bullet = (0 \subset F^0 \subset F^1 \subset F^2 \subset \cdots \subset F^n=\FF_k)$ is a filtration of $\FF_k$ by dg left submodules over $\AA_{k-1}$.  For each $i$,
  \begin{equation}
  F^i / F^{i-1} \cong \bigoplus_{\substack{\sss \subseteq [n]\setminus\{i\} \\ |\sss| = k-1}} \AA_{k-1}\ee{\sss}[-M(y_{i,\sss})]\{-2A(y_{i,\sss})\}.
  \end{equation}
  Here, $y_{i,\sss}$ is the generator defined on $\sss\sqcup\{i\}$ which is the identity when restricted to $\sss$ and takes $-1$ to $i$ (see Figure \ref{fig-yi}).
  \item For any subset $\ttt=\{t_1 < \ldots < t_k\}\subseteq[n]$, let
  \begin{equation*}
  F^i_\ttt = F^{t_i} \cap \FF_k\ee{\ttt}.
  \end{equation*}
  This gives the filtration $F^\bullet_\ttt$ on $\FF_k\ee{\ttt}$ induced by the filtration $F^\bullet$ (less some redundant pieces which yield subquotients equal to $0$), and its subquotients are
  \begin{equation}
  F^i_\ttt / F^{i-1}_\ttt \cong \AA_{k-1}\ee{\ttt\setminus\{t_i\}}[-M(y_{t_i,\ttt\setminus\{t_i\}})]\{-2A(y_{t_i,\ttt\setminus\{t_i\}})\}.
  \end{equation}
  \item Let
  \begin{equation*}\begin{split}
  \FF' & = \mathrm{span}_{\Ft}\{x \in G_k \mid x\text{ is defined at }0\}, \\
  \FF'' & = \mathrm{span}_{\Ft}\{x \in G_k \mid x\text{ is not defined at }0\}. \\
  \end{split}\end{equation*}
  Then as dg right submodules over $\AA_k$, $\FF=\FF'\oplus\FF''$.  The submodule $\FF'$ is acyclic, and 
  \begin{equation}
  \FF'' \cong \bigoplus_{\substack{\sss \subseteq [n] \\ |\sss|=k\text{ and }0\in\sss}} \ee{\sss}\AA_k.
  \end{equation}
  \item Let $\ttt \subseteq [n]$ with $|\ttt| = k-1$.  If $0 \in \ttt$, then $\ee{\ttt}\FF_k$ is acyclic.  If $0 \notin \ttt$, then $\ee{\ttt}\FF_k \cong \ee{\ttt\sqcup\{0\}} \AA_k$.
\end{enumerate}
Opposite (interchanging left with right actions) statements to (1)--(4) hold for the bimodule $\EE_k$.
\end{proposition}

 \begin{figure}[h]
 \centering
 \includegraphics[scale = .7]{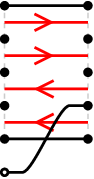} 
\vskip .2 cm
       \caption{The generator $y_{i,\sss}$ for $i=1$, $\sss=\{0,4\}$, when $\PP= (-, -, +, +)$.}
\label{fig-yi}
\end{figure}

\begin{proof}
For (1), note that $F^i$ is a filtration because the differential acting on a generator $x$ cannot increase $x(-1)$.  The subquotient is as described because in the diagram associated to any lowering generator, we can isotope the strand originating at $-1$ so as to make all its crossings to the left or to the right of all other crossings.  Since the differential on a particular subquotient does not change any crossing of this strand, doing such an isotopy for all lowering generators sending $-1$ to $i$ and for fixed right-side idempotent $\ee{\sss}$ yields a dg left module isomorphic to $\AA(\PP)\ee{\sss\setminus\{i\}}$ with the grading shift as described.  From this description, (2) follows easily too.

For (3) and (4), we can factor $F'$ as
\begin{equation*}
F'_0 \otimes_{\bfk} \left(\xymatrix{0 \ar[r] & \bfk \ar[r]^-{1} & \bfk \ar[r] & 0}\right),
\end{equation*}
where $F'_0$ is the span of all lowering generators in $F'$ with $x(-1) < x(0)$.  Hence $F'$ is acyclic.  The description of $F''$ follows from the obvious isotopy.

The analogous results for $\EE_k$ follow immediately, as the bimodule $\EE_k(P)$ is by definition the opposite bimodule to $\FF_k(-P)$.
\end{proof}

\begin{corollary}\label{cor:EF-cofibrant}
 As a dg left module over $\AA_{k-1}$, $\FF_k$ is cofibrant.  As a dg right module over $\AA_{k-1}$, $\EE_k$ is cofibrant.\end{corollary}
\begin{proof} By Proposition \ref{prop-bimodule-E}, $\FF_k$ admits a left dg module filtration with cofibrant subquotients.  This filtration splits when the dg structure is forgotten.  By \cite[Section 3]{KellerDeriving}, then, $\FF_k$ satisfies Property (P) and is thus cofibrant.
\end{proof}

We can now compute $[\FF(\PP)\dtensor{\AA(\PP)}-]$. Since $\AA\ee{\sss}$ is projective, we have that $\FF_k \dtensor{\AA(\PP)} A\ee{\sss} \simeq \FF_k \otimes_{\AA(\PP)} A\ee{\sss} = \FF_k\ee{\sss}$. In other words, the action 
\[[\FF_k\dtensor{\AA_k}-] : K_0(\AA(\PP)) \to K_0(\AA(\PP))\]
is given by 
\[[\FF_k\dtensor{\AA_k}-] : [\AA\ee{\sss}]\mapsto [\FF_k\ee{\sss}].\]
Writing $\sss=\{s_1 < \ldots < s_k\}\subseteq[n]$ and using the filtration from Proposition~\ref{prop-bimodule-E} (2), we have 
\[
[\FF_k\ee{\sss}] = [F^{1}_\sss] + [F^2_\sss / F^{1}_\sss] + \cdots + [F^k_\sss / F^{k-1}_\sss] 
= \sum_{i=1}^k[\AA_{k-1}\ee{\sss\setminus\{s_i\}}[-M(y_{s_i,\ttt\setminus\{s_i\}})]\{-2A(y_{s_i,\sss\setminus\{s_i\}})\}].
\]
Observe that $M(y_{s_i,\ttt\setminus\{s_i\}}) = i-1-s_i^+$ and $2A(y_{s_i,\sss\setminus\{s_i\}}) = s_i^- - s_i^+$, where $s_i^+$ is the number of plusses in the first $s_i$ elements of $P$, and $s_i^-$ is the number of minuses. 
So
\begin{equation}\label{eqn-f-matrix}
[\FF(\PP)\dtensor{\AA(\PP)}-]_{\ttt, \sss}=\begin{cases}
(-1)^{i-1-s_i^+}q^{s_i^- - s_i^+} & \text{if } \ttt = \sss\setminus \{s_i\}, \\
0 & \text{otherwise.}
\end{cases}
\end{equation}

The easiest way to describe the general pattern of the matrix for $[\FF(\PP)\dtensor{\AA(\PP)}-]$ and compare it to the matrix from  Subsection~\ref{subsec-gloneone} is by induction on the length of $\PP$.  Recall from the end of Subsubsection \ref{subsubsec-our-rep} that the reverse lexicographic ordering on subsets of $[n]$ has the following property: if we write the sequence of subsets as $\sss_1,\sss_2,\ldots,\sss_{2^{n+1}}$, then the first half ($\sss_1$ through $\sss_{2^n}$) is the sequence of subsets for $[n-1]$, and for $1 \leq i \leq 2^n$, we have $\sss_{2^n+i} = \sss_i \sqcup \{n\}$.

For $\PP = ()$, the empty sequence, we can directly see that $\FF\ee{\emptyset} = 0$ and $\FF\ee{\{0\}} \cong \AA\ee{\{0\}}$, so 
\[
[\FF(())] = \begin{pmatrix} 0 & 1 \\ 0 & 0 \end{pmatrix}.
\]

Let $\PP'$ be the subsequence of $\PP$ consisting of all but the last element.  Observe that for $s_i<n$, the quantities $s_i^+$ and $s_i^-$ are the same with respect to both sequences, so $[\FF(\PP)]$ has block form
\[[\FF(\PP)]=\left(
\begin{array}{c|c}
[\FF(P')] & D_{\PP}  \\ 
\hline
0 & [\FF(P')]
\end{array}\right).\]
By Equation~\ref{eqn-f-matrix}, the matrix $D_{\PP}$ is diagonal, with entry corresponding to $(\sss\setminus\{n\}, \sss)$ given by $(-1)^{|\sss|-1-p_+}q^{p_- - p_+}$, where $p_+$ is the number of plusses in $P$ and $p_-$ is the number of minuses.

The computation for $[\EE(\PP)\dtensor{\AA(\PP)}]$ is simpler. Since $\EE_k$ is right cofibrant, we can compute the induced functor $\EE_k\dtensor{\AA_k}-$ with an underived tensor product,
\begin{equation*}
\EE_k \dtensor{\AA_k} M \simeq \EE_k \otimes_{\AA_k} M.
\end{equation*}
In particular, $\EE_k \otimes_{\AA(\PP)} A\ee{\sss} = \EE_k\ee{\sss}$, so the action 
\[[\EE_k\dtensor{\AA_k}-] : K_0(\AA(\PP)) \to K_0(\AA(\PP))\]
is given by 
\[[\EE_k\dtensor{\AA_k}-] : [\AA\ee{\sss}]\mapsto [\EE_k\ee{\sss}].\]
By Proposition~\ref{prop-bimodule-E} (4) (the version for $\EE$), if $0\in \sss$, then $\EE_k\ee{\sss}$ is acyclic, so $[\EE_k\ee{\sss}]=0$, and if $0\notin \sss$, then $\EE_k\ee{\sss} = \AA_k\ee{\sss\sqcup \{0\}}$, so $[\EE_k\ee{\sss}] = [\AA_k\ee{\sss\sqcup \{0\}}]$. Thus, 
\[
[\EE(\PP)\dtensor{\AA(\PP)}-]_{\sss, \ttt}=\begin{cases}
1 & \text{if } \sss = \ttt\sqcup \{0\}, \\
0 & \text{otherwise.}
\end{cases}
\]
or, recursively, 
\[
[\EE(())] = \begin{pmatrix} 0 & 0 \\ 1 & 0 \end{pmatrix}, \quad
[\EE(\PP)]=\left(
\begin{array}{c|c}
[\EE(P')] & 0  \\ 
\hline
0 & [\EE(P')]
\end{array}\right).
\]

The above discussion and the computations of Subsection \ref{subsec-gloneone} imply the following.

\begin{corollary} The matrices of $[\EE(\PP)],[\FF(\PP)]$ with respect to the primitive basis equal the matrices of $E,F \in U_q$ acting on $V_{\PP} \otimes L(\lambda_{n+1})$ with respect to the basis of Subsection \ref{subsec-gloneone}.\qed
\end{corollary}

This completes the proof of the first part of Theorem B. We continue with proving the categorified relations.

\begin{proposition}\label{prop:ef}
Let $\PP$ be a sign sequence of length $n$. We have homotopy equivalences 
\begin{align*}
\FF(\PP)\dtensor{\AA(\PP)} \FF(\PP)&\simeq 0,\\
\EE(\PP)\dtensor{\AA(\PP)} \EE(\PP)&\simeq 0,
\end{align*} 
and an exact triangle 
\begin{equation*}
\xymatrix{
\EE(\PP) \dtensor{\AA(\PP)} \FF(\PP) \ar[r] & \AA(\PP) \ar[r] & \FF(\PP)\dtensor{\AA(\PP)} \EE(\PP) \ar[r] & \EE(\PP) \dtensor{\AA(\PP)} \FF(\PP)[1],
}
\end{equation*}
Further, for any tangle $\T$ and any decomposition $\Tdec$ of $\T$,
\begin{equation*}\begin{split}
& \EE(-\bdy^0\T) \boxtimes \ctt(\Tdec) \simeq \AA(-\bdy^0\T) \boxtimes \ctt(\Tdec) \dtensor{\AA(\bdy^1\T)} \EE(\bdy^1\T), \\
& \FF(-\bdy^0\T) \boxtimes \ctt(\Tdec) \simeq \AA(-\bdy^0\T) \boxtimes \ctt(\Tdec) \dtensor{\AA(\bdy^1\T)} \FF(\bdy^1\T)
\end{split}\end{equation*}
as type $\tAA$ bimodules over $(\AA(-\bdy^0\T), \AA(\bdy^1\T))$.
\end{proposition}

\begin{proof}
For the first two equivalences and the exact triangle, we fix a sign sequence $P$ of length $n$, and denote $\EE(\PP), \FF(\PP), \AA(\PP)$ by $\EE, \FF, \AA$, respectively. 

Since $\FF$ is left cofibrant, we have 
\[\FF\dtensor{\AA} \FF \simeq \FF \otimes_{\AA} \FF.\]
 We can represent the underived tensor product $\FF\otimes_{\AA} \FF$ graphically by concatenating diagrams. Let $x_1, x_2:\{-1\}\sqcup[n]\to[n]$ be generators of $\FF$ such that $y=(x_1, x_2)$ is nonzero in $\FF\otimes_{\AA} \FF$ and such that  $x_2\circ x_1(-1) < x_2(-1)$.  This means the respective black strands in a diagram for $(x_1, x_2)$ cross, or, equivalently, the black strands in a diagram for $x_2$ cross. Let $y_2$ be the generator whose diagram is obtained from that of $x_2$ by resolving that crossing. Then $f(x)=(x_1, y_2)$ is the generator whose diagram is obtained by resolving the respective crossing in $(x_1, x_2)$, see Figure \ref{fig:cancelF}. Note that $f(x)$ is a term in the differential of $x$, and $f$ is a bijection from those generators for $\FF\otimes_{\AA} \FF$ for which the black strands starting at height $-1$ cross to the remaining generators. Taking this component $f$ of the differential on  $\FF\otimes_{\AA} \FF$ cancels out all pairs of generators, and shows that $\FF\dtensor{\AA} \FF\simeq 0$. 
 \begin{figure}[h]
 \centering
 \includegraphics[scale = .7]{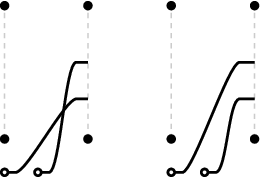} 
 \vskip .1 cm
       \caption{The differential cancels generators of $\FF\otimes \FF$ in pairs.}\label{fig:cancelF}
\end{figure}

Similarly, $\EE\dtensor{\AA} \EE\simeq 0$. 

Next, we turn to the exact triangle. 

Since $\EE$ is right cofibrant and $\FF$ is left cofibrant, $\EE\dtensor{\AA} \FF\simeq \EE\otimes_{\AA} \FF$. Diagrammatically, one can represent $\EE\otimes_{\AA} \FF$ by concatenating diagrams for $\EE$ and $\FF$, see Figure \ref{fig:EF_AA}. We can think of a strand diagram for a generator of  $\EE\otimes_{\AA}\FF$ as a partial bijection $x:\{-1\}\sqcup [n]\to \{-1\}\sqcup [n]$ with $-1$ both in the domain and range and $x(-1)\neq -1$.

 \begin{figure}[h]
 \centering
   \labellist
    	 \pinlabel $\otimes$ at 58 49
	 \pinlabel $=$ at 130 49
	 \pinlabel $\otimes$ at 319 49
	 \pinlabel $=$ at 391 49
   \endlabellist
 \includegraphics[scale = .7]{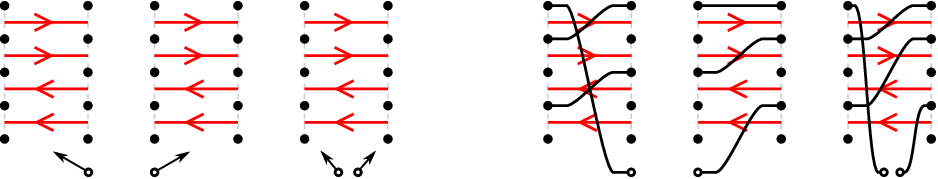} 
 \vskip .1 cm
       \caption{The bimodule $\EE\otimes_{\AA}\FF$, when $\PP= (-, -, +, +)$. Left: The background diagram for $\EE\otimes_{\AA}\FF$ is obtained by concatenating the background diagrams for $\EE$ and $\FF$. Right: an example of a nonzero generator of $\EE\otimes_{\AA}\FF$.}\label{fig:EF_AA}
\end{figure}
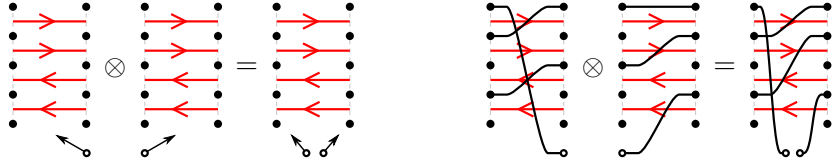

Define a map $f: \EE\otimes_{\AA}\FF\to \AA$ on generators as follows. If $x:\{-1\}\sqcup [n]\to \{-1\}\sqcup [n]$ is a generator for $\EE\otimes_{\AA}\FF$ with $x(-1)=0$ or $x(0) = -1$, then $f(x) = y$, where $y(x^{-1}) = x(-1)$ and $y|_{[n]\setminus\{x^{-1}(-1)\}} = x|_{[n]\setminus\{x^{-1}(-1)\}}$. Otherwise, $f(x) = 0$. For example, $f$ takes the generator shown at the right of Figure~\ref{fig:EF_AA} to $0$. See Figure \ref{fig:fEFA} for an example where $f(x)\neq 0$.

 \begin{figure}[h]
 \centering
   \labellist
    	 \pinlabel $f$ at 70 62
	  \pinlabel $\mapsto$ at 70 49
   \endlabellist
 \includegraphics[scale = .7]{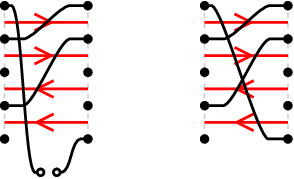} 
 \vskip .1 cm
       \caption{An example of the map $f$.}\label{fig:fEFA}
\end{figure}

Observe that the cone of $f$, $C(f)$, is generated over $\Ft$ by partial bijections $x:\{-1\}\sqcup [n]\to \{-1\}\sqcup [n]$ with $-1$ both in the domain and range, with bimodule structure as follows. The diagram associated to a generator $x$ for $C(f)$ consists of $n$ horizontal red strands from $(0, i-\frac 1 2 )$ to $(1, i-\frac 1 2)$ for $1\leq i\leq n$ (oriented according to $\PP$) and a black strand connecting $(0,i)$ to $(1, x(i))$ for each $i$ in the domain of $x$. The left and right algebra actions are given by concatenation, and the differential is given by resolving crossings, subject to the same relations as for the algebra. See Figure \ref{fig:FE_AA}.

 \begin{figure}[h]
 \centering
 \includegraphics[scale = .7]{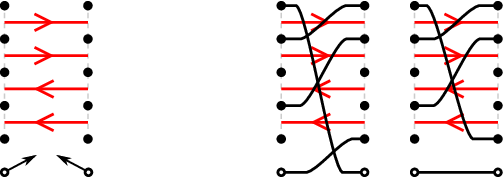} 
 \vskip .1 cm
       \caption{The bimodule $C(f)$ in the case $\PP= (-, -, +, +)$. Left: the background diagram for $C(f)$. Right: the generators for $C(f)$ corresponding to the two generators in Figure \ref{fig:fEFA}.}\label{fig:FE_AA}
\end{figure}

Below, we argue that $\FF\dtensor{\AA} \EE\simeq C(f)$.

By \cite[Theorem 5.4]{pv},
$\AA\boxtimes \ctt(\triv_{\PP})\simeq \AA\otimes (\Ft\otimes \Ft[1]\{2\})^{\otimes n}$, so 
\begin{align*}
(\FF\dtensor{\AA} \EE)\otimes (\Ft\otimes \Ft[1]\{2\})^{\otimes n} &\simeq \FF\dtensor{\AA} (\AA\boxtimes \ctt(\triv_{\PP}))\dtensor{\AA} \EE\\
&= ((\FF\dtensor{\AA} \AA)\boxtimes \ctt(\triv_{\PP}))\dtensor{\AA} \EE\\
&= (\FF \boxtimes \ctt(\triv_{\PP}))\dtensor{\AA} \EE\\
&= \FF \boxtimes( \ctt(\triv_{\PP})\dtensor{\AA} \EE)\\
&= \FF \boxtimes( \ctt(\triv_{\PP})\otimes_{\AA} \EE),
\end{align*}
where the last equality holds since $\ctt(\triv_{\PP})$ is both left and right cofibrant.

Given a generator $x$ of $\ctt(\triv_{\PP})$ and a generator $y$ of $\EE$, one can interpret $x\otimes y$ as a concatenation of diagrams by placing a diagram for $y$ to the right of a diagram for $x$ and scaling the right piece for $x$ and the diagram for $y$ by $\frac 1 2$. See Figure \ref{fig:idtensorE}.
 \begin{figure}[h]
 \centering
    \labellist
    	 \pinlabel $\otimes$ at 97 49
	 \pinlabel $=$ at 174 49
	 \pinlabel $\otimes$ at 435 49
	 \pinlabel $=$ at 510 49
   \endlabellist
 \includegraphics[scale = .7]{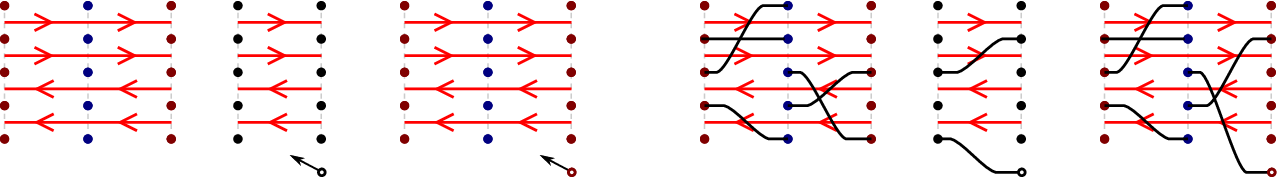} 
 \vskip .1 cm
       \caption{The $\tDA$ bimodule  $\ctt(\triv_{\PP})\otimes_{\AA} \EE$, when $\PP= (-, -, +, +)$. Left: The background diagram for $\ctt(\triv_{\PP})\otimes_{\AA} \EE$ is obtained by concatenating the even (right) piece of the background diagram for $\ctt(\triv_{\PP})$ with the background diagram for $\EE$. Right: The tensor product $(x_1, x_2)\otimes x_3$ of generators is given by concatenating a diagram for $x_2$ with a diagram for $x_3$.}\label{fig:idtensorE}
\end{figure}
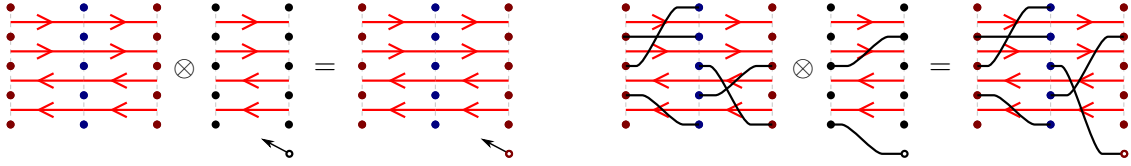

Thus, we can represent $\FF \boxtimes( \ctt(\triv_{\PP})\otimes_{\AA} \EE)$ diagrammatically as follows. Place the background diagram for $\FF$ to the left of the odd piece for $\triv_{\PP}$ and the background diagram for $\EE$ to the right of the odd piece for $\triv_{\PP}$. Generators are sequences $(x_1,x_2,x_3)$ of partial bijections such that
\begin{itemize} 
  \item $x_1:\{-1\}\sqcup [n]\to [n]$ with $-1$ in the domain of $x_1$,
  \item $x_2:[n]\to [n]$,
  \item $x_3:[n]\to \{-1\}\sqcup [n]$ with $-1$ in the range of $x_3$, and
  \item for $i=1,2$, every element of $[n]$ is either in the domain of $x_{i+1}$ or in the range of $x_i$ but not both.
\end{itemize}
We draw diagrams for generators of $\FF \boxtimes( \ctt(\triv_{\PP})\otimes_{\AA} \EE)$ as we have done for $\AA(\PP)$, $\ctt(\Tdec)$, and so forth. The differential is defined to be the sum $\bdy_+ + \bdy_- + \bdy_m$, with each summand defined by the same conditions as the corresponding differential summand for $\ctt(\Tdec)$.  The algebra action is the evident one defined in analogy with the actions on $\EE$ and $\FF$; it appears diagrammatically as a modified concatenation. See Figure \ref{fig:FidE}.
 \begin{figure}[h]
 \centering
 \includegraphics[scale = .7]{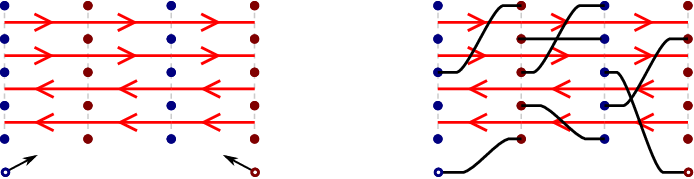} 
 \vskip .1 cm
       \caption{The $\tAA$ bimodule  $\FF\boxtimes(\ctt(\triv_{\PP})\otimes_{\AA} \EE)$, when $\PP= (-, -, +, +)$. Left: the background diagram. Right: an example of a generator.}\label{fig:FidE}
\end{figure}

Next we show that $\FF \boxtimes( \ctt(\triv_{\PP})\otimes_{\AA} \EE)$ is homotopy equivalent to $\AA\boxtimes \ctt(\triv_{\PP})\otimes_{\AA} C(f)$. To do so, and to prove the last two equivalences in Theorem B, we appeal to the interpretation of our algebras and bimodules in terms of Heegaard diagrams (see \cite[Section 8]{pv}, for example). Below, we briefly recall some basics about Heegaard diagrams, omitting technical details, and specializing to our needs.

Most generally, a bordered Heegaard diagram consists of a compact surface $\Sigma$ with one or more boundary components, along with two sets of pairwise disjoint, properly embedded curves (circles and/or arcs) on the surface, typically denoted as $\alpha$-curves and $\beta$-curves (drawn in red and blue below, respectively), and some additional arcs and/or points on the surface, typically denoted as $\zs$-arcs (drawn in green below), and $X$- and $O$-basepoints, satisfying certain conditions. Depending on which type of curves touch the boundary, the Heegaard diagram is called $\alpha$-$\alpha$ bordered, $\beta$-$\alpha$ bordered, and so on. A bordered Heegaard diagram specifies a topological object with boundary, such as a bordered $3$-manifold, a bordered-sutured $3$-manifold, or a tangle in a (bordered) $3$-manifold. See \cite{bfh2, bimod, hfmor, bs, pv}, among others.

A bordered Heegaard diagram is called admissible, if a certain technical condition is satisfied. We do not need the definition here, but we remark that every diagram can be ``converted" to an admissible diagram that describes the same topological object, and that all diagrams described in this paper are admissible. 

To the boundaries of an admissible bordered Heegaard diagram one can associate dg algebras, and to the diagram various bimodules (type $\tAA$, $\tDD$, etc.) over these algebras. The bimodules are generated by sets of intersection points between $\alpha$-curves and $\beta$-curves, with one point on each circle, and at most one point on each arc; the structure maps are defined by counting certain pseudo-holomorphic curves in $\Sigma\times I\times \R$; box tensor product corresponds to gluing diagrams along some common boundary. If two bordered Heegaard diagrams are related by a sequence of the usual Heegaard moves (isotopies, index one/two (de)stabilizations and handle slides of circles over arcs or circles), then the corresponding bimodules are homotopy equivalent, whereas an index zero/three stabilization results in tensoring with a two-dimensional vector space $\Ft\otimes \Ft[1]\{2\}$.

Next, we describe more precisely some specific types of bordered Heegaard diagrams relevant to this paper. 

In \cite{pv}, there are descriptions of $\beta$-$\alpha$ and $\alpha$-$\alpha$ diagrams for the algebras $\AA(\PP)$ and tangle decompositions $\Tdec$ respectively, such that the associated generators and structure maps are in bijection with those for the strand diagrams described in this paper. We give an outline below.

We describe the $\beta$-$\alpha$ diagram $\H_{\AA(\PP)}$ for $\AA(\PP)$; see the left diagram of Figure~\ref{fig:efa}, for example. This diagram consists of $n+1$ parallel $\alpha$-arcs and $n+1$ parallel $\beta$-arcs on a genus zero surface $\Sigma$ with three boundary components, i.e. on a pair of pants, which form a grid. Label the $\alpha$-arcs $\alpha_0, \ldots, \alpha_n$ from bottom to top, and the $\beta$-arcs $\beta_0, \ldots, \beta_n$ from right to left, as seen in Figure~\ref{fig:efa}. In the square formed by $\alpha_{i-1}$, $\alpha_i$, $\beta_{i-1}$, $\beta_i$, place an $O$ if $p_i=1$ or an $X$ if $p_i=-1$. The bijection between Heegaard diagram and strand diagram generators is defined by the bijection between intersection points and black strands sending the point $\alpha_i\cap \beta_j$ to the strand connecting $(0,j)$ to $(1,i)$. The differential is defined by counting empty rectangles (in bijection with resolving black-black crossings), and the algebra action by counting sets of partial rectangles (in bijection with concatenating with strand diagrams for algebra generators); see \cite[Chapter 4]{pv}, where these counts are described in detail. 

Given a tangle decomposition $\Tdec$, the $\alpha$-$\alpha$ diagram $\H_{\Tdec}$ and the correspondence of its generators and structure maps with the  strand diagrams ones are defined in a similar way.  The precise description is rather lengthy, and not needed for the arguments in the remaining proof, so we direct the interested reader to \cite[Chapter 4]{pv}, and illustrate an example in  Figure~\ref{fig:hd-strand}.

\begin{figure}[h]
 \centering
  \labellist
    \pinlabel  \textcolor{green}{$\mathbf{z}_1$} at 100 100
    \pinlabel  \textcolor{green}{$\mathbf{z}_2$} at 100 -5
     \pinlabel  \textcolor{red}{$\alphas^1$} at 215 50
     \pinlabel  \textcolor{red}{$\alphas^c$} at 110 22
     \pinlabel  \textcolor{red}{$\alphas^0$} at 7 50
         \pinlabel  \textcolor{blue}{$\betas$} at 30 10
       \endlabellist
 \includegraphics[scale = 1]{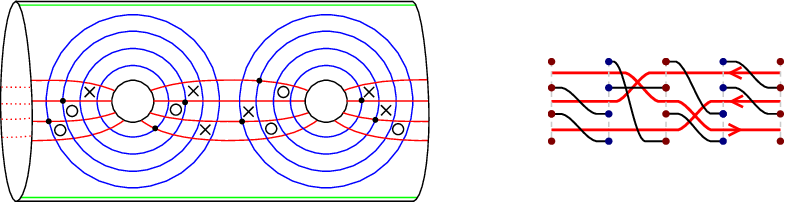} 
 \vskip .1 cm
       \caption{Left: an $\alpha$-$\alpha$ bordered Heegaard diagram for a tangle decomposition, along with a generator.  Right: the corresponding tangle decomposition and generator.}\label{fig:hd-strand}
\end{figure}
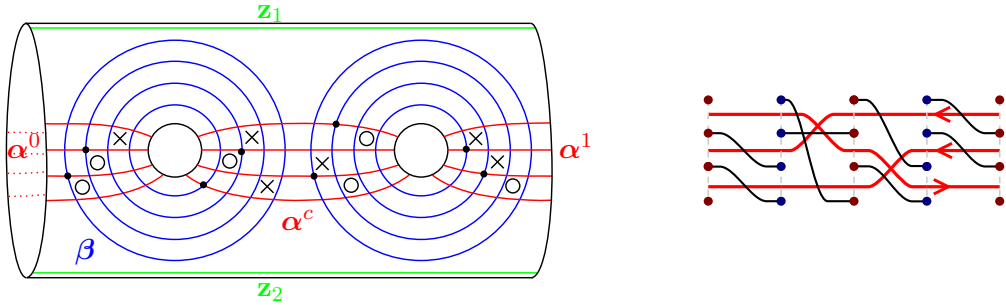

Similarly, below we construct $\beta$-$\alpha$ bordered Heegaard diagrams $\H_{\EE(\PP)}$ and $\H_{\FF(\PP)}$, so that the generators and structure maps for the associated $\tAA$ bimodules are in bijection with those in the strand diagram descriptions for $\EE(\PP)$ and $\FF(\PP)$. These diagrams combine properties of diagrams seen in \cite{bimod, bs, pv} and the definition and properties of bordered Floer homology generalize to them in a straight forward way.

We start with $\H_{\AA(\PP)}$. To get $\H_{\EE(\PP)}$, we add a $1$-handle to $\Sigma$ with feet just  bottom-left and bottom-right of the grid, as well as an $\alpha$-circle, denoted $\alpha_{-1}$, going once over the handle. To get $\H_{\FF(\PP)}$, we add a $1$-handle to $\Sigma$ with feet just  bottom-right and top-right of the grid, as well as a $\beta$-circle, denoted $\beta_{-1}$, going once over the handle. See the middle and right diagrams in Figure~\ref{fig:efa}, for example.  The bijection between Heegaard diagram and strand diagram generators and structure maps is analogous to that for $\H_{\AA(\PP)}$. For the curious reader, we remark that the topological objects described by these diagrams are certain trivial tangles in $S^2\times I$ with a solid torus drilled out. 

  \begin{figure}[h]
 \centering
   \labellist
    \pinlabel  \textcolor{green}{$\mathbf{z}_1$} at 30 105
    \pinlabel  \textcolor{green}{$\mathbf{z}_2$} at 30 -5
    \pinlabel  \textcolor{red}{$\alpha_0$} at 73 35
     \pinlabel  \textcolor{red}{$\alpha_1$} at 73 48
     \pinlabel  \textcolor{red}{$\alpha_2$} at 73 60
     \pinlabel  \textcolor{blue}{$\beta_0$} at -5 6
     \pinlabel  \textcolor{blue}{$\beta_1$} at -5 17
     \pinlabel  \textcolor{blue}{$\beta_2$} at -5 28
       \endlabellist
  \includegraphics[scale = 1]{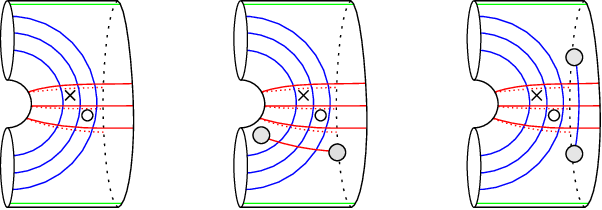} 
 \vskip .1 cm
       \caption{Left: the $\beta$-$\alpha$ diagram $\H_{\AA((+,-))}$ for the algebra $\AA((+,-))$. Middle: The diagram $\H_{\EE((+,-))}$. Right: The diagram $\H_{\FF((+,-))}$.}\label{fig:efa}
\end{figure}

We now return to the proof of Proposition~\ref{prop:ef}.

The combinatorics (i.e. generators and structure maps) of the strand diagram presentation of the bimodule $\FF \boxtimes( \ctt(\triv_{\PP})\otimes_{\AA} \EE)$ described earlier correspond to the combinatorics in the left Heegaard diagram of Figure \ref{fig:EFhd}. The $\beta$-circle that goes over the dark grey $1$-handle can be slid over the outermost $\beta$-circle to its right, to produce the second Heegaard diagram in Figure \ref{fig:EFhd}. 
The handle slide induces a homotopy equivalence between the structures associated to the two Heegaard diagrams.  The combinatorics of the right Heegaard diagram of Figure \ref{fig:EFhd} correspond to the combinatorics of the strand diagram presentation for the bimodule $\AA\boxtimes \ctt(\triv_{\PP})\otimes_{\AA} C(f)$.

 \begin{figure}[h]
 \centering
   \labellist
   \pinlabel  $\rotatebox{130}{\tiny$\boldsymbol{\dots}$}$ at 134 50
   \pinlabel  $\rotatebox{53}{\tiny$\boldsymbol{\dots}$}$ at 77 50
   \pinlabel  $\rotatebox{130}{\tiny$\boldsymbol{\dots}$}$ at 32 50
      \pinlabel  $\rotatebox{130}{\tiny$\boldsymbol{\dots}$}$ at 326 50
   \pinlabel  $\rotatebox{53}{\tiny$\boldsymbol{\dots}$}$ at 269 50
   \pinlabel  $\rotatebox{130}{\tiny$\boldsymbol{\dots}$}$ at 224 50
    \pinlabel  \textcolor{blue}{$\rotatebox{45}{\tiny$\dots$}$} at 129 70
    \pinlabel  \textcolor{blue}{$\rotatebox{45}{\tiny$\dots$}$} at 26 70
  \pinlabel  \textcolor{red}{$\rotatebox{90}{\tiny$\dots$}$} at 57 50
    \pinlabel  \textcolor{red}{$\rotatebox{90}{\tiny$\dots$}$} at 157 50
        \pinlabel  \textcolor{blue}{$\rotatebox{45}{\tiny$\dots$}$} at 320 70
    \pinlabel  \textcolor{blue}{$\rotatebox{45}{\tiny$\dots$}$} at 217 70
  \pinlabel  \textcolor{red}{$\rotatebox{90}{\tiny$\dots$}$} at 247 50
    \pinlabel  \textcolor{red}{$\rotatebox{90}{\tiny$\dots$}$} at 347 50
    \pinlabel \footnotesize{$\underbrace{\phantom{aaaaaa}}$} at 28 -7
        \pinlabel \footnotesize{$\underbrace{\phantom{aaaaaaaaaaaaa}}$} at 104 -7
            \pinlabel \footnotesize{$\underbrace{\phantom{aaaaaa}}$} at 221 -7
        \pinlabel \footnotesize{$\underbrace{\phantom{aaaaaaaaaaaaa}}$} at 297 -7
    \pinlabel  \footnotesize{$\FF$} at 28 -27
    \pinlabel \footnotesize{$\ctt(\triv_{\PP})\otimes_{\AA} \EE$} at 110 -27
    \pinlabel  \footnotesize{$\AA$} at 215 -27
    \pinlabel \footnotesize{$\ctt(\triv_{\PP})\otimes_{\AA} C(f)$} at 300 -27
       \endlabellist
 \includegraphics[scale = .7]{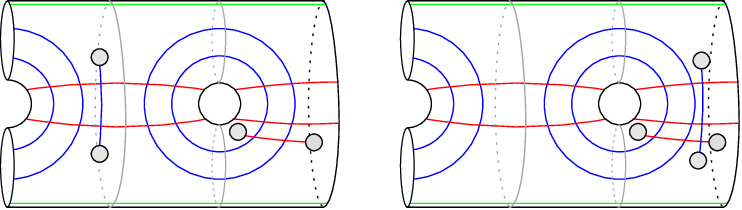} 
 \vskip .7 cm
       \caption{Left: a Heegaard diagram for the bimodule $\FF \boxtimes( \ctt(\triv_{\PP})\otimes_{\AA} \EE)$. Right: a Heegaard diagram for the bimodule  $\AA\boxtimes \ctt(\triv_{\PP})\otimes_{\AA} C(f)$.}\label{fig:EFhd}
\end{figure}

In terms of strand diagrams, the handle slide can be thought of as moving the dot at $(0,-1)$ two units to the right. See Figure \ref{fig:FEM}.
 \begin{figure}[h]
 \centering
 \includegraphics[scale = .7]{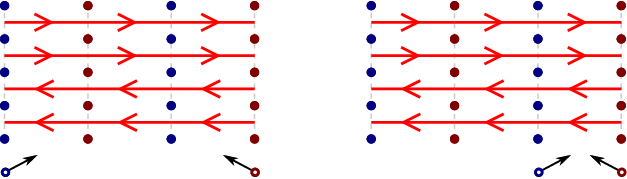} 
 \vskip .1 cm
       \caption{Left: the background diagram for the bimodule $\FF \boxtimes( \ctt(\triv_{\PP})\otimes_{\AA} \EE)$. Right: the background diagram for the bimodule  $\AA\boxtimes \ctt(\triv_{\PP})\otimes_{\AA} C(f)$. In this example, $\PP=(-,-,+,+)$.}\label{fig:FEM}
\end{figure}

Now, since $\AA\boxtimes \ctt(\triv_{\PP})\simeq \AA\otimes (\Ft\otimes \Ft[1]\{2\})^{\otimes n}$, we have $\AA\boxtimes \ctt(\triv_{\PP})\otimes_{\AA} C(f)\simeq \AA\otimes_{\AA} C(f)\otimes(\Ft\otimes \Ft[1]\{2\})^{\otimes n}\simeq C(f)\otimes (\Ft\otimes \Ft[1]\{2\})^{\otimes n}$.

Last, we observe there is a triangle
\begin{equation*}
\xymatrix{
\EE \otimes_{\AA} \FF \ar[r]^-{f} & \AA \ar[r] & C(f) \ar[r] & \EE \otimes_{\AA} \FF[1].
}
\end{equation*}
Replacing objects with equivalent ones, we get a triangle
\begin{equation*}
\xymatrix{
\EE \dtensor{\AA} \FF \ar[r] & \AA \ar[r] & \FF\dtensor{\AA} \EE \ar[r] & \EE \dtensor{\AA} \FF[1].
}
\end{equation*}

Last, we show that $\EE$ and $\FF$ commute with $\ctt$. We may assume that the rightmost elementary tangle of $\Tdec$ is trivial, so that $\ctt(\Tdec)$ is automatically right cofibrant: For any decomposition $\Tdec$ of a tangle $\T$, we know that if $\Tdec'$ is obtained from $\Tdec$ by concatenating with a trivial elementary tangle to the right, then $\ctt(\Tdec') \simeq \ctt(\Tdec)\otimes (\Ft\otimes \Ft[1]\{2\})^{\otimes |\bdy^1\T|}$. So if $\EE$ and $\FF$ commute with $\ctt(\Tdec')$, as in the statement of the proposition, then it follows they commute with $\ctt(\Tdec)$.

So assume $\Tdec$ is a decomposition for a tangle $\T$ and the rightmost elementary tangle in $\Tdec$ is trivial. We use a Heegaard diagram interpretation of the bimodules, as above. The combinatorics of the bimodule $\FF(-\bdy^0\T) \boxtimes \ctt(\Tdec)$ correspond to the combinatorics in the top left Heegaard diagram of Figure \ref{fig:EFT}. The $\beta$-circle that goes over the dark grey $1$-handle can be slid over the outermost $\beta$-circle corresponsing to each elementary tangle, one by one, to produce the top right Heegaard diagram in Figure \ref{fig:EFT}. These
handle sides induce homotopy equivalences between the algebraic structures for the respective Heegaard diagrams. Since the bimodule for the rightmost elementary tangle is cofibrant, the combinatorics of the top right Heegaard diagram correspond to the combinatorics for $\AA(-\bdy^0\T) \boxtimes \ctt(\Tdec) \otimes_{\AA(\bdy^1\T)} \FF(\bdy^1\T)$, which is equivalent to  $\AA(-\bdy^0\T) \boxtimes \ctt(\Tdec) \dtensor{\AA(\bdy^1\T)} \FF(\bdy^1\T)$. The argument for $\EE$ is analogous, see the bottom of Figure \ref{fig:EFT}.
 \begin{figure}[h]
 \centering
  \labellist
    \pinlabel  \textcolor{gray}{$\dots$} at 170 50
    \pinlabel  \textcolor{gray}{$\dots$} at 170 195
    \pinlabel  \textcolor{gray}{$\dots$} at 490 50
    \pinlabel  \textcolor{gray}{$\dots$} at 490 195
    \pinlabel  \footnotesize{$\FF(-\bdy^0\T)$} at 28 127
    \pinlabel \footnotesize{$\ctt(\T_1)$} at 105 127
    \pinlabel \footnotesize{$\dots$} at 164 127
    \pinlabel \footnotesize{$\ctt(\triv_{\bdy^1\T})$} at 235 127
    \pinlabel  \footnotesize{$\AA(-\bdy^0\T)$} at 347 127
    \pinlabel \footnotesize{$\ctt(\T_1)$} at 424 127
    \pinlabel \footnotesize{$\dots$} at 479 127
    \pinlabel \footnotesize{$\ctt(\triv_{\bdy^1\T})\otimes \FF(\bdy^1\T)$} at 558 127
        \pinlabel \footnotesize{$\underbrace{\phantom{aaaaaa}}$} at 30 145
        \pinlabel \footnotesize{$\underbrace{\phantom{aaaaaaaaaaa}}$} at 98 145
         \pinlabel \footnotesize{$\underbrace{\phantom{aaaaaaaaaaaaa}}$} at 234 145
                 \pinlabel \footnotesize{$\underbrace{\phantom{aaaaaa}}$} at 349 145
        \pinlabel \footnotesize{$\underbrace{\phantom{aaaaaaaaaaa}}$} at 418 145
         \pinlabel \footnotesize{$\underbrace{\phantom{aaaaaaaaaaaaa}}$} at 554 145
       \endlabellist
 \includegraphics[scale = .7]{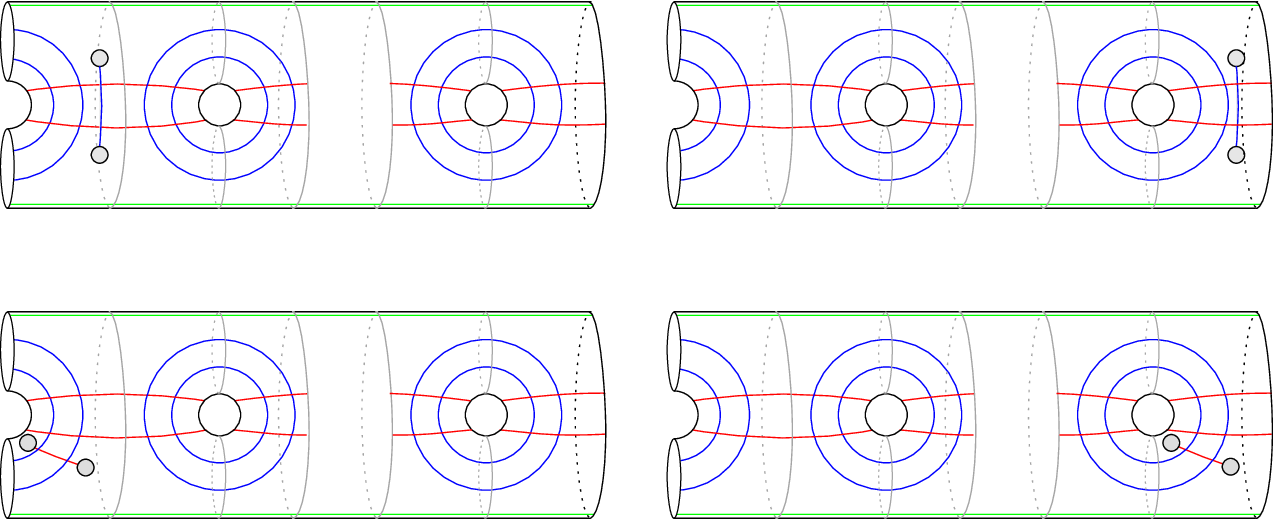} 
 \vskip .1 cm
       \caption{Top: the first and last diagram in a sequence of handle slides,  showing that $\FF(-\bdy^0\T) \boxtimes \ctt(\Tdec) \simeq \AA(-\bdy^0\T) \boxtimes \ctt(\Tdec) \dtensor{\AA(\bdy^1\T)} \FF(\bdy^1\T)$ (pieces of the Heegaard diagram are labeled with their respective bimodules). Bottom: the first and last diagram in a sequence of handle slides, showing that $\EE(-\bdy^0\T) \boxtimes \ctt(\Tdec) \simeq \AA(-\bdy^0\T) \boxtimes \ctt(\Tdec) \dtensor{\AA(\bdy^1\T)} \EE(\bdy^1\T)$.}\label{fig:EFT}
\end{figure}
\end{proof}
This concludes the proof of Theorem B of the Introduction.


\bibliographystyle{alpha}

\bibliography{master}

\begin{thebibliography}{MOST07}

\bibitem[BK90]{BondalKapranov}
A.~Bondal and M.~Kapranov.
\newblock Enhanced triangulated categories.
\newblock {\em Mat. Sb.}, 181(5):669--683, 1990.
\newblock English version: \emph{Mathematics of the USSR-Sbornik}, 1991, 70
  (1):93--107.

\bibitem[BL94]{BL}
J.~Bernstein and V.~Lunts.
\newblock {\em Equivariant Sheaves and Functors}, volume 1578 of {\em Lecture
  Notes in Mathematics}.
\newblock Springer, 1994.

\bibitem[BN05]{BN}
D.~Bar-Natan.
\newblock Khovanov's homology for tangles and cobordisms.
\newblock {\em Geom. Topol.}, 9:1443--1499, 2005.
\newblock \href{http://arxiv.org/abs/math/0410495}{arXiv:math/0410495}.

\bibitem[BS11]{brst}
Jonathan Brundan and Catharina Stroppel.
\newblock Highest weight categories arising from {K}hovanov's diagram algebra
  {I}: cellularity.
\newblock {\em Mosc. Math. J.}, 11(4):685--722, 821--822, 2011.

\bibitem[CK14]{chkh}
Yanfeng Chen and Mikhail Khovanov.
\newblock An invariant of tangle cobordisms via subquotients of arc rings.
\newblock {\em Fund. Math.}, 225(1):23--44, 2014.

\bibitem[EQ16]{eq}
A.~P. Ellis and Y.~Qi.
\newblock The differential graded odd nil{H}ecke algebra.
\newblock {\em Comm. Math. Phys.}, 344(1):275--331, 2016.
\newblock \href{http://arxiv.org/abs/1504.01712}{arXiv:1504.01712}.

\bibitem[Jac04]{Jacobsson}
M.~Jacobsson.
\newblock An invariant of link cobordisms from {K}hovanov homology.
\newblock {\em Algebr. Geom. Topol.}, 4:1211--1251, 2004.
\newblock \href{http://arxiv.org/abs/math/0206303}{arXiv:math/0206303}.

\bibitem[Kel94]{KellerDeriving}
B.~Keller.
\newblock Deriving dg categories.
\newblock {\em Ann. Scient. \'{E}c. Norm. Sup.}, 27:63--102, 1994.

\bibitem[Kel06]{Keller}
B.~Keller.
\newblock On differential graded categories.
\newblock In {\em International Congress of Mathematicians (Madrid)},
  volume~II, pages 151--190. Eur. Math. Soc., 2006.

\bibitem[Kho02]{kh3}
M.~Khovanov.
\newblock A functor-valued invariant of tangles.
\newblock {\em Algebr. Geom. Topol.}, 2:665--741, 2002.
\newblock \href{http://arxiv.org/abs/math/0103190}{arXiv:math/0103190}.

\bibitem[Kho06]{KhCob}
M.~Khovanov.
\newblock An invariant of tangle cobordisms.
\newblock {\em Trans. Amer. Math. Soc.}, 358:315--327, 2006.
\newblock \href{http://arxiv.org/abs/math/0207264}{arXiv:math/0207264}.

\bibitem[Kho14]{Khgl12}
M.~Khovanov.
\newblock How to categorify one-half of quantum {${\mathit{gl}}(1|2)$}.
\newblock In {\em Knots in {P}oland {III}. {P}art {III}}, volume 103 of {\em
  Banach Center Publ.}, pages 211--232. Polish Acad. Sci. Inst. Math., Warsaw,
  2014.
\newblock \href{http://arxiv.org/abs/1007.3517}{arXiv:1007.3517}.

\bibitem[KL09]{KL1}
M.~Khovanov and A.~Lauda.
\newblock A diagrammatic approach to categorification of quantum groups {I}.
\newblock {\em Represent. Theory}, 13:309--347, 2009.
\newblock \href{http://arxiv.org/abs/0803.4121}{arXiv:0803.4121}.

\bibitem[KL10]{KL3}
M.~Khovanov and A.~Lauda.
\newblock A diagrammatic approach to categorification of quantum groups {III}.
\newblock {\em Quantum Topology}, 1(1):1--92, 2010.
\newblock \href{http://arxiv.org/abs/0807.3250}{arXiv:0807.3250}.

\bibitem[LOT11]{hfmor}
Robert Lipshitz, Peter~S. Ozsv\'{a}th, and Dylan~P. Thurston.
\newblock Heegaard {F}loer homology as morphism spaces.
\newblock {\em Quantum Topol.}, 2(4):381--449, 2011.

\bibitem[LOT15]{bimod}
R.~Lipshitz, P.~Ozsv{\'a}th, and D.~Thurston.
\newblock Bimodules in bordered {H}eegaard {F}loer homology.
\newblock {\em Geom. Topol.}, 19(2):525--724, 2015.

\bibitem[LOT18]{bfh2}
Robert Lipshitz, Peter~S. Ozsvath, and Dylan~P. Thurston.
\newblock Bordered {H}eegaard {F}loer homology.
\newblock {\em Mem. Amer. Math. Soc.}, 254(1216):viii+279, 2018.

\bibitem[Man07]{mskein}
C.~Manolescu.
\newblock An unoriented skein exact triangle for knot {F}loer homology.
\newblock {\em Math. Res. Lett.}, 14:839--852, 2007.
\newblock \href{http://arxiv.org/abs/math/0609531}{arXiv:math/0609531}.

\bibitem[MOS09]{mos}
C.~Manolescu, P.~Ozsv{\'a}th, and S.~Sarkar.
\newblock A combinatorial description of knot {F}loer homology.
\newblock {\em Ann. of Math. (2)}, 169(2):633--660, 2009.
\newblock \href{http://arxiv.org/abs/math/0607691}{arXiv:math/0607691}.

\bibitem[MOST07]{oszcomb}
C.~Manolescu, P.~Ozsv{\'a}th, Z.~Szab{\'o}, and D.~Thurston.
\newblock On combinatorial link {F}loer homology.
\newblock {\em Geom. Topol.}, 11:2339--2412, 2007.
\newblock \href{http://arxiv.org/abs/math/0610559}{arXiv:math/0610559}.

\bibitem[MSK09]{MSW}
S.~Morrison, D.~Scott, and Walker K.
\newblock Fixing the functoriality of {K}hovanov homology.
\newblock {\em Geom. Topol.}, (13):1499--1582, 2009.
\newblock \href{http://arxiv.org/abs/math/0701339}{arXiv:math/0701339}.

\bibitem[OS04]{hfk}
P.~Ozsv{\'a}th and Z.~Szab{\'o}.
\newblock Holomorphic disks and knot invariants.
\newblock {\em Adv. Math.}, 186(1):58--116, 2004.
\newblock \href{http://arxiv.org/abs/math/0209056}{arXiv:math/0209056}.

\bibitem[OS08]{oszlink}
P.~Ozsv{\'a}th and Z.~Szab{\'o}.
\newblock Holomorphic disks, link invariants and the multi-variable {A}lexander
  polynomial.
\newblock {\em Algebr. Geom. Topol.}, 8(2):615--692, 2008.

\bibitem[OS09]{oszskein}
Peter Ozsv\'ath and Zolt\'an Szab\'o.
\newblock A cube of resolutions for knot {F}loer homology.
\newblock {\em J. Topol.}, 2(4):865--910, 2009.
\newblock \href{http://arxiv.org/abs/0707.1165}{arXiv:0707.1165v1}.

\bibitem[Pet18]{dec}
I.~Petkova.
\newblock The decategorification of bordered {Heegaard} {F}loer homology.
\newblock {\em J. Symplectic Geom.}, 16(1):227--277, 2018.
\newblock \href{http://arxiv.org/abs/1212.4529}{arXiv:1212.4529}.

\bibitem[PV16a]{pv}
I.~Petkova and V.~V\'ertesi.
\newblock Combinatorial tangle {F}loer homology.
\newblock {\em Geom. Topol.}, 20(6):3219--3332, 2016.
\newblock \href{http://arxiv.org/abs/1410.2161}{arXiv:1410.2161}.

\bibitem[PV16b]{pv2}
I.~Petkova and V.~V\'ertesi.
\newblock An introduction to tangle {F}loer homology.
\newblock In {\em Proceedings of the {G}\"okova {G}eometry-{T}opology
  {C}onference 2015}, pages 168--186. G\"okova Geometry/Topology Conference
  (GGT), G\"okova, 2016.

\bibitem[Ras03]{jrth}
J.~A. Rasmussen.
\newblock {\em Floer homology and knot complements}.
\newblock ProQuest LLC, Ann Arbor, MI, 2003.
\newblock Thesis (Ph.D.)--Harvard University.

\bibitem[Rou08]{Rou2KM}
R.~Rouquier.
\newblock 2-{K}ac-{M}oody algebras.
\newblock 2008.
\newblock \href{http://arxiv.org/abs/0812.5023}{arXiv:0812.5023}.

\bibitem[RT91]{RT}
N.~Reshetikhin and V.~Turaev.
\newblock Invariants of 3-manifolds via link polynomials and quantum groups.
\newblock {\em Inv. Math.}, 103(3):547--597, 1991.

\bibitem[Sar15]{s1}
A.~Sartori.
\newblock The {A}lexander polynomial as quantum invariant of links.
\newblock {\em Ark. Mat.}, 53(1):177--202, 2015.
\newblock \href{http://arxiv.org/abs/1308.2047}{arXiv:1308.2047}.

\bibitem[Sar16]{Sartori2}
A.~Sartori.
\newblock Categorification of tensor powers of the vector representation of
  {$U_q(\mathit{gl}(1|1))$}.
\newblock {\em Selecta Math. (N.S.)}, 22(2):669--734, 2016.
\newblock \href{http://arxiv.org/abs/1305.6162}{arXiv:1305.6162}.

\bibitem[Tia14]{Tian2}
Y.~Tian.
\newblock A categorification of {$U_T(\mathit{sl}(1|1))$} and its tensor
  product representations.
\newblock {\em Geom. Topol.}, 18(3):1635--1717, 2014.

\bibitem[Tia16]{Tian1}
Y.~Tian.
\newblock Categorification of {C}lifford algebras and
  {${\bf{U}}_q(\mathit{sl}(1|1))$}.
\newblock {\em J. Symplectic Geom.}, 14(2):541--585, 2016.
\newblock \href{http://arxiv.org/abs/1210.5680}{arXiv:1210.5680}.

\bibitem[Vir06]{v1}
O.~Ya. Viro.
\newblock Quantum relatives of the {A}lexander polynomial.
\newblock {\em Algebra i Analiz}, 18(3):63--157, 2006.
\newblock \href{http://arxiv.org/abs/math/0204290}{arXiv:math/0204290}.

\bibitem[Web13]{Webster}
B.~Webster.
\newblock Knot invariants and higher representation theory.
\newblock 2013.
\newblock \href{http://arxiv.org/abs/1309.3796}{arXiv:1309.3796}.

\bibitem[Zar09]{bs}
Rumen Zarev.
\newblock Bordered {F}loer homology for sutured manifolds.
\newblock 2009.
\newblock \href{http://arxiv.org/abs/0908.1106}{arXiv:0908.1106}.

\bibitem[Zha09]{Zhang}
H.~Zhang.
\newblock The quantum general linear supergroup, canonical bases and
  {K}azhdan-{L}usztig polynomials.
\newblock {\em Science in China Series A: Mathematics}, 52(3):401--416, 2009.

\end{thebibliography}

\end{document}